\documentclass[oneside,12pt,a4paper]{article}
\usepackage{amsmath}
\usepackage{amsthm,verbatim,bbm,amsfonts, yfonts, amssymb,newclude,nicefrac,enumerate,bm,mathabx}
\usepackage{geometry}
  \usepackage{paralist}
  \usepackage{graphics}
\usepackage[colorlinks=true]{hyperref}

\hypersetup{urlcolor=blue, citecolor=red}

\usepackage[shortlabels,inline]{enumitem}

\usepackage[capitalise]{cleveref}

\usepackage[backend=biber, style=numeric,
maxbibnames=99]{biblatex} 
\newtheorem{theorem}{Theorem}[section]
\newtheorem{lemma}[theorem]{Lemma}

\newtheorem{corollary}[theorem]{Corollary}

\newtheorem{setting}[theorem]{Setting}

\theoremstyle{definition}

\usepackage{mleftright}

\usepackage{mathtools}

\DeclarePairedDelimiter{\abs}{\lvert}{\rvert}
\DeclarePairedDelimiter{\norm}{\lVert}{\rVert}

\newcommand{\E}{\mathbb{E}}

\newcommand{\R}{\mathbb{R}}
\newcommand{\N}{\mathbb{N}}

\newcommand{\cK}{\mathcal{K}}

\renewcommand{\d}{{\mathrm d}}

\newcommand{\normm}[1]{ \big\lVert #1 \big\rVert }
\newcommand{\normmm}[1]{ \Big\lVert #1 \Big\rVert }

\title{
  \vspace{-1cm}
  On existence and uniqueness properties
  for \\
  solutions of stochastic 
  fixed point equations \\
  with gradient-dependent nonlinearities}

\author{
Martin Hutzenthaler$^1$
and 
Katharina Pohl$^2$
\bigskip
\\
\small{$^1$ Faculty of Mathematics, University of Duisburg-Essen,}
\vspace{-0.1cm}\\
\small{Essen, Germany, e-mail: \texttt{martin.hutzenthaler}\textcircled{\texttt{a}}\texttt{uni-due.de}}
\smallskip
\\
\small{$^2$ Faculty of Mathematics, University of Duisburg-Essen,}
\vspace{-0.1cm}\\
\small{Essen, Germany, e-mail: \texttt{katharina.pohl}\textcircled{\texttt{a}}\texttt{uni-due.de}}
}

\begin{document}
\date{\today}
\maketitle


\begin{abstract}
\hskip -.2in
\noindent
The combination of the Itô formula
and the Bismut-Elworthy-Li formula
implies that suitable smooth solutions of semilinear Kolmogorov partial differential equations (PDEs) are also solutions to certain stochastic fixed point equations (SFPEs). 
In this paper we generalize
known results on
existence and uniqueness of solutions of SFPEs associated with 
PDEs with Lipschitz continuous,
gradient-independent nonlinearities
to the case of gradient-dependent 
nonlinearities.
The main challenge arises from the fact that
in the case of a non-differentiable terminal condition and a gradient-dependent
nonlinearity 
the Bismut-Elworthy-Li formula leads to a
singularity of the solution of the SFPE in the last time point.

\end{abstract}

\newpage
\tableofcontents

\newpage

\section{Introduction}

The well-known Feynman-Kac formula
forges a strong connection 
between classical solutions of
semilinear Kolmogorov partial
differential equations (PDEs)
and 
stochastic differential equations (SDEs).
In particular, 
if $X^x_t\colon [t,T]\times\R^d\to\R^d$,
$t\in [0,T]$, $x\in\R^d$,
solve the SDE
\begin{equation}
  d X^x_{t,s} = \mu(s,X^x_{t,s})\,\d s
  +\sigma(s,X^x_{t,s}) \,\d W_s,
  \qquad  s\in [t,T],
  \, X^x_{t,t} = x, 
\end{equation}
and $u\in C^{1,2}([0,T]\times \R^d, \R^d)$
satisfies for all 
$s\in [0,T]$, $y\in\R^d$ that
$u(T,y)=g(y)$ and 
\begin{equation}
\label{eq:pde}
  \big(\tfrac{\partial u}{\partial t}
  + (\nabla_x u)\mu
  +\tfrac{1}{2}\langle \sigma,
  (\operatorname{Hess}_x u)\sigma\rangle
  \big)(s,y) +
  f(s,y,u(s,y), (\nabla_x u)(s,y)) =0,
\end{equation}
then 
the Itô formula
demonstrates for all $t\in [0,T]$,
$s\in [t,T]$,
$x\in\R^d$ that a.s.\!
\begin{equation}
   u(s,X^x_{t,s})= g(X^x_{t,T})+\int_s^T f(r,X^x_{t,r},u(r,X^x_{t,r}), (\nabla_x u)(r,X^x_{t,r})) \,\d r
   -\int_s^T ((\nabla_x u)\sigma)(r,X^x_{t,r}) \,\d W_r.
\end{equation}
This leads to the following stochastic
fixed point equation (SFPE)
\begin{equation}
\label{eq:idea}
  u(t,x)= \E\Big[g(X^x_{t,T})+\int_t^T f(r,X^x_{t,r},u(r,X^x_r), (\nabla_x u)(r,X^x_{t,r})) \,\d r\Big],
  \qquad t\in [0,T], \, 
  X^x_{t,t}=x.
\end{equation} 
PDEs of the form  \eqref{eq:pde} often
do not admit a classical solution
(cf.\! \cite{HHJ2015}).
For this reason we concentrate on
finding solutions of the related
SFPE \eqref{eq:idea} instead.
If the nonlinearity $f$
of the SFPE in \eqref{eq:idea}
is gradient-independent,
then the SFPE in
\eqref{eq:idea} is closed 
and can be solved with 
fixed point methods. 
In this case and under the additional 
assumption of Lipschitz continuity 
of the nonlinearity, 
the results in 
\cite{BGHJ2019}
show the existence 
of a unique solution of the SFPE
in \eqref{eq:idea}.

In this article we 
extend the results in \cite{BGHJ2019}
to gradient-dependent, 
Lipschitz continuous nonlinearities.
In this case,
the SFPE in 
\eqref{eq:idea} is not closed.
Our attempt to solve the SFPE in 
\eqref{eq:idea} follows the idea
that -under suitable assumptions- 
applying the Bismut-Elworthy-Li formula
to \eqref{eq:idea} 
shows for all $t\in[0,T)$ that
\begin{equation}
\label{eq:idea_BEL}
\begin{split}
  &\big( \nabla_x u \big)(t,x)
  = \E\Big[g(X^x_{t,T})\tfrac{1}{T-t}\int_t^T
  (\sigma(r,X^x_{t,r}))^{-1}
  \big(\tfrac{\partial}{\partial x}X^x_{t,r}\big)
  \,\d W_r\Big]\\
  &\qquad +\E\Big[\int_t^T f(r,X^x_{t,r},u(r,X^x_{t,r}), (\nabla_x u)(r,X^x_{t,r}))
  \tfrac{1}{s-t} \int_t^s (\sigma(r,X^x_{t,r}))^{-1}
  \big(\tfrac{\partial}{\partial x}X^x_{t,r}\big)
  \,\d W_r \,\d s\Big]
\end{split}
\end{equation}
(cf., e.g., \cite[Theorem 2.1]{ElworthyLi}).
The goal of this paper is to prove
the existence of a unique solution
of the closed system of SFPEs \eqref{eq:idea}
and \eqref{eq:idea_BEL}
for a Lipschitz continuous, 
gradient-independent nonlinearity.
The main
challenge of this approach arises
from the fact that a non-differentiable terminal condition and 
the Bismut-Elworthy-Li
formula lead to a singularity
of \eqref{eq:idea_BEL}
in the last time point. 
To still follow the idea of \cite{BGHJ2019}
to prove the existence and uniqueness 
of the solution of the SFPE 
by Banach's fixed point theorem
it is necessary to define a new, 
suitable norm 
of the Banach space on which
Banach's fixed point theorem is applied
(cf.\! Theorem~\ref{thm:ex_cont_solution}).

Another challenge 
comes from allowing 
non-Lipschitz continuous drift 
coefficients which only
satisfy a one-sided
Lipschitz condition. 
This generalized assumption
allows 
our results to solve 
an even bigger class of
SFPEs and their 
associated PDEs. 
To illustrate the findings 
of this article, we present 
Theorem~\ref{thm:main_thm}
below 
which shows 
existence and uniqueness
of solutions of SFPEs associated with
suitable Kolmogorov PDEs
with Lipschitz continuous and
gradient-dependent nonlinearities
within the class of at most 
polynomially growing continuous
functions
and is 
a special case of our 
main result, 
Theorem~\ref{thm:ex_cont_sol_Z}.

\begin{theorem}
\label{thm:main_thm}
  Let $d \in \N$,
  $ c, T \in (0, \infty)$,
  let $\norm{\cdot}\colon\R^d\to[0,\infty)$, 
  $\lvert\lvert\lvert\cdot\rvert\rvert\rvert\colon\R^{d+1}\to[0,\infty)$,
  $\norm{\cdot}_F\colon\R^{d\times d}\to[0,\infty)$ be norms,
  let $(\Omega, \mathcal{F}, \mathbb{P},
  (\mathbb{F}_s)_{s \in [0,T]})$
  be a filtered probability space,
  let $\mu \in  C^{2}(\R^d, \R^d)$, 
  $\sigma \in C^{2} (\R^d, \R^{d \times d})$
  satisfy for all  
  $x, y, v \in \R^d$
  that
  $v^* \sigma(x) (\sigma(x))^* v 
  \geq \frac{1}{c}\norm{v}^{2}$
  and
  \begin{equation}
  \max\Big\{(x-y)^*(\mu(x)-\mu(y)), 
  \tfrac{1}{2}\norm{\sigma(x)-\sigma(y)}_F^2\Big\}
  \leq c\norm{x-y}^2,
  \end{equation}
  for every $t\in[0,T]$, 
  $x \in \R^d$ let
  $X_t^x = (X^{x}_{t,s})_{s \in [t,T]} \colon [t,T] \times \Omega \to \R^d$ be an 
  $(\mathbb{F}_s)_{s \in [t,T]}$-adapted 
  stochastic process with continuous 
  sample paths satisfying that 
  for all $s \in [t,T]$ 
  it holds a.s.\! that
  \begin{equation}\label{eq:SDE}
  X^x_{t,s} = x + \int_t^s \mu( X^x_{t,r}) \,\d r 
  + \int_t^s \sigma(X^x_{t,r}) \,\d W_r,
  \end{equation}
  assume for all $t\in [0,T]$,
  $\omega \in \Omega$
  that
  $\left([t,T] \times \R^d \ni (s,x) 
  \mapsto X^x_{t,s}(\omega) \in \R^d \right) \in C^{0,1}([t,T] \times \R^d, \R^d)$,
  for every $t\in[0,T]$, 
  $x \in \R^d$ let
  $Z^x_t = (Z^x_{t,s})_{s \in (t,T]} 
  \colon (t,T] \times \Omega \to \R^{d+1}$ 
  be an $(\mathbb{F}_s)_{s \in (t,T]}$-adapted stochastic process 
  with continuous sample paths
  satisfying that 
  for all $s \in (t,T]$ 
  it holds a.s.\! that
  \begin{equation}
    Z^x_{t,s} = 
    \begin{pmatrix}
      1\\
      \frac{1}{s-t} \int_t^s
  (\sigma(X^x_{t,r}))^{-1} \; \Big(\frac{\partial}{\partial x} X^x_{t,r}\Big) \,\d W_r
    \end{pmatrix},     
  \end{equation}   
   let $f \in C([0,T) \times \R^d \times \R^{d+1}, \R)$,
  $g \in C(\R^d, \R)$ satisfy for all
  $t \in [0,T)$, $x \in \R^d$, $v, w \in \R^{d+1}$
  that
  $\lvert f(t,x,v) - f(t,x,w) \rvert 
  \leq c\lvert\lvert\lvert  v-w\rvert\rvert \rvert $
  and 
  $\max\{\abs{g(x)},\abs{f(t,x,0)}\}
  \leq c(\norm{x}^c+1)$.
  Then there exists a unique 
  $v \in C([0,T) \times \R^d, \R^{d+1})$
  such that
  $(v(t,x)\sqrt{T-t})_{t\in [0,T), x\in\R^d}$
  grows at most polynomially and
  for all
  $t \in [0,T)$, $x \in \R^d$
  it holds that
  \begin{equation}
  \label{eq:main_sfpe}
    v(t,x) = \E \left[ g(X^x_{t,T})Z^x_{t,T} + \int_t^T f(r, X^x_{t,r}, v(r, X^x_{t,r}))Z^x_{t,r} \,\d r \right].
  \end{equation}
  \end{theorem}
  

Theorem~\ref{thm:main_thm}
is an immediate consequence of our 
main result,
Theorem~\ref{thm:ex_cont_sol_Z}\footnote{applied with $O\curvearrowleft \R^d$,
$V\curvearrowleft ([0,T]\times\R^d\ni(t,x)\mapsto 1+\norm{x}^{c+1}\in (0,\infty))$},
in Section~$\ref{subse:Z_exist_uniq}$
below.  
Let us comment on some of the
mathematical objects appearing
in Theorem~\ref{thm:main_thm}.
The positive, real number $c\in (0,\infty)$
in Theorem~\ref{thm:main_thm}
is used to formulate an
ellipticity condition, 
a growth condition, and 
Lipschitz type conditions
on the functions
$\mu$, $\sigma$, and $f$.
The positive, real number $T\in(0,\infty)$ in
Theorem~\ref{thm:main_thm}
determines the time horizon of
the SDE in \eqref{eq:SDE}.
The functions $\mu\in C^2(\R^d,\R)$
and $\sigma\in C^2(\R^d,\R^{d\times d})$
in
Theorem~\ref{thm:main_thm}
specify the drift and the
diffusion coefficient
of the SDE in
\eqref{eq:SDE}, respectively.
The stochastic process $X^x_t\colon
[t,T]\times \Omega \to \R^d$ 
in
Theorem~\ref{thm:main_thm}
determines 
the solution of the SDE in 
\eqref{eq:SDE}.
The second to $(d+1)$-th component of the
stochastic process $Z^x_t\colon
(t,T]\times\Omega\to \R^{d+1}$
in
Theorem~\ref{thm:main_thm} specify
the factor arising from the 
Bismut-Elworthy-Li formula.
The function $f\in C([0,T)\times\R^d
\times\R^{d+1},\R)$ in
Theorem~\ref{thm:main_thm} 
can be interpreted as
the nonlinearity of a 
semilinear 
PDE associated with
the SDE in \eqref{eq:SDE}
while the function $g\in C(\R^d,\R)$
specifies the terminal condition 
of this semilinear PDE.


The remainder of this article is
organized as follows.
In Section~\ref{sec:general_sfpes}
we study
SFPEs in an abstract setting
for general stochastic processes
$X$ and 
$Z$ in \eqref{eq:main_sfpe}.
The main result of 
Section~\ref{subsec:exist_uniq}
is Theorem~\ref{thm:ex_cont_solution}
which proves the existence and uniqueness
of solutions to SFPEs in an abstract setting by using Banach's fixed point theorem.
The well-definedness of the mapping
to which Banach's fixed point theorem
is applied in the proof of 
Theorem~\ref{thm:ex_cont_solution}
is demonstrated in
Sections~\ref{subsec:integrability}
and \ref{subsec:continuity}.
In Section~\ref{subsec:contractivity}
we establish the contractivity property
needed for using
Banach's fixed point theorem in the 
proof of Theorem~\ref{thm:ex_cont_solution}.
In Section~\ref{sec:sfpes_Z} we
investigate SFPEs of the form \eqref{eq:main_sfpe} 
where $X$ is an SDE solution and
$Z$ is a specific
stochastic process
arising from the Bismut-Elworthy-Li
formula. 
In Theorem~\ref{thm:ex_cont_sol_Z},
the main result of this article,
we apply the abstract result 
Theorem~\ref{thm:ex_cont_solution} 
to these
stochastic processes. 
In Sections~\ref{subsec:Z_bounds}
and \ref{subsec:Z_convergence} we
establish several auxiliary results
to ensure the assumptions of 
Theorem~\ref{thm:ex_cont_solution}.

\section{Abstract stochastic fixed point 
equations (SFPEs)}
\label{sec:general_sfpes}

In this section we study existence,
uniqueness, 
and general properties of abstract SFPEs.
The main result of this section
is
Theorem~\ref{thm:ex_cont_solution}
in Section~\ref{subsec:exist_uniq}
below. 
Theorem~\ref{thm:ex_cont_solution}
is an application of Banach's fixed point theorem to the function
in \eqref{eq:ex_cont_sol2}.
Lemma~\ref{lem:continuity2} in Section~\ref{subsec:regularity}
below ensures well-definedness of this function.
Lemma~\ref{lem:contraction} in Section~\ref{subsec:contractivity} below establishes
the contractivity property needed to
apply Banach's fixed point theorem.
Lemma~\ref{lem:integral} 
makes it possible
to overcome the challenges of the 
singularity of the stochastic process
arising from the Bismut-Elworthy-Li formula
in Lemma~\ref{lem:continuity1},
Lemma~\ref{lem:continuity2},
and Lemma~\ref{lem:contraction}.
Throughout this section we 
frequently use the following setting.

\begin{setting}\label{setting:sfpe}
  Let $d, m \in \N$,
  $c, T \in (0, \infty)$,
  let $\lvert \lvert \cdot \rvert \rvert\colon \R^d \to [0,\infty)$,
  $\lvert \lvert \lvert \cdot \rvert \rvert \rvert\colon \R^m \to [0,\infty)$
  be norms,
  let $O \subseteq \R^d$
  be a non-empty open set,
  for every $r \in (0, \infty)$
  let $K_r\subseteq [0,T)$,
  $O_r \subseteq O$
  satisfy
  $K_r=[0,\max\{T-\frac{1}{r},0\}]$ and
  $O_r = \{  x \in O\colon \lvert \lvert x \rvert \rvert \leq r \text{ and } \{ y \in \R^d\colon \lvert \lvert y-x \rvert \rvert 
  < \frac{1}{r} \} \subseteq O \}$,
  let $(\Omega, \mathcal{F}, \mathbb{P})$
  be a probability space,
  for every $t \in [0,T]$, $x \in O$
  let $X^x_{t} = (X^x_{t,s})_{s \in [t,T]} \colon [t,T] \times \Omega \to O$ 
  and $Z^x_{t} = (Z^x_{t,s})_{s \in (t,T]} \colon (t,T] \times \Omega \to \R^m$ 
  be measurable,
  and let $V \in C([0,T]\times O,(0, \infty))$
  satisfy for all
  $t \in [0,T)$, $s \in (t, T]$, 
  $x \in O$
  that
  $\E [ V(s, X^x_{t,s}) \lvert \lvert \lvert 
  Z^x_{t,s} \rvert \rvert \rvert] 
  \leq \frac{c}{\sqrt{s-t}} V(t,x)$.
\end{setting}

\subsection{Integrability of certain stochastic processes}
\label{subsec:integrability}

The following lemma shows the
calculation of two specific
integrals which will be used
to prove upper bounds and ensure
the well-definedness of the SFPE.

\begin{lemma}\label{lem:integral}
  Let $a\in\R$, $b\in(a,\infty)$.
  Then 
  \begin{enumerate}[label=(\roman*)]
    \item\label{it:int1}
    for all $\lambda\in [0,\infty)$ 
    it holds that
    \begin{equation}
      \int_a^b \frac{1}{\sqrt{(b-x)(x-a)}}
      e^{-\lambda x} \,\d x
      = e^{-\frac{\lambda (b+a)}{2}}
      \int_0^\pi e^{\frac{\lambda (b-a)}{2}\cos(\theta)}\,\d \theta
    \end{equation}
    and
    \item\label{it:int2}
    for all $\lambda\in (0,\infty)$
    it holds that
    \begin{equation}
    \int_a^{b} \frac{1}{\sqrt{(b-x)(x-a)}}\, 
    e^{-\lambda x} \,\d x
    \leq \sqrt{\frac{\pi^3}{4\lambda(b-a)}}
    \,e^{-\lambda a}.
  \end{equation}
  \end{enumerate}
\end{lemma}
  
\begin{proof}[Proof of Lemma~\ref{lem:integral}] 
  First note that
  the substitution 
  $(0,\pi)\ni\theta\mapsto 
  x(\theta)=\frac{c}{2}(1-\cos(\theta))
  \in(0,c)$
  and the fact that for all
  $x\in\R$ it holds that
  $(\sin(x))^2+(\cos(x))^2=1$
  ensure that for all
  $\lambda \in[0,\infty)$,
  $c\in (0,\infty)$
  it holds that
  \begin{equation}
  \begin{split}
    &e^{\frac{\lambda c}{2}} \int_0^{c} \frac{1}{\sqrt{x(c-x)}}\, 
    e^{-\lambda x} \,\d x
    = \int_0^{c} \frac{1}{\sqrt{x(c-x)}}\, 
    e^{\lambda (\frac{c}{2}-x)} \,\d x\\
    &=\int_0^\pi \frac{1}{\sqrt{\frac{c}{2}(1-\cos(\theta))(c-\frac{c}{2}(1-\cos(\theta)))}}
    e^{\frac{\lambda c}{2}\cos(\theta)}
    \frac{c}{2}\sin(\theta)\,\d \theta\\
    &=\int_0^\pi \frac{1}{\sqrt{(1-\cos(\theta))(1+\cos(\theta))}}
    e^{\frac{\lambda c}{2}\cos(\theta)}
    \sin(\theta)\,\d \theta\\
    &=\int_0^\pi \frac{1}{\sqrt{1-(\cos(\theta))^2}}
    e^{\frac{\lambda c}{2}\cos(\theta)}
    \sin(\theta)\,\d \theta
     =\int_0^\pi
    e^{\frac{\lambda c}{2}\cos(\theta)}
    \,\d \theta.
  \end{split}
  \end{equation} 
  Hence, we obtain 
  for all $\lambda\in [0,\infty)$ 
  that
  \begin{equation}
  \begin{split}
    &\int_a^b \frac{1}{\sqrt{(b-x)(x-a)}}
      \,e^{-\lambda x} \,\d x
    = e^{-\lambda a}\int_0^{b-a} \frac{1}{\sqrt{(b-a-y)y}}
      \,e^{-\lambda y} \,\d y\\
    &= e^{-\frac{\lambda(b-a)}{2}-\lambda a} \int_0^\pi
    e^{\frac{\lambda (b-a)}{2}\cos(\theta)}
    \,\d \theta
    = e^{-\frac{\lambda(b+a)}{2}} \int_0^\pi
    e^{\frac{\lambda (b-a)}{2}\cos(\theta)}
    \,\d \theta.
  \end{split}
  \end{equation}
  This establishes item~\ref{it:int1}.
  Next,
  the fact that
  $\int_0^\infty e^{-z^2}\d z
  =\frac{\sqrt{\pi}}{2}$    
  and the fact that for all 
  $y\in[0,\frac{\pi}{2}]$ 
  it holds that
  $\sin(y)\geq \frac{2y}{\pi}$    
  yield for all 
  $c\in(0,\infty)$
  that
  \begin{equation}\label{eq:Bessel_bound}
  \begin{split}
    &\int_0^{\pi}
    e^{c\cos(\theta)}\,\d\theta
    =2\int_0^{\frac{\pi}{2}}
    e^{c\cos(2y)}\,\d y
    =2\int_0^{\frac{\pi}{2}}
    e^{c(1-2(\sin(y))^2)}\,\d y\\
    &=2 e^c\int_0^{\frac{\pi}{2}}
    e^{-2c(\sin(y))^2}\,\d y
    \leq 2 e^c\int_0^{\frac{\pi}{2}}
    e^{-\frac{8cy^2}{\pi^2}}\,\d y
    \leq \frac{2\pi}{\sqrt{8c}}e^c\int_0^{\infty}
    e^{-z^2}\,\d z
    = \sqrt{\frac{\pi^3}{8c}}e^c.
  \end{split}
  \end{equation}
  Item~\ref{it:int1} therefore
  shows that for all
  $\lambda\in (0,\infty)$
  it holds that
  \begin{equation}
  \begin{split}
    &\int_a^{b} \frac{1}{\sqrt{(b-x)(x-a)}}\, 
    e^{-\lambda x} \,\d x
    =e^{-\frac{\lambda(b+a)}{2}} \int_0^\pi
    e^{\frac{\lambda (b-a)}{2}\cos(\theta)}
    \,\d \theta\\
    &\leq e^{-\frac{\lambda (b+a)}{2}}
    \sqrt{\frac{2\pi^3}{8\lambda(b-a)}}
    e^{\frac{\lambda(b-a)}{2}}
    =\sqrt{\frac{\pi^3}{4\lambda(b-a)}}
    \, e^{-\lambda a}.
  \end{split}
  \end{equation}
  The proof of Lemma~\ref{lem:integral}
  is thus complete.  
\end{proof}

The next lemma uses 
Lemma~\ref{lem:integral}
to establish
integrability properties for a
certain class of stochastic processes.
Lemma~\ref{lem:integrability}
is a generalization of the results in \cite[Lemma 2.1]{BGHJ2019}.

\begin{lemma}\label{lem:integrability}
Let $d, m \in \N$,
$c, T \in (0, \infty)$, 
let $\lvert \lvert \lvert \cdot \rvert \rvert \rvert\colon \mathbb{R}^m \to [0, \infty)$ be a norm,
let $O \subseteq \R^d$
be a non-empty open set,
let $(\Omega, \mathcal{F}, \mathbb{P})$
be a probability space,
for every $t \in [0,T]$, $x \in O$ 
let $X^x_t = (X^x_{t,s})_{s \in [t,T]} \colon [t,T] \times \Omega \to O$
and $Z^x_{t} = (Z^x_{t,s})_{s \in (t,T]} \colon (t,T] \times \Omega \to \R^m$
be measurable,
let $g\colon O \to \R$, 
$h\colon [0,T] \times O \to \R$, and
$V\colon[0,T] \times O \to (0, \infty)$ 
be measurable,
and assume for all
$t \in [0,T)$, $s \in (t,T]$, $x \in O$
that 
$\E [ V(s, X^x_{t,s}) \lvert \lvert \lvert Z^x_{t,s} \rvert \rvert \rvert ]
\leq \frac{c}{\sqrt{s-t}} V(t,x)$
and
$\sup _{t \in [0,T)} \sup _{x \in O} \left[ \frac{\lvert g(x) \rvert}{V(T,x)} + \frac{\lvert h(t,x) \rvert}{V(t,x)}\sqrt{T-t} \right] < \infty$.
Then it holds for all
$t \in [0,T)$, $x \in O$ that
\begin{equation}\label{eq:integrability}
\E \left[ \lvert g(X^x_{t,T}) \rvert \; \lvert \lvert \lvert Z^x_{t,T} \rvert \rvert \rvert
+ \int_t^T \lvert h(r, X^x_{t,r}) \rvert \; \lvert \lvert \lvert Z^x_{t,r} \rvert \rvert \rvert \,\d r \right]
    <\infty.
\end{equation} 
\end{lemma}

\begin{proof}[Proof of Lemma \ref{lem:integrability}]
  Throughout this proof let 
  $\alpha \in [0, \infty)$
  satisfy for all
  $t \in [0,T)$, $x \in O$
  that
  \begin{equation}
  \label{eq:V_constant}
     \lvert g(x) \rvert \leq \alpha V(T,x)
     \qquad
     \text{ and }
     \qquad
     \lvert h(t,x)\rvert\sqrt{T-t}  
     \leq \alpha V(t,x).
  \end{equation}
  Observe that the assumption that 
  $g\colon O \to \R$ and
  $h\colon [0,T] \times O \to \R$ 
  are measurable
  and the fact that 
  for all $t \in [0,T]$, $x \in O$
  it holds that 
  $X^x_t\colon [t,T] \times \Omega \to O$ 
  and
  $Z^x_t \colon (t,T] \times \Omega \to \R^m$ 
  are measurable
  imply that
  for all $t \in [0,T)$, $x \in O$
  it holds that
  $\Omega \ni \omega \mapsto g(X^x_{t,T}(\omega))Z^x_{t,T}(\omega)\in\R^m$ 
  and
  $(t, T] \times \Omega \ni (s, \omega) \mapsto h(s, X^x_{t,s}(\omega))Z^x_{t,s}(\omega)\in\R^m$
  are measurable.
  Furthermore, note that
  Fubini's theorem,
  \eqref{eq:V_constant}, and 
  the hypothesis that for all
  $t \in [0,T)$, $s \in (t,T]$, $x \in O$
  it holds that 
  $\E [ V(s, X^x_{t,s}) \lvert \lvert \lvert 
  Z^x_{t,s} \rvert \rvert \rvert ]
   \leq \frac{c}{\sqrt{s-t}} V(t,x)$
  demonstrate that 
  for all
  $t \in [0,T)$, $x \in O$
  it holds that
  \begin{equation}
    \begin{split}
    &\E \left[ \lvert g(X^x_{t,T}) \rvert \; \lvert \lvert \lvert Z^x_{t,T} \rvert \rvert \rvert
    + \int_t^T \lvert h(r, X^x_{t,r}) \rvert \; \lvert \lvert \lvert Z^x_{t,r} \rvert \rvert \rvert \,\d r \right] \\
    &= \E \left[ \lvert g(X^x_{t,T}) \rvert \; \lvert \lvert \lvert Z^x_{t,T} \rvert \rvert \rvert \right]
    + \int_t^T \E \left[ \lvert h(r, X^x_{t,r}) \rvert \; \lvert \lvert \lvert Z^x_{t,r} \rvert \rvert \rvert \right] \,\d r  \\
    &\leq \E \left[ \alpha V(T, X^x_{t,T}) \lvert\lvert \lvert Z^x_{t,T} \rvert \rvert \rvert \right]
    +  \int_t^T\frac{1}{\sqrt{T-r}}
    \E \left[ \alpha V(r, X^x_{t,r}) \lvert \lvert \lvert Z^x_{t,r} \rvert \rvert \rvert \right] \,\d r \\
    &\leq  \frac{\alpha c}{\sqrt{T-t}} V(t,x) 
    +\int_t^T  \frac{\alpha c}{\sqrt{(T-r)(r-t)}} V(t,x)\,\d r.
  \end{split}
  \end{equation}
  Item~\ref{it:int1} of 
  Lemma~\ref{lem:integral}
  (applied for every $t\in [0,T)$
  with $a\curvearrowleft t$, 
  $b\curvearrowleft T$,
  $\lambda\curvearrowleft 0$
  in the notation of
  Lemma~\ref{lem:integral})
  therefore shows that
  for all $t\in [0,T)$,
  $x\in O$
  it holds that
  \begin{equation}
  \begin{split}
    \E \left[ \lvert g(X^x_{t,T}) \rvert \; \lvert \lvert \lvert Z^x_{t,T} \rvert \rvert \rvert
    + \int_t^T \lvert h(r, X^x_{t,r}) \rvert \; \lvert \lvert \lvert Z^x_{t,r} \rvert \rvert \rvert \,\d r \right] 
    \leq \alpha c\bigg( 
    \frac{1}{\sqrt{T-t}}  
    +\pi \bigg)V(t,x)
    <\infty.
  \end{split}
  \end{equation}
  The proof of Lemma~\ref{lem:integrability}
  is thus complete.
\end{proof}

\subsection{Continuity of SFPEs with respect to coefficient functions}
\label{subsec:continuity}

In this section we establish some
general convergence and approximation 
results for SFPE solutions.
The following lemma
generalizes \cite[Lemma 2.2]{BGHJ2019}
and demonstrates
several properties of
approximating functions.

\begin{lemma}\label{lem:approximation_result1}
  Assume Setting~\ref{setting:sfpe},
  let $g_n \in C(O, \R)$, $n \in \N_0$,
  and $h_n \in C([0,T) \times O, \R)$, 
  $n \in \N_0$,
  satisfy for all 
  $n \in \N$ that
  $\inf_{r \in (0, \infty)} [ \sup_{t \in [0,T)\setminus K_r} \sup_{x \in O \setminus O_r} (  \frac{\lvert g_n(x) \rvert}{V(T,x)} 
  + \frac{\lvert h_n(t,x) \rvert}{V(t,x)}
  \sqrt{T-t}) ] 
  = 0$,
  and assume that
  \begin{equation}  \label{eq:approximation_result1_ass}
    \limsup \limits_{n \to \infty} \left[ 
    \sup \limits_{t \in [0,T)} 
    \sup \limits_{x \in O} \left( 
    \frac{\lvert g_n(x) - g_0(x) \rvert}{V(T,x)} 
    + \frac{\lvert h_n(t,x) - h_0(t,x) \rvert}{V(t,x)}\sqrt{T-t} \right) \right] 
    = 0.
\end{equation}
  Then
  \begin{enumerate}[label=(\roman*)]
  \item\label{it:app_res1}
  it holds for every
  $n \in \N_0$ that
  \begin{equation}
     \sup \limits_{t \in [0,T)} 
     \sup \limits_{x \in O}
     \left[ \frac{\lvert g_n(x) \rvert}{V(T,x)}
     + \frac{\lvert h_n(t,x) \rvert}{V(t,x)}
     \sqrt{T-t} \right]
     < \infty,
  \end{equation}
  \item\label{it:app_res2}
  it holds for every
 $n \in \N_0$
  that there exists a unique
  $v_n \colon [0,T) \times O \to \R^m$ 
  which satisfies for all
  $t \in [0,T)$, $x \in O$
  that
  \begin{equation}\label{eq:ar1_2}
     v_n(t,x) = \E\left[g_n(X^x_{t,T})Z^x_{t,T} 
     + \int_t^T h_n(r, X^x_{t,r}) 
     Z^x_{t,r} \,\d r \right],
  \end{equation}
  \item\label{it:app_res3}
  it holds that
  \begin{equation}
    \limsup \limits_{n \to \infty} \left[ 
    \sup \limits_{t \in [0,T)} \sup \limits_{x \in O} \left( \frac{\lvert \lvert \lvert v_n(t,x) 
    - v_0(t,x) \rvert \rvert \rvert}{V(t,x)}  \, \sqrt{T-t} \right) \right] 
    = 0,
  \end{equation}
  and
  \item\label{it:app_res4}
  it holds for every
  compact set 
  $\mathcal{K}\subseteq [0,T)\times O$
  that
  \begin{equation}
    \limsup \limits_{n \to \infty} \left[ 
    \sup \limits_{(t,x) \in \mathcal{K}} \, 
    \left( \lvert \lvert \lvert v_n(t,x) - v_0(t,x) \rvert \rvert \rvert  \, \sqrt{T-t} \right)  \right] 
    = 0.
  \end{equation}
  \end{enumerate}
\end{lemma}

\begin{proof}[Proof of Lemma~\ref{lem:approximation_result1}]
  First note that for all
  $r \in (0, \infty)$ it holds that
  $K_r$ and $O_r$ are compact sets.
  Combining this with the fact that
  $O \ni x \mapsto \frac{g_n(x)}{V(T,x)}\in\R$
  and
  $[0,T) \times O \ni (t,x) 
  \mapsto \frac{h_n(t,x)}{V(t,x)}\sqrt{T-t} 
  \in \R$ 
  are continuous 
  ensures that for all
  $n \in \N_0$, $r \in (0, \infty)$
  it holds that 
  \begin{equation}
    \sup_{t\in K_r}\sup_{x\in O_r} \left[ 
    \frac{\lvert g_n(x) \rvert}{V(T,x)} +
    \frac{\lvert h_n (t,x) \rvert}{V(t,x)}
    \sqrt{T-t}\right]
    < \infty.
  \end{equation}
  The hypothesis that for all
  $n \in \N$ it holds that
  $\inf_{r \in (0, \infty)} [ 
  \sup_{t\in [0,T)\setminus K_r}
  \sup_{x \in O \setminus O_r}
   ( \frac{\lvert g_n(x) \rvert}{V(T,x)} 
   + \frac{\lvert h_n(t,x) \rvert}{V(t,x)}
   \sqrt{T-t})] 
   = 0$
  hence implies that 
  for all $n \in \N$ it holds that
  \begin{equation}
  \sup \limits_{t \in [0,T)} \sup \limits_{x \in O} \left[ \frac{\lvert g_n(x) \rvert}{V(T,x)} 
  + \frac{\lvert h_n(t,x) \rvert}{V(t,x)}
  \sqrt{T-t} \right] 
  < \infty.
  \end{equation}
  This and \eqref{eq:approximation_result1_ass}
  establish item~\ref{it:app_res1}.
  Moreover, note that combining 
  item~\ref{it:app_res1}
  with Lemma \ref{lem:integrability}
  proves item~\ref{it:app_res2}.
  Next observe that the hypothesis that
  for all
  $t \in [0,T)$, $s \in (t,T]$, $x \in O$
  it holds that 
  $\E [ V(s, X^x_{t,s}) \lvert \lvert \lvert Z^x_{t,s} \rvert \rvert \rvert ]
  \leq \frac{c}{\sqrt{s-t}} V(t,x)$
  implies that
  for all
  $n \in \N$, $t \in [0,T)$, $x \in O$
  it holds that
  \begin{equation}
  \begin{split}
&\frac{\E \left[ \lvert \lvert \lvert g_n(X^x_{t,T})Z^x_{t,T} - g_0(X^x_{t,T})Z^x_{t,T} \rvert \rvert \rvert \right]}{V(t,x)} 
= \E \left[ \frac{\lvert g_n(X^x_{t,T}) - g_0(X^x_{t,T}) \rvert }{V(T, X^x_{t,T})} \cdot \frac{V(T, X^x_{t,T})  \lvert \lvert \lvert Z^x_{t,T} \rvert \rvert \rvert}{V(t,x)} \right] \\
&\leq \left[ \sup \limits_{y \in O} \left( \frac{\lvert g_n(y) - g_0(y) \rvert}{V(T,y)} \right) \right] \frac{\E \left[ V(T,X^x_{t,T}) \lvert \lvert \lvert Z^x_{t,T} \rvert \rvert \rvert \right]}{V(t,x)} 
\leq \sup \limits_{y \in O} \left( \frac{\lvert g_n(y) - g_0(y) \rvert}{V(T,y)} \right)\frac{c}{\sqrt{T-t}}.
\end{split}
  \end{equation}
  This and 
  \eqref{eq:approximation_result1_ass}
  show that
  \begin{equation}\label{eq:ar1_inequality_g}
  \begin{split}
    \limsup \limits_{n \to \infty} \left[ 
    \sup \limits_{t \in [0,T)} \sup \limits_{x \in O} \left( \frac{\lvert \lvert \lvert \E \left[ g_n(X^x_{t,T}) Z^x_{t,T} \right] 
    - \E \left[ g_0(X^x_{t,T}) Z^x_{t,T} \right] \rvert \rvert \rvert}{V(t,x)} \, \sqrt{T-t} \right) \right]
    = 0.
  \end{split}
  \end{equation}
  In addition, observe that
  the hypothesis that
  for all
  $t \in [0,T)$, $s \in (t,T]$, $x \in O$
  it holds that 
  $\E[ V(s, X^x_{t,s}) \lvert \lvert \lvert Z^x_{t,s} \rvert \rvert \rvert ] 
  \leq \frac{c}{\sqrt{s-t}} V(t,x)$
  ensures that
  for all
  $n \in \N$, $t \in [0,T)$, $x \in O$
  it holds that
  \begin{equation}
  \begin{split}
    &\frac{\E \left[ \int_t^T \lvert \lvert \lvert h_n(r,X^x_{t,r})Z^x_{t,r} 
    - h_0(r,X^x_{t,r})Z^x_{t,r} \rvert \rvert \rvert \,\d r \right]}{V(t,x)} \\
  &= \int_t^T \E \left[ \frac{\lvert h_n(r, X^x_{r,t}) - h_0(r, X^x_{t,r}) \rvert}{V(r, X^x_{t,r})} \cdot
  \frac{\sqrt{T-r}}{\sqrt{T-r}}\cdot 
  \frac{V(r, X^x_{t,r}) \, \lvert \lvert \lvert Z^x_{t,r} \rvert \rvert \rvert}{V(t,x)} \right] \,\d r \\
  &\leq \int_t^T \left[ \sup \limits_{s \in [0,T)} \sup \limits_{y \in O} \left( \frac{\abs{ h_n(s,y) - h_0(s,y)}}{V(s,y)}
  \sqrt{T-s} \right) \right] 
  \frac{\E \left[ V(r, X^x_{t,r}) \; \lvert \lvert \lvert Z^x_{t,r} \rvert \rvert \rvert  \right]}{\sqrt{T-r}\, V(t,x)} \,\d r\\
  &\leq  \left[ \sup \limits_{s \in [0,T)} \sup \limits_{y \in O} \left( \frac{\lvert h_n(s,y) - h_0(s,y)\rvert}{V(s,y)}\sqrt{T-s} \right) \right] 
  \int_t^T \frac{c}{\sqrt{(T-r)(r-t)}}  \,\d r.
  \end{split}
  \end{equation}
  Item~\ref{it:int1}
  of Lemma~\ref{lem:integral}
  (applied for every $t\in [0,T)$ 
  with $a\curvearrowleft t$
  $b\curvearrowleft T$,
  $\lambda\curvearrowleft 0$
  in the notation of 
  Lemma~\ref{lem:integral})
  hence demonstrates that
  for all
  $n \in \N$, $t \in [0,T)$, $x \in O$
  it holds that
  \begin{equation}
  \begin{split}
    &\frac{\E \left[ \int_t^T \lvert \lvert \lvert h_n(r,X^x_{t,r})Z^x_{t,r} 
    - h_0(r,X^x_{t,r})Z^x_{t,r} \rvert \rvert \rvert \,\d r \right]}{V(t,x)} \\
    &\leq c\pi \left[ \sup \limits_{s \in [0,T)} \sup \limits_{y \in O} \left( \frac{\lvert h_n(s,y) - h_0(s,y)}{V(s,y)}\sqrt{T-s} \right) \right].
  \end{split}
  \end{equation}
  Combining this with
  \eqref{eq:approximation_result1_ass}
  shows that
  \begin{equation}
  \begin{split}
    &\limsup \limits_{n \to \infty} \left[ \sup \limits_{t \in [0,T)} \sup \limits_{x \in O} \left( \frac{\bigg\lvert \bigg\lvert \bigg\lvert \E \left[ \int_t^T h_n(r, X^x_{t,r})Z^x_{t,s} \,\d r \right] - \E \left[ \int_t^T h_0(r, X^x_{t,r}) Z^x_{t,r} \,\d r \right] \bigg\rvert \bigg\rvert \bigg\rvert }{V(t,x)}  \, \sqrt{T-t} \right) \right]\\
    &=0.
  \end{split}
  \end{equation}
  The triangle inequality,
  item~\ref{it:app_res2},
  and
  \eqref{eq:ar1_inequality_g}
  hence ensure that
  \begin{equation}
     \limsup \limits_{n \to \infty} \left[ 
     \sup \limits_{t \in [0,T)} \sup \limits_{x \in O} \left( \frac{\lvert \lvert \lvert v_n(t,x) - v_0(t,x) \rvert \rvert \rvert}{V(t,x)}  \, \sqrt{T-t} \right) \right] 
     = 0.
  \end{equation}
  This demonstrates item~\ref{it:app_res3}.
  Moreover, observe that item~\ref{it:app_res3}
  and the hypothesis that
  $V\colon[0,T] \times O \to (0, \infty)$
  is continuous demonstrate that
  for all 
  $\mathcal{K} \subseteq [0,T)\times O$ compact
  it holds that
  \begin{equation}
  \begin{split}
    &\limsup \limits_{n \to \infty} \left[ 
    \sup \limits_{(t,x) \in \mathcal{K}} \left( \lvert \lvert \lvert v_n(t,x) - v_0(t,x) \rvert \rvert \rvert  \, \sqrt{T-t} \right) \right] \\
    &\leq \limsup \limits_{n \to \infty} \left[ 
    \sup \limits_{(t,x) \in \mathcal{K}} \left( \frac{\lvert \lvert \lvert v_n(t,x) - v_0(t,x) \rvert \rvert \rvert}{V(t,x)}  \, \sqrt{T-t} \right) \right] \left[
    \sup \limits_{(t,x) \in \mathcal{K}} V(t,x) \right]\\ 
    &= 0.
  \end{split}
  \end{equation}
  This establishes item~\ref{it:app_res4}.
  The proof of 
  Lemma~\ref{lem:approximation_result1}
  is thus complete.
\end{proof}

The next lemma shows the
construction of an approximating series
of compactly supported, continuous 
function for a certain class of 
continuous functions.
Lemma~\ref{lem:h_approximation} is a generalization of the results in \cite[Lemma 2.3]{BGHJ2019}.

\begin{lemma}\label{lem:h_approximation}
  Let $d \in \N$,
  $T \in (0, \infty)$, 
  let $\lvert \lvert  \cdot  \rvert \rvert\colon \mathbb{R}^d \to [0, \infty)$
  be a norm,
  let $O \subseteq \R^d$
  be a non-empty open set,
  for every 
  $r \in (0, \infty)$
  let $K_r\subseteq [0,T)$,
  $O_r \subseteq O$
  satisfy
  $K_r=[0,\max\{T-\frac{1}{r},0\}]$ and
  $O_r = \{  x \in O\colon \lvert \lvert x \rvert \rvert \leq r \text{ and } \{ y \in \R^d\colon \lvert \lvert y-x \rvert \rvert 
  <\frac{1}{r} \} \subseteq O \}$,
  and let $h \in C([0,T) \times O, \R)$
  satisfy
  $\inf_{r\in(0,\infty)}[\sup_{t\in [0,T)\setminus K_r}
  \sup_{x\in O\setminus O_r}
  (\abs{h(t,x)}\sqrt{T-t})]=0$.
  Then there exist compactly supported
  $h_n\in C([0,T)\times O,\R)$,
  $n\in\N$, which satisfy that
  \begin{equation}
    \limsup_{n\to\infty}\big[
    \sup_{t\in [0,T)}\sup_{x\in O}
    \big(\abs{h_n(t,x)-h(t,x)}
    \sqrt{T-t}\big)\big]
    =0.
  \end{equation}
\end{lemma}

\begin{proof}[Proof of Lemma~\ref{lem:h_approximation}]
  Throughout this proof let
  $U_n\subseteq [0,T)\times O$, $n\in\N$,
  satisfy for all $n\in\N$ that
  $U_n=\{(t,x)\in [0,T)\times O\colon
  (\exists (s,y)\in K_n\times O_n\colon
  \max\{\abs{s-t}, \norm{y-x}\}
  < \frac{1}{2n})\}$.
  Observe that for every $n\in\N$
  it holds that $K_n\times O_n
  \subseteq [0,T)\times O$ 
  is a compact set,
  $U_n\subseteq [0,T)\times O$ is an open set,
  and $ K_n\times O_n \subseteq U_n$.
  Urysohn's lemma 
  (cf., e.g., 
  \cite[Lemma 2.12]{Rudin1974})
  therefore demonstrates that for 
  all $n\in\N$ there exists
  $\varphi_n\in C([0,T)\times O, \R)$
  which satisfies for all 
  $t\in[0,T)$, $x\in O$ that
  $\mathbbm{1}_{K_n\times O_n}(t,x)
  \leq \varphi_n(t,x)
  \leq \mathbbm{1}_{U_n}(t,x)$.
  Let $h_n\colon [0,T)\times O\to \R$,
  $n\in\N$, satisfy for all
  $n\in\N$, $t\in[0,T)$, $x\in O$
  that
  $h_n(t,x)=\varphi(t,x)h(t,x)$.
  This and the fact that for all
  $n\in\N$ it holds that $\varphi_n$
  has compact support 
  implies that for all $n\in\N$
  it holds that
  $h_n\colon [0,T)\times O\to \R$
  is a continuous function with
  compact support.
  In addition, observe that
  \begin{equation}
  \begin{split}
    &\limsup_{n\to\infty}\big[\sup_{t\in[0,T)}
    \sup_{x\in O}\big(
    \abs{h_n(t,x)-h(t,x)}\sqrt{T-t}\big) \big]\\
    &=  \limsup_{n\to\infty}\big[\sup_{t\in[0,T)}
    \sup_{x\in O}([\varphi_n(t,x)-1]
    \abs{h(t,x)}\sqrt{T-t}) \big]\\
    &\leq \limsup_{n\to\infty}\big[ \sup_{t\in [0,T)\setminus K_n}\sup_{x\in O\setminus O_n}
    \big(\abs{h(t,x)}\sqrt{T-t}\big) \big]
    =0.
  \end{split}
  \end{equation}
  The proof of Lemma~\ref{lem:h_approximation}
  is thus complete.
\end{proof}

The following corollary is a consequence
of Lemma~\ref{lem:h_approximation} and
proves for a certain class of 
continuous functions the existence of a 
series of compactly supported,
continuous functions which satisfy
a specific approximation property.
Corollary~\ref{cor:h_approximation} 
is a generalization of the results in \cite[Corollary 2.4]{BGHJ2019}.

\begin{corollary}\label{cor:h_approximation}
  Let $d \in \N$,
  $T \in (0, \infty)$, 
  let $\lvert \lvert  \cdot  \rvert \rvert\colon \mathbb{R}^d \to [0, \infty)$
  be a norm,
  let $O \subseteq \R^d$
  be a non-empty open set,
  for every 
  $r \in (0, \infty)$
  let $K_r\subseteq [0,T)$,
  $O_r \subseteq O$
  satisfy
  $K_r=[0,\max\{T-\frac{1}{r},0\}]$ and
  $O_r = \{  x \in O\colon \lvert \lvert x \rvert \rvert \leq r \text{ and } \{ y \in \R^d\colon \lvert \lvert y-x \rvert \rvert 
  <\frac{1}{r} \} \subseteq O \}$,
  let $h \in C([0,T) \times O, \R)$, 
  $V\in C([0,T]\times O,(0,\infty))$,
  and assume that
  $\inf_{r\in(0,\infty)}[\sup_{t\in [0,T)\setminus K_r}
  \sup_{x\in O\setminus O_r}
  (\frac{\abs{h(t,x)}}{V(t,x)}\sqrt{T-t})]=0$.
  Then there exist compactly supported
  $h_n\in C([0,T)\times O,\R)$,
  $n\in\N$, which satisfy that
  \begin{equation}
    \limsup_{n\to\infty}\bigg[
    \sup_{t\in [0,T)}\sup_{x\in O}
    \bigg(\frac{\abs{h_n(t,x)-h(t,x)}}{V(t,x)}
    \sqrt{T-t}\bigg)\bigg]
    =0.
  \end{equation}
\end{corollary}

\begin{proof}[Proof of Corollary~\ref{cor:h_approximation}]
  Throughout this proof let 
  $g\colon [0,T)\times O\to \R$
  satisfy for all $t\in[0,T)$, $x\in O$
  that 
  $g(t,x)=\frac{h(t,x)}{V(t,x)}$.
  Note that the fact that
  $h\in C([0,T)\times O,\R)$ and
  $V\in C([0,T]\times O, (0,\infty))$
  and the assumption that
  $\inf_{r\in(0,\infty)}[\sup_{t\in [0,T)\setminus K_r}$
  $\sup_{x\in O\setminus O_r}
  (\frac{\abs{h(t,x)}}{V(t,x)}\sqrt{T-t})]=0$
  ensure that $g\in C([0,T)\times O,\R)$
  and 
  \begin{equation}\label{eq:g_n_convergence1}
    \inf\nolimits_{r\in (0,\infty)}\big[
    \sup\nolimits_{t\in[0,T)\setminus K_r}
    \sup\nolimits_{x\in O\setminus O_r}
    \big(\abs{g(t,x)}\sqrt{T-t}\big)\big]=0.
  \end{equation}
  Lemma~\ref{lem:h_approximation}
  (applied with $h\curvearrowleft g$
  in the notation of Lemma~\ref{lem:h_approximation})
  hence demonstrates that there exist
  compactly supported 
  $g_n\in C([0,T)\times O,\R)$, $n\in\N$,
  which satisfy that
  \begin{equation}\label{eq:g_n_convergence2}
    \limsup_{n\to\infty}\big[ \sup_{t\in[0,T)}
    \sup_{x\in O} \big(
    \abs{g_n(t,x)-g(t,x)}
    \sqrt{T-t}\big) \big] 
    =0.
  \end{equation}
  In the next step let
  $h_n\colon [0,T)\times O\to \R$, $n\in\N$,
  satisfy for all $n\in\N$, $t\in [0,T)$,
  $x\in O$ that
  $h_n(t,x)=g_n(t,x)V(t,x)$.
  Observe that \eqref{eq:g_n_convergence2}
  demonstrates that for all $n\in\N$
  it holds that $h_n\in C([0,T)\times O,\R)$
  and 
  \begin{equation}
  \begin{split}
    &\limsup_{n\to\infty} \bigg[\sup_{t\in[0,T)}
    \sup_{x\in O} \bigg(\frac{\abs{h_n(t,x)-h(t,x)}}{V(t,x)}\sqrt{T-t} \bigg)\bigg]\\
    &= \limsup_{n\to\infty} \bigg[\sup_{t\in[0,T)}
    \sup_{x\in O} \Big( \abs{g_n(t,x)-g(t,x)}
    \sqrt{T-t} 
    \Big)\bigg]
    =0.
  \end{split}
  \end{equation}
  The proof of 
  Corollary~\ref{cor:h_approximation}
  is thus complete.
\end{proof}

\subsection{Continuity of solutions of SFPEs}
\label{subsec:regularity}

In this section we establish 
well-definedness and continuity 
properties of SFPE solutions.
The following lemma shows 
well-definedness and
continuity of SFPE solutions
under the assumption that the terminal
condition and the nonlinearity are
bounded.
Lemma~\ref{lem:continuity1}
is a generalization of the results in \cite[Lemma 2.5]{BGHJ2019}.

\begin{lemma}\label{lem:continuity1}
Assume Setting~\ref{setting:sfpe},
let $g \in C(O, \R)$,
$h \in C([0,T) \times O, \R)$ 
be bounded,
and assume $\inf_{t\in[0,T]}\inf_{x\in O}
V(t,x)>0$ 
and for all
$\varepsilon \in (0, \infty)$, $t \in [0,T)$, $s\in (t,T]$, $x\in O$
that
\begin{equation}\label{eq:X_Z_cont}
  \limsup\nolimits_{[0,s)\times O \ni(u,y)\to(t,x)} 
  \big[ \mathbb{P} ( \lvert \lvert X^{y}_{u, s} - X^{x}_{t, s} \rvert \rvert > \varepsilon)
  + \E [ \lvert \lvert \lvert Z^{y}_{u, s} 
  - Z^{x}_{t, s} \rvert \rvert 
  \rvert ] \big] = 0.
\end{equation}
Then
\begin{equation}\label{eq:cont1}
[0,T) \times O \ni (t,x) \mapsto \E \left[ g(X^x_{t,T})Z^x_{t,T} + \int_t^T h(r, X^x_{t,r}) Z^x_{t,r} \,\d r \right] \in \R^m
\end{equation}
is well-defined and continuous.
\end{lemma}

\begin{proof}[Proof of Lemma \ref{lem:continuity1}]
Throughout this proof let 
$M\in [0,\infty)$
satisfy 
$M=\sup_{t\in [0,T)}\sup_{x\in O} 
[\abs{g(x)}+\abs{h(t,x)}]$.
First note that 
\begin{equation}
  \sup_{t\in[0,T)}\sup_{x\in O}
  \bigg[\frac{\abs{g(x)}}{V(T,x)}
  +\frac{\abs{h(t,x)}}{V(t,x)}\sqrt{T-t}\bigg]
  \leq \frac{M(1+\sqrt{T})}{\inf_{t\in[0,T]}\inf_{x\in O}V(t,x)}
  <\infty.
\end{equation}
Lemma~\ref{lem:integrability}
and the fact that 
for all $t\in[0,T)$,
$s\in (t,T]$, $x\in O$
it holds that
$\E[V(s,X^x_{t,s})\lvert\lvert\lvert Z^x_{t,s}\rvert\rvert\rvert]
\leq \frac{c}{\sqrt{s-t}}V(t,x)$
hence demonstrate that 
for all
$t \in [0,T)$, $x \in O$ 
it holds that
\begin{equation}\label{eq:gZ_hZ_int}
\E \bigg[ \lvert g(X^x_{t,T}) \rvert \; \lvert \lvert \lvert Z^x_{t,T} \rvert \rvert \rvert
+ \int_t^T \lvert h(r, X^x_{t,r}) \rvert \; \lvert \lvert \lvert Z^x_{t,r} \rvert \rvert \rvert \,\d r \bigg]
    <\infty.
\end{equation}
This proves that
\eqref{eq:cont1}
is well-defined.
In the next step 
for every $t\in [0,T]$, $x\in O$ 
let $Z^x_{t,t}=0$,
$h(T,x)=g(x)$,
and let $(t_n,x_n)\in [0,T) \times O$, 
$n \in \N_0$,
satisfy 
$\limsup_{n \to \infty } [ \lvert t_n - t_0 \rvert + \lvert \lvert x_n - x_0 \rvert \rvert ]
= 0$.
Note that the fact that
$\limsup_{n \to \infty } \left[ \lvert t_n - t_0 \rvert + \lvert \lvert x_n - x_0 \rvert \rvert \right]
= 0$
ensures that there exists
$q\in (0,\infty)$ which satisfies 
for all $n\in\N_0$ that 
$(t_n,x_n)\in K_{q}\times O_{q}$.
  Furthermore, observe that
  the assumption that
  $\sup_{t\in [0,T)}\sup_{x\in O} 
  \abs{h(t,x)} \leq M$  
  proves that
  for all $n\in\N$, $s\in [0,T)$
  it holds that
  \begin{equation}\label{eq:hZ_dom}
  \begin{split}
    &\E\big[\lvert\lvert\lvert h(s,X^{x_n}_{t_n,\max\{s,t_n\}})
    Z^{x_n}_{t_n,\max\{s,t_n\}}
    -h(s,X^{x_0}_{t_0,\max\{s,t_0\}})
    Z^{x_0}_{t_0,\max\{s,t_0\}}
    \rvert\rvert\rvert \big]\\
    &\leq \E\big[\lvert\lvert\lvert h(s,X^{x_n}_{t_n,\max\{s,t_n\}})
    Z^{x_n}_{t_n,\max\{s,t_n\}}
    -h(s,X^{x_n}_{t_n,\max\{s,t_n\}})
    Z^{x_0}_{t_0,\max\{s,t_0\}}
    \rvert\rvert\rvert \big]\\
    &\quad + \E\big[\lvert\lvert\lvert 
    h(s,X^{x_n}_{t_n,\max\{s,t_n\}})
    Z^{x_0}_{t_0,\max\{s,t_0\}}
    -h(s,X^{x_0}_{t_0,\max\{s,t_0\}})
    Z^{x_0}_{t_0,\max\{s,t_0\}}
    \rvert\rvert\rvert \big]\\
    &\leq \E\big[\abs{h(s,X^{x_n}_{t_n,\max\{s,t_n\}})}\,\lvert\lvert\lvert 
    Z^{x_n}_{t_n,\max\{s,t_n\}}
    -Z^{x_0}_{t_0,\max\{s,t_0\}}
    \rvert\rvert\rvert \big]\\
    &\quad + \E\big[\lvert\lvert\lvert h(s,X^{x_n}_{t_n,\max\{s,t_n\}})
    Z^{x_0}_{t_0,\max\{s,t_0\}}
    -h(s,X^{x_0}_{t_0,\max\{s,t_0\}})
    Z^{x_0}_{t_0,\max\{s,t_0\}}
    \rvert\rvert\rvert \big]\\
    &\leq M \E\big[\lvert\lvert\lvert 
    Z^{x_n}_{t_n,\max\{s,t_n\}}
    -Z^{x_0}_{t_0,\max\{s,t_0\}}
    \rvert\rvert\rvert \big]\\
     &\quad + \E\big[\lvert\lvert\lvert h(s,X^{x_n}_{t_n,\max\{s,t_n\}})
     Z^{x_0}_{t_0,\max\{s,t_0\}}
    -h(s,X^{x_0}_{t_0,\max\{s,t_0\}})
    Z^{x_0}_{t_0,\max\{s,t_0\}}
    \rvert\rvert\rvert \big] .
  \end{split}
  \end{equation}
  Next note that
  the fact that
  $\sup_{t\in [0,T)}\sup_{x\in O} 
  \abs{h(t,x)}\leq M$
  ensures
  that for all $n\in\N$,
  $s\in[0,T)$ it holds that
  \begin{equation}\label{eq:h_dom}
    \lvert\lvert\lvert h(s,X^{x_n}_{t_n,\max\{s,t_n\}})
     Z^{x_0}_{t_0,\max\{s,t_0\}}
    -h(s,X^{x_0}_{t_0,\max\{s,t_0\}})
    Z^{x_0}_{t_0,\max\{s,t_0\}}
    \rvert\rvert\rvert
    \leq 2M \lvert\lvert\lvert Z^{x_0}_{t_0,\max\{s,t_0\}}
    \rvert\rvert\rvert.
  \end{equation}
  Moreover, observe that
  the hypothesis that for all
  $t\in[0,T)$, $s\in (t,T]$, $x\in O$ 
  it holds that 
  $\E[V(s,X^x_{t,s})\lvert\lvert\lvert   Z^x_{t,s}\rvert\rvert\rvert]
  \leq \frac{c}{\sqrt{s-t}}V(t,x)$,
  the 
  fact that
  for all $n\in\N_0$ it holds that 
  $(t_n,x_n)\in K_{q}\times O_{q}$,
  and the assumption that
  $\inf_{t\in[0,T]}\inf_{x\in O}
V(t,x)>0$
  show that 
  for all $n\in\N$, $s\in(t_0,T]$
  it holds that
  \begin{equation}\label{eq:Z0_bd}
  \begin{split}
    &\E[\lvert\lvert\lvert Z^{x_0}_{t_0,s}\rvert\rvert\rvert]
    = \E\bigg[ 
    \frac{V(s,X^{x_0}_{t_0,s})}{V(s,X^{x_0}_{t_0,s})} 
    \lvert\lvert\lvert Z^{x_0}_{t_0,s}\rvert\rvert\rvert\bigg]
    \leq \bigg(\sup_{r\in [0,T]}\sup_{y\in O}
    \frac{1}{V(r,y)}\bigg)
    \E\bigg[V(s,X^{x_0}_{t_0,s})  
    \lvert\lvert\lvert Z^{x_0}_{t_0,s}\rvert\rvert\rvert\bigg]\\
    &\leq \bigg(\sup_{r\in [0,T]}\sup_{y\in O}
    \frac{1}{V(r,y)}\bigg)
   \frac{c}{\sqrt{s-t_0}} V(t_0,x_0)
   <\infty.
  \end{split}
  \end{equation}  
  Lebesgue's dominated convergence theorem,
  \eqref{eq:X_Z_cont},
  \eqref{eq:hZ_dom},
  and \eqref{eq:h_dom} 
  hence demonstrate that
  for all $s\in [0,T)$
  it holds that
  \begin{equation}\label{eq:g_h_conv}
    \limsup_{n\to\infty}
    \E\big[\lvert\lvert\lvert 
    h(s,X^{x_n}_{t_n,\max\{s,t_n\}})
    Z^{x_n}_{t_n,\max\{s,t_n\}}
    -h(s,X^{x_0}_{t_0,\max\{s,t_0\}})
    Z^{x_0}_{t_0,\max\{s,t_0\}}
    \rvert\rvert\rvert \big]
    =0.
  \end{equation}
  Moreover, observe that
  \eqref{eq:hZ_dom}
  and the assumption that
  $\sup_{t\in [0,T)}\sup_{x\in O} 
  \abs{h(t,x)}
  \leq M$ 
  imply  
  that for all $n\in\N$, $s\in[0,T)$
  it holds that  
  \begin{equation}
  \begin{split}
    &\E\big[\lvert\lvert\lvert h(s,X^{x_n}_{t_n,\max\{s,t_n\}})
    Z^{x_n}_{t_n,\max\{s,t_n\}}
    -h(s,X^{x_0}_{t_0,\max\{s,t_0\}})
    Z^{x_0}_{t_0,\max\{s,t_0\}}
    \rvert\rvert\rvert \big]\\
    &\leq M \E\big[\lvert\lvert\lvert 
    Z^{x_n}_{t_n,\max\{s,t_n\}}
    -Z^{x_0}_{t_0,\max\{s,t_0\}}
    \rvert\rvert\rvert \big]
    + 2M\E\big[\lvert\lvert\lvert 
     Z^{x_0}_{t_0,\max\{s,t_0\}}
    \rvert\rvert\rvert \big].
  \end{split}
  \end{equation}
  The dominated convergence theorem,
  \eqref{eq:X_Z_cont},
  \eqref{eq:Z0_bd}, and \eqref{eq:g_h_conv}
  therefore show that 
  \begin{equation}\label{eq:hZ_conv1}
  \begin{split}
    \limsup \limits_{n \to \infty} \bigg( 
    \int_{t_0}^T\E\Big[ 
    &\lvert\lvert\lvert h(r, X^{x_n}_{t_n,\max\{r,t_n\}})
    Z^{x_n}_{t_n,\max\{r,t_n\}}
     -h(r, X^{x_0}_{t_0,\max\{r,t_0\}})
    Z^{x_0}_{t_0,\max\{r,t_0\}}
    \rvert\rvert\rvert\Big] \d r\bigg)
    =0.
  \end{split}
\end{equation}    
  In the next step note that 
  for all $k, n\in\N_0$ with 
  $t_n\leq t_k$
  it holds that
  \begin{equation}\label{eq:hZ_conv2}
  \begin{split}
    &\bigg\lvert\bigg\lvert\bigg\lvert
    \E\bigg[\int_{t_n}^T h(r, X^{x_n}_{t_n,r})
    Z^{x_n}_{t_n,r}\,\d r\bigg]
    -  \E\bigg[\int_{t_k}^T h(r, X^{x_k}_{t_k,r})
    Z^{x_k}_{t_k,r}\,\d r\bigg]
    \bigg\rvert\bigg\rvert\bigg\rvert\\
    &= \bigg\lvert\bigg\lvert\bigg\lvert
    \E\bigg[\int_{t_n}^{t_k} h(r, X^{x_n}_{t_n,r})
    Z^{x_n}_{t_n,r}\,\d r\bigg]
     + \E\bigg[\int_{t_k}^T 
    \Big( h(r, X^{x_n}_{t_n,r})
    Z^{x_n}_{t_n,r}
    -h(r, X^{x_k}_{t_k,r})
    Z^{x_k}_{t_k,r}
    \Big)\,\d r\bigg]
    \bigg\rvert\bigg\rvert\bigg\rvert\\
    &\leq 
     \E \bigg[ \int_{t_n}^{t_k}  
  \abs{h(r, X^{x_n}_{t_n,r})}\ 
  \lvert\lvert\lvert Z^{x_n}_{t_n,r}\rvert\rvert\rvert \,\d r\bigg]
    + \E\bigg[\int_{t_k}^T 
    \lvert\lvert\lvert h(r, X^{x_n}_{t_n,r})
    Z^{x_n}_{t_n,r}
    -h(r, X^{x_k}_{t_k,r})
    Z^{x_k}_{t_k,r}
    \rvert\rvert\rvert\,\d r\bigg].
  \end{split}
  \end{equation}
  Furthermore, observe that
  Fubini's theorem, the
  fact that $\sup_{t\in[0,T)\sup_{x\in O}}
  \abs{h(t,x)}\leq M$,
  and the assumption that
  for all $t\in[0,T)$,
$s\in (t,T]$, $x\in O$
it holds that
$\E[V(s,X^x_{t,s})\lvert\lvert\lvert Z^x_{t,s}\rvert\rvert\rvert]
\leq \frac{c}{\sqrt{s-t}}V(t,x)$
  ensure that for all $k,n\in\N_0$
  with $t_n\leq t_k$
  it holds that
  \begin{equation}\label{eq:hZ_conv3}
  \begin{split}
    &\E \bigg[ \int_{t_n}^{t_k}  
  \abs{h(r, X^{x_n}_{t_n,r})}\ 
  \lvert\lvert\lvert Z^{x_n}_{t_n,r}\rvert\rvert\rvert \,\d r\bigg] 
  \leq M \int_{t_n}^{t_k} \E\bigg[\frac{V(r, X^{x_n}_{t_n,r})}{V(r, X^{x_n}_{t_n,r})}\lvert\lvert\lvert Z^{x_n}_{t_n,r}\rvert\rvert\rvert\bigg] \,\d r\\
  &\leq \frac{M}{\inf_{t\in[0,T]}\inf_{x\in O}V(t,x)} \int_{t_n}^{t_k} \E\bigg[V(r, X^{x_n}_{t_n,r})
  \lvert\lvert\lvert Z^{x_n}_{t_n,r}\rvert\rvert\rvert\bigg] \,\d r\\
  &\leq \frac{M}{\inf_{t\in[0,T]}\inf_{x\in O}V(t,x)} \int_{t_n}^{t_k} V(t_n,x_n) \frac{c}{\sqrt{r-t_n}} \,\d r\\
  &\leq  \frac{cM}{\inf_{t\in[0,T]}\inf_{x\in O} V(t,x)} 
  \Big[\sup_{u\in K_q}\sup_{y\in O_q}V(u,y)\Big]
  \int_{0}^{t_k-t_n}  \frac{1}{\sqrt{z}} \,\d z.
  \end{split}
  \end{equation}
  Combining this with
  \eqref{eq:hZ_conv1},
  \eqref{eq:hZ_conv2},
  and the fact
  that $\lim_{\varepsilon\to 0}
  \int_0^{\abs{\varepsilon}} \frac{1}{\sqrt{z}}\,\d z   
  =0$
  proves that 
  \begin{equation}\label{eq:hZ_conv5}
  \begin{split}
  &\limsup \limits_{n \to \infty} \; \bigg\lvert \bigg\lvert \bigg\lvert \E \bigg[ \int_{t_n}^T  h(r, X^{x_n}_{t_n,r}) Z^{x_n}_{t_n, r}\,\d r \bigg] - 
  \E \bigg[ \int_{t_0}^T  h(r, X^{x_0}_{t_0,r}) Z^{x_0}_{t_0, r}\,\d r \bigg] \bigg\rvert \bigg\rvert \bigg\rvert \\
    &\leq \limsup \limits_{n \to \infty} \bigg( 
    \frac{cM}{\inf_{t\in[0,T]}\inf_{x\in O} V(t,x)} 
   \Big[\sup_{u\in K_q}\sup_{y\in O_q}
   V(u,y)\Big] 
   \int_{0}^{\abs{t_0-t_n}} 
  \frac{1}{\sqrt{z}}\,\d z \\
  & \quad + \E\bigg[\int_{\max\{t_0,t_n\}}^T 
    \lvert\lvert\lvert h(r, X^{x_n}_{t_n,r})
    Z^{x_n}_{t_n,r}
    -h(r, X^{x_0}_{t_0,r})
    Z^{x_0}_{t_0,r}
    \rvert\rvert\rvert\,\d r\bigg]\bigg)
  =0.
\end{split}
\end{equation}
  This and \eqref{eq:g_h_conv}
  (applied with $h\curvearrowleft ([0,T)\times O \ni (t,x)\mapsto g(x)\in \R)$)
  show that
  \begin{equation}\label{eq:gZ_hZ_conv}
  \begin{split}
    \limsup_{n\to\infty}
    \bigg( &\E\big[\lvert\lvert\lvert g(X^{x_n}_{t_n,T})Z^{x_n}_{t_n,T}
    -g(X^{x_0}_{t_0,T})Z^{x_0}_{t_0,T}
    \rvert\rvert\rvert \big]\\
    & +\bigg\lvert\bigg\lvert\bigg\lvert\E\bigg[\int_{t_n}^T 
    h(r,X^{x_n}_{t_n,r})
    Z^{x_n}_{t_n,r}\,\d r\bigg]
     -\E\bigg[\int_{t_0}^T
    h(r,X^{x_0}_{t_0,r})
    Z^{x_0}_{t_0,r}
    \,\d r\bigg]
    \bigg\rvert\bigg\rvert\bigg\rvert
    \bigg)
    =0.
  \end{split}
  \end{equation}
  This proves that \eqref{eq:cont1}
  is continuous.
  The proof of Lemma~\ref{lem:continuity1}
  is thus complete. 
\end{proof}

The next lemma extends 
the results of Lemma~\ref{lem:continuity1}
by demonstrating well-definedness
and continuity of SFPE solutions 
for continuous terminal conditions and 
nonlinearities.
Lemma~\ref{lem:continuity2}
is a generalization of
the results in \cite[Lemma 2.6]{BGHJ2019}

\begin{lemma}\label{lem:continuity2}
  Assume Setting~\ref{setting:sfpe},
  assume $\inf_{t\in[0,T]}\inf_{x\in O}
V(t,x)>0$,
  assume for all
  $\varepsilon \in (0, \infty)$, 
  $t \in [0,T)$, $s\in(t,T]$, 
  $x\in O$ that
  \begin{equation}
    \limsup\nolimits_{[0,s)\times O\ni (u,y)\to(t,x)}
    \Big[ 
  \mathbb{P} \big( \lvert \lvert X^{y}_{u, s} 
    - X^{x}_{t, s} \rvert \rvert 
    > \varepsilon \big) 
    +\E \big[ \lvert \lvert \lvert Z^{y}_{u,s} 
     - Z^{x}_{t,s} \rvert \rvert \rvert \big] \Big] 
     = 0,
  \end{equation}
  let $g \in C(O, \R)$, 
 $h \in C([0,T) \times O, \R)$,
  and $v\colon[0,T) \times O \to \R^m$
  satisfy 
  $\inf_{r \in (0, \infty)} 
  [\sup_{t \in [0,T)\setminus K_r}$ 
  $\sup_{x \in O \setminus O_r} ( \frac{\lvert g(x) \rvert}{V(T,x)} 
  + \frac{\lvert h(t,x) \rvert}{V(t,x)}
  \sqrt{T-t} ) ]
  =0 $
  and for all 
  $t \in [0,T)$, $x \in O$
  that 
  \begin{equation}
    v(t,x) = \E \left[ g(X^x_{t,T})Z^x_{t,T} 
    + \int_t^T h(r, X^x_{t,r}) Z^x_{t,r} \,\d r \right]
  \end{equation}
  (cf. Lemma~\ref{lem:integrability}).
  Then
  \begin{enumerate}[label=(\roman*)]
  \item\label{it:cont2_1}
  it holds that
  $v \in C([0,T) \times O, \R^m)$
  and
  \item\label{it:cont2_2}
  if - in addition to the
  above assumptions -
  it holds for all
  $u\in (0,\infty)$ that
  \begin{equation}
    \limsup_{r \to\infty} \Bigg[ \sup_{t \in [0,T)\setminus K_r} \sup_{x \in O \setminus O_r} \Bigg(
    \frac{\E\Big[\int_t^T \mathbbm{1}_{O_u}(X^x_{t,s})\,
    \lvert\lvert\lvert Z^x_{t,s}\rvert\rvert\rvert
    (\d s + \delta_T(\d s))\Big]
    \sqrt{T-t}}{V(t,x)}\Bigg) \Bigg]
    =0,
  \end{equation}  
  then it holds that
  \begin{equation}
    \lim \limits_{r \to \infty} \sup \limits_{t \in [0,T)\setminus K_r} 
    \sup \limits_{x \in O \setminus O_r} \left( \frac{\lvert \lvert \lvert v(t,x) \rvert \rvert \rvert}{V(t,x)}  \, \sqrt{T-t} \right) 
    = 0.
  \end{equation}
  \end{enumerate}
\end{lemma}

\begin{proof}[Proof of Lemma \ref{lem:continuity2}]
  First observe that
  Corollary~\ref{cor:h_approximation}
  demonstrates that there exists
  compactly supported
  $g_n\in C(O, \R)$, $n \in \N$,
  and $h_n\in C([0,T) \times O, \R)$, 
  $n \in \N$,
  which satisfy
  \begin{equation}\label{eq:con2_gh}
     \limsup \limits_{n \to \infty} \left[ 
     \sup \limits_{t \in [0,T)} \sup \limits_{x \in O} \left( \frac{\lvert g_n(x)-g(x) \rvert}{V(T,x)} 
     + \frac{\lvert h_n(t,x) - h(t,x) \rvert}{V(t,x)}\sqrt{T-t} \right) \right]
     = 0.
  \end{equation}
  Let 
  $v_n\colon [0,T) \times O \to \R^m$
  satisfy for all
  $n \in \N$, $t \in [0,T)$, $x \in O$ 
  that
  \begin{equation}\label{eq:v_n}
    v_n(t,x) = \E \left[ g_n(X^x_{t,T})Z^x_{t,T} 
    + \int_t^T h_n(s, X^x_{t,s})Z^x_{t,s}\,\d s \right]
  \end{equation}
  (cf.\ Lemma~\ref{lem:integrability}).
  Observe that Lemma~\ref{lem:continuity1}
  and the fact that for all $n\in\N$
  it holds that $g_n$ and $h_n$
  are compactly supported
  and continuous
  show that for all
  $n \in \N$ it holds that
  $v_n\colon [0,T) \times O \to \R^m$
  is continuous.
  Furthermore, note that the fact that
  $g_n\colon O \to \R$, $n \in \N$,
  and $h_n\colon[0,T) \times O \to \R$, 
  $n \in \N$,
  are compactly supported 
  ensures that 
  for every $n \in \N$ there exists
  $r_n \in (0, \infty)$ 
  which satisfies that for all
  $r\in[r_n,\infty)$,
  $t \in [0,T)\setminus K_r$, 
  $x \in O \setminus O_r$
  it holds that
  $g_n(x)=0=h_n(t,x)$.
  This demonstrates that for all
  $n\in\N$ it holds that
  \begin{equation}
    \inf_{r\in(0,\infty)} \bigg[ \sup_{t\in [0,T)\setminus K_r} 
    \sup_{x\in O\setminus O_r}
    \Big( \frac{\abs{g_n(x)}}{V(T,x)}
    +\frac{\abs{h_n(t,x)}}{V(t,x)}\sqrt{T-t}
    \Big)\bigg]
    =0.
  \end{equation}
  Item~\ref{it:app_res4} of 
  Lemma~\ref{lem:approximation_result1},
  \eqref{eq:con2_gh}, 
  and the fact that for all
  $n \in \N$ it holds that
  $v_n\colon [0,T) \times O \to \R^m$
  is continuous
  therefore imply that
  $v \colon[0,T) \times O \to \R^m$
  is continuous.
  This etablishes item~\ref{it:cont2_1}.
  To prove item~\ref{it:cont2_2} 
  assume for all $u\in (0,\infty)$ that
  \begin{equation}\label{eq:cont2_2ass}
    \limsup_{r\to\infty} \Bigg[ \sup_{t \in [0,T)\setminus K_r} \sup_{x \in O \setminus O_r} \Bigg(
    \frac{\E\Big[\int_t^T \mathbbm{1}_{O_u}(X^x_{t,s})\,
    \lvert\lvert\lvert Z^x_{t,s}\rvert\rvert\rvert
    (\d s + \delta_T(\d s))\Big]
    \sqrt{T-t}}{V(t,x)}\Bigg) \Bigg]
    =0.
  \end{equation}  
  This
  and the fact that
  for all $n\in\N$,
  $x \in O \setminus O_{r_n}$
  it holds that
  $g_n(x)=0$
  ensure 
  for all $n\in\N$
  that
  \begin{equation}\label{eq:cont2_2_g}
  \begin{split}
    &\limsup_{r\to\infty} \Bigg[ \sup_{t \in [0,T)\setminus K_r} \sup_{x \in O \setminus O_r} \Bigg(
    \frac{\E [\abs{g_n(X^x_{t,T})}\, 
    \lvert\lvert\lvert Z^x_{t,T}\rvert\rvert\rvert ]
    \sqrt{T-t}}{V(t,x)}\Bigg) \Bigg]\\
    &\leq 
    \limsup_{r\to\infty} 
    \Bigg[ \Big[\sup_{y\in O_{r_n}} 
    \,\abs{g_n(y)}\Big]\bigg[ \sup_{t \in [0,T)\setminus K_r} \sup_{x \in O \setminus O_r} \Bigg(
    \frac{\E [ \mathbbm{1}_{O_{r_n}}(X^x_{t,T}) \lvert\lvert\lvert
     Z^x_{t,T}\rvert\rvert\rvert ]
    \sqrt{T-t}}{V(t,x)}\Bigg) \bigg]
    \Bigg]
    =0.
  \end{split}
\end{equation}    
  Moreover, note that
  \eqref{eq:cont2_2ass}
  and the fact that
  for all $n\in\N$, 
  $t \in [0,T)\setminus K_{r_n}$, 
  $x \in O \setminus O_{r_n}$
  it holds that
  $h_n(t,x)=0$
  demonstrate that for all $n\in\N$
  it holds that
  \begin{equation}\label{eq:cont2_2_h}
  \begin{split}
    &\limsup_{r\to\infty} \Bigg[ \sup_{t \in [0,T)\setminus K_r} \sup_{x \in O \setminus O_r} \Bigg(
    \frac{\E \Big[\int_t^T
     \,\abs{h_n(s, X^x_{t,s})}\, 
    \lvert\lvert\lvert Z^x_{t,s}\rvert\rvert\rvert \,\d s\Big]
    \sqrt{T-t}}{V(t,x)}\Bigg) \Bigg]\\
    &\leq 
    \limsup_{r\to\infty}\Bigg[
    \Big[\sup_{u\in K_{r_n}}\sup_{y\in O_{r_n}} \abs{h_n(u,y)}\Big]\\
    &\hspace{6em}\cdot \bigg[ \sup_{t \in [0,T)\setminus K_r} \sup_{x \in O \setminus O_r} \Bigg(
    \frac{\E \Big[ \int_t^T 
    \mathbbm{1}_{O_{r_n}}(X^x_{t,s})\,
    \lvert\lvert\lvert
     Z^x_{t,s}\rvert\rvert\rvert\,\d s 
     \Big]
    \sqrt{T-t}}{V(t,x)}\Bigg) \bigg]\Bigg]
    =0.
  \end{split}
\end{equation}    
  Combining this, \eqref{eq:v_n}, 
  and \eqref{eq:cont2_2_g}
  shows that for all
  $n \in \N$ it holds that
  \begin{equation}
  \begin{split}
    \limsup \limits_{r \to \infty} \left[ \sup \limits_{t \in [0,T)\setminus K_r} \sup \limits_{x \in O \setminus O_r} \left( \frac{\lvert \lvert \lvert v_n(t,x) \rvert \rvert \rvert}{V(t,x)}  \, \sqrt{T-t} \right) \right] 
    = 0.
    \end{split}
    \end{equation}  
  Item~\ref{it:app_res3} of 
  Lemma~\ref{lem:approximation_result1}
  therefore proves that 
  \begin{equation}
  \begin{split}
    &\limsup \limits_{r \to \infty} \left[ \sup \limits_{t \in [0,T)\setminus K_r} \sup \limits_{x \in O \setminus O_r} \left( \frac{\lvert \lvert \lvert v(t,x) \rvert \rvert \rvert}{V(t,x)}  \, \sqrt{T-t} \right) \right] \\
    &\leq \inf \limits_{n \in \N} \left( \limsup \limits_{r \to \infty} \left[ \sup \limits_{t \in [0,T)\setminus K_r} \sup \limits_{x \in O \setminus O_r} \left( \frac{\lvert \lvert \lvert v(t,x) - v_n(t,x) \rvert \rvert \rvert + \lvert \lvert \lvert v_n(t,x) \rvert \rvert \rvert}{V(t,x)}  \, \sqrt{T-t} \right) \right] \right) \\
    &= \inf \limits_{n \in \N} \left( \limsup \limits_{r \to \infty} \left[ \sup \limits_{t \in [0,T)\setminus K_r} \sup \limits_{x \in O \setminus O_r} \left( \frac{\lvert \lvert \lvert v(t,x) - v_n(t,x) \rvert \rvert \rvert}{V(t,x)}  \, \sqrt{T-t} \right) \right] \right) \\
    &\leq \inf \limits_{n \in \N} \left( \sup \limits_{t \in [0,T)} \sup \limits_{x \in O} \left( \frac{\lvert \lvert \lvert v(t,x) - v_n(t,x) \rvert \rvert \rvert}{V(t,x)}  \, \sqrt{T-t} \right) \right) \\
    &\leq \limsup \limits_{n \to \infty} \left( \sup \limits_{t \in [0,T)} \sup \limits_{x \in O} \left( \frac{\lvert \lvert \lvert v(t,x) - v_n(t,x) \rvert \rvert \rvert}{V(t,x)}  \, \sqrt{T-t} \right) \right) 
    = 0.
  \end{split}
  \end{equation}
  This establishes item~\ref{it:cont2_2}.
  The proof of 
  Lemma \ref{lem:continuity2}
  is thus complete.
\end{proof}

\subsection{Contractivity of SFPEs}
\label{subsec:contractivity}

The following lemma 
derives a contractivity
property for SFPEs.
Lemma~\ref{lem:contraction} 
is a generalization of the
results in \cite[Lemma 2.8]{BGHJ2019}.

\begin{lemma}\label{lem:contraction}
  Assume Setting~\ref{setting:sfpe},
  let $L\in(0,\infty)$,
  let $f\colon [0,T) \times O \times \R^m \to \mathbb{R}$ be measurable,
  assume for all $t \in [0,T)$, $x \in O$, $y,z \in \R^m$ that
  $|f(t,x,y)-f(t,x,z)| \leq 
  L \lvert\lvert\lvert y-z \rvert\rvert\rvert$,
  let $v,w \colon [0,T) \times O \to \R^m$ 
  be measurable,
  and assume that
  \begin{equation}\label{eq:contraction1}
    \sup \limits_{t \in [0,T)} \sup \limits_{x \in O}
\left[ \frac{\lvert \lvert \lvert v(t,x)\rvert \rvert \rvert 
    + \lvert \lvert \lvert w(t,x)\rvert \rvert \rvert}{V(t,x)}  \, \sqrt{T-t} \right] 
    < \infty.
  \end{equation}
  Then it holds for all 
  $\lambda \in (0, \infty)$, 
  $t \in [0,T)$, 
  $x \in O$ that
  \begin{equation}\label{eq:contraction2}
  \begin{split}
    &\E \left[ \int_t^T \abs{ f(r, X^x_{t,r}, v(r, X^x_{t,r}))
    -f(r, X^x_{t,r}, w(r, X^x_{t,r}))} \, \lvert \lvert \lvert Z^x_{t,r} \rvert \rvert \rvert \,\d r \right] \\
    &\leq 
    cL\sqrt{\frac{\pi^3}{4\lambda(T-t)}}    
    V(t,x) e^{-\lambda t}
    \left[ \sup \limits_{s \in [0,T)} \sup \limits_{y \in O} \left( \frac{e^{\lambda s} \lvert \lvert \lvert v(s,y) - w(s,y) \rvert \rvert \rvert}{V(s,y)} \sqrt{T-s} \right) \right].
  \end{split}
  \end{equation}
\end{lemma}

\begin{proof}[Proof of Lemma \ref{lem:contraction}]
  First note that the assumption that 
  $f$, $v$, and $w$
  are measurable and the fact that 
  for all $t \in [0,T]$, $x \in O$ 
  it holds that
  $X^x_{t}$
  and
  $Z^x_{t}$
  are measurable
  show that for all 
  $t \in [0,T)$, $x \in O$ 
  it holds that 
  $(t,T] \times \Omega \ni (s, \omega) \mapsto
\left| f(s, X^x_{t,s}(\omega), v(s, X^x_{t,s}(\omega))) 
- f(s, X^x_{t,s}(\omega), w(s, X^x_{t,s}(\omega))) \right| \, \lvert \lvert \lvert Z^x_{t,s} \rvert \rvert \rvert   \in \R$
  is measurable.
  Fubini's theorem and 
  the assumption that for all
  $t \in [0,T)$, $x \in O$, $y,z \in \R^m$
  it holds that
  $| f(t,x,y) - f(t,x,z) |$ 
  $\leq L \lvert \lvert \lvert y-z\rvert \rvert \rvert$
  therefore imply that for all
  $\lambda \in (0, \infty)$, $t \in [0,T)$,
  $x \in O$ 
  it holds that
  \begin{equation}\label{eq:contradiction1}
  \begin{split}
    & \E \left[ \int_t^T |f(r, X^x_{t,r}, v(r, X^x_{t,r})) -  f(r, X^x_{t,r}, w(r, X^x_{t,r}))| 
    \,\lvert \lvert \lvert Z^x_{t,r}\rvert \rvert \rvert   \,\d r \right] \\
    & \leq \E \left[ \int_t^T L \lvert \lvert \lvert v(r, X^x_{t,r})-w(r, X^x_{t,r}) \rvert \rvert \rvert \, \lvert \lvert \lvert Z^x_{t,r}\rvert \rvert \rvert   \,\d r \right] \\
    & = L \int_t^T \E \left[ \frac{e^{\lambda r} \lvert \lvert \lvert v(r, X^x_{t,r})-w(r, X^x_{t,r})\rvert \rvert \rvert}{V(r, X^x_{t,r})} \sqrt{T-r}  \, 
    \frac{V(r, X^x_{t,r})}{\sqrt{T-r}} \lvert \lvert \lvert Z^x_{t,r}\rvert \rvert \rvert \right] e^{-\lambda r} \,\d r \\
    & \leq L  \left[ \sup \limits_{s \in [0,T)} \sup \limits_{y \in O} \left( \frac{e^{\lambda s} \lvert \lvert \lvert v(s,y)-w(s,y) \rvert \rvert \rvert}{V(s,y)} \sqrt{T-s}  \right) \right]
    \int_t^T
    \frac{\E \left[ V(r, X^x_{t,r}) \lvert \lvert \lvert Z^x_{t,r}\rvert \rvert \rvert \right]}{\sqrt{T-r}}  
     e^{-\lambda r} \,\d r.
    \end{split}
    \end{equation}
  This, the fact that
  for all 
  $t \in [0,T)$, 
  $r \in (t,T]$, $x \in O$
  it holds that 
  $\E [ V(r, X^x_{t,r}) \lvert \lvert \lvert Z^x_{t,r}\rvert \rvert \rvert ] \leq \frac{c}{\sqrt{r-t}} V(t,x)$,
  and
  item~\ref{it:int2} of 
  Lemma~\ref{lem:integral}
  (applied for every $t\in[0,T)$ 
  with $a\curvearrowleft t$,
  $b\curvearrowleft T$
  in the notation of 
  Lemma~\ref{lem:integral})
  ensure that
  for all
  $\lambda \in (0, \infty)$, $t \in [0,T)$,
  $x \in O$ 
  it holds that
  \begin{equation}\label{eq:contradiction2}
    \begin{split}
    & \E \left[ \int_t^T |f(r, X^x_{t,r}, v(r, X^x_{t,r})) -  f(r, X^x_{t,r}, w(r, X^x_{t,r}))| 
    \,\lvert \lvert \lvert Z^x_{t,r}\rvert \rvert \rvert   \,\d r \right] \\
    & \leq L  \left[ \sup \limits_{s \in [0,T)} \sup \limits_{y \in O} \left( \frac{e^{\lambda s} \lvert \lvert \lvert v(s,y)-w(s,y) \rvert \rvert \rvert}{V(s,y)} \sqrt{T-s}  \right) \right]
     \int_t^T\frac{c V(t,x)}{\sqrt{(r-t)(T-r)}}  e^{-\lambda r} \, \d r \\
    &=cL \left[ \sup \limits_{s \in [0,T)} \sup \limits_{y \in O} \left( \frac{e^{\lambda s} \lvert \lvert \lvert v(s,y)-w(s,y) \rvert \rvert \rvert}{V(s,y)} \sqrt{T-s}\right) \right] V(t,x)
     \int_t^{T}
    \frac{1}{\sqrt{(r-t)(T-r)}}\, 
    e^{-\lambda r} \,\d r\\
    &\leq cL\sqrt{\frac{\pi^3}{4\lambda(T-t)}}    
    V(t,x) e^{-\lambda t}
    \left[ \sup \limits_{s \in [0,T)} \sup \limits_{y \in O} \left( \frac{e^{\lambda s} \lvert \lvert \lvert v(s,y)-w(s,y) \rvert \rvert \rvert}{V(s,y)} \sqrt{T-s}\right) \right] .
  \end{split}
  \end{equation}
  The proof of Lemma~\ref{lem:contraction}
  is thus complete.
\end{proof}

  \subsection{Existence and uniqueness of solutions of SFPEs}
\label{subsec:exist_uniq}

In this section we use the results
of Section~\ref{subsec:continuity}-\ref{subsec:contractivity} to obtain
our main result of this section,
Theorem~\ref{thm:ex_cont_solution}.
The following lemma
constructs a vector space of SFPE solutions 
and shows that this space is a Banach space.
Lemma~\ref{lem:banach_space2} generalizes 
a result in the proof of 
\cite[Theorem 2.9]{BGHJ2019}.

\begin{lemma}\label{lem:banach_space2}
  Let $d \in \N$,
  $T \in (0, \infty)$,
  let $\lvert \lvert \cdot \rvert \rvert\colon \R^d \to [0,\infty)$
  be a norm,
  let $O \subseteq \R^d$
  be a non-empty open set,
  for every 
  $r \in (0, \infty)$
  let $K_r\subseteq [0,T)$,
  $O_r \subseteq O$
  satisfy
  $K_r=[0,\max\{T-\frac{1}{r},0\}]$,
  $O_r = \{  x \in O\colon \lvert \lvert x \rvert \rvert \leq r \text{ and } \{ y \in \R^d\colon \lvert \lvert y-x \rvert \rvert 
  <\frac{1}{r} \} \subseteq O \}$,
  let $(B,\lvert\lvert\lvert\cdot\rvert\rvert\rvert)$
  be an $\R$-Banach space,
  let 
  $V \in C([0,T] \times O, (0, \infty))$
  satisfy 
  $\sup_{r \in (0, \infty)} [ \inf_{t \in [0,T)\setminus K_r} \inf_{x \in O \setminus O_r}$ 
  $V(t,x) ] 
  = \infty$,
  let $\mathcal{V}$ satisfy
  \begin{equation}
    \mathcal{V} 
    = \bigg\{ v \in C([0,T) \times O, B)\colon \limsup \limits_{r \to \infty} \bigg[ \sup \limits_{t \in [0,T)\setminus K_r} \sup \limits_{x \in O \setminus O_r} 
    \bigg( \frac{ \lvert \lvert \lvert v(t,x) \rvert \rvert \rvert }{V(t,x)}  \, \sqrt{T-t} \bigg) \bigg] = 0 \bigg\},
  \end{equation}
  and let
  $\lvert \lvert \cdot \rvert \rvert_{\lambda}
  \colon \mathcal{V} \to [0, \infty)$,
  $\lambda \in \R$,
  satisfy for all 
  $\lambda \in \R$, $w \in \mathcal{V}$
  that
  \begin{equation}
    \lvert \lvert w \rvert \rvert_{\lambda} 
    = \sup \limits_{t \in [0,T)} \sup \limits_{x \in O} \left( \frac{e^{\lambda t} \lvert \lvert \lvert w(t,x) \rvert \rvert \rvert} {V(t,x)} \, \sqrt{T-t} \right).
  \end{equation}
  Then
  for all
  $ \lambda \in \R$
  it holds that
  $(\mathcal{V}, \lvert \lvert \cdot \rvert \rvert_{\lambda})$
  is an $\R$-Banach space.
\end{lemma}

\begin{proof}[Proof of Lemma \ref{lem:banach_space2}]
  Throughout this proof
  let $\mathcal{W}_1, \mathcal{W}_2$ 
  satisfy that
  \begin{equation}
    \mathcal{W}_1 
    = \left\{ v \in C([0,T) \times O, B)\colon \sup \nolimits_{t \in [0,T)} 
    \sup \nolimits_{x \in O} \, \left( \lvert \lvert \lvert v(t,x) \rvert \rvert \rvert  \, \sqrt{T-t} \right) < \infty \right\}
  \end{equation}
  and 
  \begin{equation}
     \mathcal{W}_2 
     = \Big\{ v \in C([0,T) \times O, B)\colon 
     \limsup \limits_{r \to \infty} \big[ \sup \limits_{t \in [0,T)\setminus K_r} 
     \sup \limits_{x \in O \setminus O_r} \big( \lvert \lvert \lvert v(t,x) \rvert \rvert \rvert  \, \sqrt{T-t} \big) \big] = 0  \Big\},
  \end{equation}
  and let
  $\lvert \lvert \cdot \rvert \rvert_{\mathcal{W}_i}\colon \mathcal{W}_i \to [0, \infty)$,
  $i \in \{ 1, 2\}$,
  satisfy for all 
  $i \in \{ 1,2 \}$, $w \in \mathcal{W}_i$
  that
  \begin{equation}
     \lvert \lvert w \rvert \rvert_{\mathcal{W}_i} 
     = \sup \nolimits_{t \in [0,T)} \sup \nolimits_{x \in O} \,  \left( \lvert \lvert \lvert w(t,x) \rvert \rvert \rvert \, \sqrt{T-t} \right). 
\end{equation}
  First we claim that
  $\lvert \lvert \cdot \rvert \rvert_{\mathcal{W}_2}\colon \mathcal{W}_2 \to [0, \infty)$
  is well-defined.
  For this, let 
  $w\in\mathcal{W}_2$.
  Observe that the assumption
  that $w\in C([0,T)\times O,B)$ 
  and the fact that
  for all $r\in(0,\infty)$
  it holds that $K_r$ and $O_r$
  are compact
  ensures that
  for all $r\in(0,\infty)$
  it holds that
  \begin{equation}
    \sup\nolimits_{t\in K_r}
    \sup\nolimits_{x\in O_r}
    (\lvert\lvert\lvert w(t,x)
    \rvert\rvert\rvert\sqrt{T-t})
    <\infty.
\end{equation}   
  Combining this with the fact
  that
  $\limsup_{r \to \infty}[ \sup_{t \in [0,T)\setminus K_r} \sup_{x \in O \setminus O_r} ( \lvert \lvert \lvert w(t,x) \rvert \rvert \rvert  \, \sqrt{T-t})] = 0 $
  shows that
  \begin{equation}
    \sup\nolimits_{t\in [0,T)}
    \sup\nolimits_{x\in O}
    (\lvert\lvert\lvert w(t,x)
    \rvert\rvert\rvert\sqrt{T-t})
    <\infty.
  \end{equation}
  This demonstrates that
  $\norm{\cdot}_{\mathcal{W}_2}$
  is well-defined.
  In particular, this
  demonstrates that
  $\mathcal{W}_2 \subseteq \mathcal{W}_1$. 
  Next note that 
  $(\mathcal{W}_1, \lvert \lvert \cdot \rvert \rvert_{\mathcal{W}_1})$
  is a normed $\R$-vector space.
  Let $(w_n)_{n\in\N}\subseteq
  \mathcal{W}_1$ be a
  Cauchy sequence.
  Observe that for all 
  $\varepsilon\in (0,\infty)$
  there exists $N_\varepsilon
  \in\N$ which satisfies for all
  $m,n \in\N\cap [N_\varepsilon,\infty)$
  that
  $\lvert \lvert w_n - w_m \rvert \rvert_{\mathcal{W}_1} 
     < \varepsilon$.
  This implies that for all
  $\varepsilon\in (0,\infty)$,
  $k,n \in\N\cap [N_\varepsilon,\infty)$,
  $t\in[0,T)$, $x\in\R^d$
  it holds that  
  \begin{equation}\label{eq:w_cauchy}
     \lvert \lvert \lvert w_n(t,x) - w_k(t,x) \rvert \rvert \rvert \, \sqrt{T-t}
     \leq \lvert \lvert w_n - w_k \rvert \rvert_{\mathcal{W}_1} 
     < \varepsilon.
  \end{equation}
  This ensures
  for all $t \in [0,T)$, $x \in O$
  that
  $(w_n(t,x))_{n \in \mathbb{N}}$
  is a Cauchy sequence
  in $B$.
  The fact that 
  $B$ is complete hence
  demonstrates that for all
  $t \in [0,T)$, $x \in O$
  there exists $w(t,x) \in B$ 
  which satisfies
  $w(t,x) =\lim\nolimits_{n\to\infty} w_n(t,x)$.
  Moreover, note that \eqref{eq:w_cauchy}
  demonstrates that
  for all $\varepsilon\in(0,\infty)$, 
  $n\in\N\cap[N_\varepsilon,\infty)$
  it holds that
  \begin{equation}\label{eq:w_unif_conv}
    \begin{split}
    &\sup_{t\in[0,T)}\sup_{x\in O}
    \big(\lvert \lvert \lvert w(t,x) - w_n(t,x) \rvert \rvert \rvert \, \sqrt{T-t}\big)
    = 
    \sup_{t\in[0,T)}\sup_{x\in O}
    \big(\lvert \lvert \lvert \lim \limits_{m \to \infty} w_m(t,x)-w_n(t,x) 
    \rvert \rvert \rvert \, \sqrt{T-t}\big)\\
    &\leq \limsup \limits_{m \to \infty} 
    \sup_{t\in[0,T)}\sup_{x\in O}
    \big(\lvert \lvert \lvert w_m(t,x) - w_n(t,x) \rvert \rvert \rvert \, \sqrt{T-t}\big)
    < \varepsilon.
\end{split}
\end{equation}
  This and 
  the fact that 
  $w_{N_1}\in\mathcal{W}_1$
  demonstrate that
  \begin{equation}
  \begin{split}
    &\sup_{t\in[0,T)}\sup_{x\in O}
    (\lvert\lvert\lvert w(t,x)\rvert\rvert\rvert \sqrt{T-t})\\
    &\leq \sup_{t\in[0,T)}\sup_{x\in O}
    (\lvert\lvert\lvert w(t,x)-w_{N_1}(t,x)\rvert\rvert\rvert \sqrt{T-t})
    +\sup_{t\in[0,T)}\sup_{x\in O}
    (\lvert\lvert\lvert w_{N_1}(t,x)\rvert\rvert\rvert \sqrt{T-t})\\
    &\leq 1 
    + \norm{w_{N_1}}_{\mathcal{W}_1}
    <\infty.
  \end{split}
  \end{equation}
  Combining this with the fact that
  the uniform limit of continuous functions 
  is continuous
  shows that $w \in \mathcal{W}_1$. 
  This proves that $\mathcal{W}_1$
  is a Banach space.
  Next we prove that 
  $\mathcal{W}_2$
  is a closed subset of
  $(\mathcal{W}_1, \lvert \lvert \cdot \rvert \rvert_{\mathcal{W}_1})$.
  For this, let $(w_n)_{n\in\N}
  \subseteq\mathcal{W}_2$
  and $w\in\mathcal{W}_1$
  satisfy that
  \begin{equation}
  \label{eq:w2_unifconv}
    \lim\nolimits_{n\to\infty}
  \norm{w_n-w}_{\mathcal{W}_1} 
    =0.
  \end{equation}
  For every 
  $\varepsilon\in(0,\infty)$
  let $n_\varepsilon\in\N$ and
  $r_\varepsilon\in(0,\infty)$
  satisfy 
  $\norm{w-w_{n_\varepsilon}}_{\mathcal{W}_1}
  <\frac{\varepsilon}{2}$
  and
  $\sup_{t\in[0,T)\setminus K_{r_\varepsilon}}
  \sup_{x\in O\setminus O_{r_\varepsilon}}$
  $(\lvert\lvert\lvert
  w_{n_\varepsilon}(t,x)\rvert\rvert\rvert \sqrt{T-t})
  <\frac{\varepsilon}{2}$.
  Hence, we obtain that
  \begin{equation}
  \begin{split}
    &\limsup_{r\to\infty}[\sup_{t\in[0,T)\setminus K_r}\sup_{x\in O\setminus O_r}\lvert\lvert\lvert w(t,x)\rvert\rvert\rvert \sqrt{T-t})]
    \leq \limsup_{\varepsilon\to 0}[\sup_{t\in[0,T)\setminus K_{r_\varepsilon}}
  \sup_{x\in O\setminus O_{r_\varepsilon}} (\lvert\lvert\lvert
  w(t,x)\rvert\rvert\rvert \sqrt{T-t})]\\
  &\leq \limsup_{\varepsilon \to 0}
  \varepsilon
  =0.
    \end{split}
\end{equation}    
  This implies that
  $w\in\mathcal{W}_2$.
  Combining this with the fact
  that $\mathcal{W}_2\subseteq
  \mathcal{W}_1$
  demonstrates that
  $\mathcal{W}_2$
  is a complete subset of
  $(\mathcal{W}_1, \lvert \lvert \cdot \rvert \rvert_{\mathcal{W}_1})$.
  Next note that
  for all 
  $\lambda \in \R$
  it holds that
  $(\mathcal{V}, \lvert \lvert \cdot \rvert \rvert_{\lambda})$
  is an 
  $\R$-vector space.
  We claim that
  $(\mathcal{V}, \lvert \lvert \cdot \rvert \rvert_{0})$
  is an 
  $\R$-Banach space.
  For this, let
  $v_n \in \mathcal{V}$, $n \in \N$,
  satisfy
  \begin{equation}
    \limsup\nolimits_{n \to \infty} 
    [ \sup\nolimits_{k,l\in \N\cap[n,\infty)} \lvert \lvert v_k - v_l \rvert \rvert_0 = 0 ].
  \end{equation}
  This implies that 
  $\frac{v_n}{V}\colon[0,T) \times O \to \R$,
  $n \in \N$,
  is a Cauchy sequence in 
  $(\mathcal{W}_2, \lvert \lvert \cdot \rvert \rvert_{\mathcal{W}_2})$.
  The fact that 
  $(\mathcal{W}_2, \lvert \lvert \cdot \rvert \rvert_{\mathcal{W}_2})$
  is a Banach space
  therefore implies that 
  there exists a unique
  $\phi \in \mathcal{W}_2$
  which satisfies
  \begin{equation}
    \limsup_{n \to \infty}
    \normmm{\frac{v_n}{V}-\phi}_{\mathcal{W}_2}
    = 0.
  \end{equation}
  Observe that 
  $\phi V = \left( [0,T) \times O \ni (t,x) \mapsto \phi(t,x) V(t,x) \in \R^m \right) \in \mathcal{V}$
  and
  \begin{equation}
    \limsup \limits_{n \to \infty} \lvert \lvert v_n - \phi V \rvert \rvert_0
    =\limsup_{n\to\infty}
    \normmm{\frac{v_n}{V}-\phi}_{\mathcal{W}_2}
    = 0. 
  \end{equation}
  This proves that 
  $(\mathcal{V}, \lvert \lvert \cdot \rvert \rvert_{0})$
  is an 
  $\R$-Banach space.
  The fact that for all
  $\Lambda \in \R$, $\lambda \in [\Lambda, \infty)$, $v \in \mathcal{V}$,
  it holds that
  \begin{equation}
    \lvert \lvert v \rvert \rvert_{\Lambda} 
    \leq \lvert \lvert v \rvert \rvert_{\lambda}
    \leq e^{(\lambda - \Lambda)T} \lvert \lvert v \rvert \rvert_{\Lambda} 
\end{equation}
  hence shows that for all
  $\lambda \in \R$
  it holds that
  $(\mathcal{V}, \lvert \lvert \cdot \rvert \rvert_{\lambda})$
  is an 
  $\R$-Banach space.
  The proof of
  Lemma~\ref{lem:banach_space2}
  is thus complete.
\end{proof}



The next theorem is our main
result of this section and shows 
the existence and uniqueness of
an SFPE solution under the assumption
of a Lipschitz continuous nonlinerarity. 
Theorem~\ref{thm:ex_cont_solution} is an extension of \cite[Theorem 2.9]{BGHJ2019}
to the case of 
gradient-dependent nonlinearities.

\begin{theorem}\label{thm:ex_cont_solution}
  Assume Setting~\ref{setting:sfpe},
  let $L\in (0, \infty)$,
  assume for all 
  $\varepsilon \in (0, \infty)$, 
  $t \in [0,T)$, $s\in(t,T]$, 
  $x\in O$ that
  \begin{equation}
    \limsup\nolimits_{[0,s)\times O \ni (u,y)\to (t,x)} 
    \Big[ \mathbb{P} \big( \lvert \lvert X^{y}_{u, s} 
    - X^{x}_{t, s } \rvert \rvert 
    > \varepsilon \big) 
     +\E \big[ \lvert \lvert \lvert Z^{y}_{u, s} 
    - Z^{x}_{t, s} \rvert \rvert \rvert \big] \Big] 
    = 0,
  \end{equation}
  let $f \in C([0,T) \times O \times \R^m, \R)$,
  $g \in C(O, \R)$ satisfy for all
  $t \in [0,T)$, $x \in O$, $v, w \in \R^m$
  that
  $\lvert f(t,x,v) - f(t,x,w) \rvert 
  \leq L\lvert\lvert\lvert  v-w\rvert\rvert \rvert $,
  and assume that
  $\inf_{r \in (0, \infty)} [ \sup_{t \in [0,T)\setminus K_r} $
  $\sup_{x \in O \setminus O_r} 
  (\frac{\lvert  g(x) \rvert}{V(T,x)} 
   +\frac{\lvert f(t,x,0) \rvert}{V(t,x)}
  \sqrt{T-t})] 
  = 0$,
  $\inf_{t\in[0,T]}\inf_{x\in O}
  V(t,x)>0$,  
  $\sup_{r \in (0, \infty)} [ \inf_{t \in [0,T)\setminus K_r} \inf_{x \in O \setminus O_r}$ 
  $V(t,x) ] 
  = \infty$,
  and for all
  $u\in (0,\infty)$ that
  \begin{equation}
  \label{eq:bfpt_function}
    \limsup_{r \to\infty} \Bigg[ \sup_{t \in [0,T)\setminus K_r} \sup_{x \in O \setminus O_r} \Bigg(
    \frac{\E\Big[\int_t^T \mathbbm{1}_{O_u}(X^x_{t,s})\,
    \lvert\lvert\lvert Z^x_{t,s}\rvert\rvert\rvert
    (\d s + \delta_T(\d s))\Big]
    \sqrt{T-t}}{V(t,x)}\Bigg) \Bigg]
    =0.
  \end{equation}  
  Then there exists a unique 
  $v \in C([0,T) \times O, \R^m)$
such that
\begin{enumerate}[label=(\roman*)]
\item\label{it:ex_cont_sol1}
it holds that
\begin{equation}\label{eq:ex_cont_solution1}
\limsup \limits_{r \to \infty} \left[ \sup \limits_{t \in [0,T)\setminus K_r} \sup \limits_{x \in O \setminus O_r} \left( \frac{\lvert \lvert \lvert v(t,x) \rvert \rvert \rvert}{V(t,x)}  \, \sqrt{T-t} \right) \right] = 0,
\end{equation}
\item\label{it:ex_cont_sol2}
 it holds that
  \begin{equation}\label{eq:ex_cont_sol2}
     [0,T) \times O \ni (t,x) \mapsto 
     \E \left[  g(X^x_{t,T}) Z^x_{t,T}  
     + \int_t^T  f(r, X^x_{t,r}, v(r, X^x_{t,r})) Z^x_{t,r}  \,\d r \right] \in \R^m
  \end{equation}
is well-defined and continuous,
and
\item\label{it:ex_cont_sol3}
it holds for all
$t \in [0,T)$, $x \in O$
that
\begin{equation}\label{eq:ex_cont_solution2}
v(t,x) = \E \left[ g(X^x_{t,T})Z^x_{t,T} + \int_t^T f(r, X^x_{t,r}, v(r, X^x_{t,r}))Z^x_{t,r} \,\d r \right].
\end{equation}
\end{enumerate}
\end{theorem}

\begin{proof}[Proof of Theorem \ref{thm:ex_cont_solution}]
Let $\mathcal{V}$ satisfy
\begin{equation}
\mathcal{V} 
= \bigg\{ v \in C([0,T) \times O, \R^m)\colon
\limsup \limits_{r \to \infty} \bigg[ \sup \limits_{t \in [0,T)\setminus K_r} \sup \limits_{x \in O \setminus O_r} \left( \frac{ \lvert \lvert \lvert v(t,x) \rvert \rvert \rvert }{V(t,x)}   \, \sqrt{T-t} \right) \bigg] = 0 \bigg\},
\end{equation}
and let
$\lvert \lvert \cdot \rvert \rvert_{\lambda}\colon \mathcal{V} \to [0, \infty)$,
$\lambda \in \R$,
satisfy for all 
$\lambda \in \R$, $w \in \mathcal{V}$
that
\begin{equation}
\lvert \lvert w \rvert \rvert_{\lambda} 
= \sup \limits_{t \in [0,T)} \sup \limits_{x \in O} \left( \frac{e^{\lambda t} \lvert \lvert \lvert w(t,x) \rvert \rvert \rvert} {V(t,x)}  \, \sqrt{T-t} \right).
\end{equation}
Observe that 
Lemma \ref{lem:banach_space2} 
proves that for all
$\lambda \in \R$
it holds that
$(\mathcal{V}, \lvert \lvert \cdot \lvert \lvert_{\lambda})$
is an 
$\R$-Banach space.
  Note that for all 
  $w\in\mathcal{V}$
  it holds that
    $[0,T) \times O \ni (t,x) 
    \mapsto f(t,x,w(t,x)) \in \R$
  is a continuous function
  which satisfies that for all
  $t \in [0,T)$, $x \in O$
  it holds that
  \begin{equation}
  \begin{split}
     \lvert f(t,x,w(t,x)) \rvert 
     \leq \lvert f(t,x,0) \rvert 
     + \lvert   f(t,x,w(t,x)) - f(t,x,0) \rvert
     \leq \lvert f(t,x,0) \rvert 
     + L \lvert \lvert \lvert w(t,x) \rvert \rvert \rvert.
  \end{split}
  \end{equation}
  The assumption 
  that 
  $\inf_{r \in (0, \infty)} [ \sup_{t \in [0,T)\setminus K_r} \sup_{x \in O \setminus O_r} (  \frac{\lvert f(t,x,0) \rvert}{V(t,x)} \sqrt{T-t}) ]
  =0 $
  hence ensures that
  for all $w\in\mathcal{V}$
  it holds that
 \begin{equation}\label{eq:f_bd}
  \begin{split}
    &\inf \limits_{r \in (0, \infty)} \left[ \sup \limits_{t \in [0,T)\setminus K_r} \sup \limits_{x \in O \setminus O_r} \left(  \frac{\lvert f(t,x,w(t,x)) \rvert }{V(t,x)}
    \sqrt{T-t} \right) \right] \\
    &\leq  \inf \limits_{r \in (0, \infty)}   
    \left[ \sup \limits_{t \in [0,T)\setminus K_r} 
    \sup \limits_{x \in O \setminus O_r} 
    \left(\frac{\lvert f(t,x,0) \rvert}{V(t,x)}
    \sqrt{T-t} 
    + L \frac{\lvert \lvert \lvert w(t,x) \rvert \rvert \rvert}{V(t,x)}\sqrt{T-t} \right) \right]
    = 0.
\end{split}
\end{equation}
  Lemma \ref{lem:integrability}
  and item~\ref{it:app_res1} 
  of Lemma \ref{lem:approximation_result1}
  hence demonstrate that
  for all $w\in\mathcal{V}$
  it holds that 
  \begin{equation}\label{eq:phi_w}
     [0,T) \times O \ni (t,x) \mapsto 
     \E \bigg[  g(X^x_{t,T}) Z^x_{t,T}  
     + \int_t^T  f(r, X^x_{t,r}, w(r, X^x_{t,r})) Z^x_{t,r}  \,\d r \bigg] \in \R^m
  \end{equation}
  is well-defined.
  Furthermore, note that item~\ref{it:cont2_1}
  and \ref{it:cont2_2} of
  Lemma \ref{lem:continuity2} 
  (applied for every $w\in\mathcal{V}$ 
  with  
  $h \curvearrowleft ([0,T) \times O \ni (t, x) \mapsto f(t,x,w(t,x)) \in \R)$ 
  in the notation of Lemma \ref{lem:continuity2})
  and \eqref{eq:f_bd}
  prove that for every $w\in \mathcal{V}$
  it holds that the function in 
  \eqref{eq:phi_w}
  is in $\mathcal{V}$.
  This shows that
  there exists 
  $\Phi\colon \mathcal{V} \to \mathcal{V}$
  which satisfies for all
  $t \in [0,T)$, $x \in O$, $w \in \mathcal{V}$
  that
  \begin{equation}
    (\Phi(w))(t,x) 
    = \E \left[ g(X^x_{t,T})Z^x_{t,T} + \int_t^T f(r, X^x_{t,r}, w(r, X^x_{t,r})) Z^x_{t,r} \,\d r \right].
  \end{equation}
  Furthermore, note that 
  Lemma~\ref{lem:contraction}
  demonstrates that 
  for all
  $\lambda \in (0, \infty)$, 
  $w,\tilde{w} \in \mathcal{V}$
  it holds that
  \begin{equation}
  \begin{split}
    &\lvert \lvert \Phi(w) - \Phi(\tilde{w})  \rvert \rvert_{\lambda}\\ 
    &\leq  \sup_{t\in[0,T)}
    \sup_{x\in O}\Bigg(
    \frac{e^{\lambda t}\E \Big[ \int_t^T \left| f(r, X^x_{t,r}, v(r, X^x_{t,r}))
    -f(r, X^x_{t,r}, w(r, X^x_{t,r})) \right| \, \lvert \lvert \lvert Z^x_{t,r} \rvert \rvert \rvert \,\d r \Big]}{V(t,x)}
    \sqrt{T-t}\Bigg)\\
    &\leq cL\sqrt{\frac{\pi^3}{4\lambda}}\, 
    \lvert \lvert w-\tilde{w} \rvert \rvert_{\lambda}.
  \end{split}
  \end{equation}
  This implies that for all
  $\lambda\in [c^2L^2\pi^3, \infty)$,
  $w,\tilde{w} \in \mathcal{V}$
  it holds that
  \begin{equation}
    \lvert \lvert \Phi(w) - \Phi(\tilde{w})    \rvert \rvert_{\lambda} 
    \leq \frac{1}{2} \lvert \lvert w-\tilde{w}   \rvert \rvert_{\lambda}.
  \end{equation}
  Banach's fixed point theorem
  therefore shows that 
  there exists a unique
  $v \in \mathcal{V}$
  which satisfies 
  $\Phi(v)= v$.
  The proof of 
  Theorem~\ref{thm:ex_cont_solution}
  is thus complete.
\end{proof}

\section{SFPEs associated with stochastic differential equations (SDEs)}
\label{sec:sfpes_Z}

The goal of this section is to apply
the abstract existence and uniqueness
result in Theorem~\ref{thm:ex_cont_solution}
to the case where $X$ is an SDE
solution and $Z$ is the stochastic
process in \eqref{eq:Z_1} arising from the Bismut-Elworthy-Li formula.
Section~\ref{subsec:Z_bounds}-\ref{subsec:Z_convergence} provide several
boundedness and convergence properties
of SDE solutions which will be needed
to prove our main result, Theorem~\ref{thm:ex_cont_sol_Z}.
Throughout this section we
frequently use the following setting.

\begin{setting}
\label{setting:sfpe2}
  Let $d \in \N$, 
  $\alpha, c, T \in (0, \infty)$,
  let
  $\langle \cdot, \cdot \rangle \colon \R^d \times \R^d \to \R$
  and $\norm{\cdot}\colon\R^d\to[0,\infty)$
  satisfy for all
  $x=(x_1,\ldots,x_d)$, $y=(y_1,\ldots,y_d)\in\R^d$ that
  $\langle x, y\rangle
  =\sum_{i=1}^d x_i y_i$
  and $\norm{x}=\textstyle (\sum_{i=1}^d \abs{x_i}^2)^{1/2}$,
  let $\norm{\cdot}_F\colon\R^{d\times d}\to[0,\infty)$ satisfy for all
  $A=(A_{ij})_{i, j\in\{1,\ldots, d\}}\in\R^{d\times d}$ that
  $\norm{A}_F = (\sum_{i=1}^d\sum_{j=1}^d \abs{A_{ij}}^2)^{\frac{1}{2}}$,
  let $O\subseteq \R^d$ be an
  open set,
  let $(\Omega, \mathcal{F}, \mathbb{P}, (\mathbb{F}_s)_{s \in [0,T]})$ 
  be a filtered probability space
  satisfying the usual conditions,
  let $W \colon [0,T] \times \Omega \to \R^d$ 
  be a standard $(\mathbb{F}_s)_{s \in [0,T]}$-Brownian motion,
  let $\mu \in  C^{0,1}([0,T] \times O, \R^d)$, 
  $\sigma \in C^{0,1} ([0,T] \times O, \R^{d \times d})$
  satisfy
  for all 
  $s\in[0,T]$, $x, y \in O$, $v\in\R^d$
  that
  \begin{equation}
  \label{eq:mu_sigma_globmon1}
  \max\Big\{\langle x-y,\mu(s,x)-\mu(s,y)\rangle, 
  \frac{1}{2}\norm{\sigma(s,x)-\sigma(s,y)}_F^2\Big\}
  \leq \frac{c}{2}\norm{x-y}^2
  \end{equation}
  and $v^* \sigma(s,x) (\sigma(s,x))^* v \geq \alpha \norm{v}^2$,
  for every  $t\in [0,T]$,
  $x \in O$ let
  $X^x_t = ((X^{x,1}_{t,s},\ldots, X^{x,d}_{t,s}))_{s \in [t,T]} \colon [t,T] \times \Omega \to O$   
  be an 
  $(\mathbb{F}_s)_{s \in [t,T]}$-adapted 
  stochastic process with continuous
  sample paths satisfying that for all 
  $s \in [t,T]$ it holds a.s.\! that
  \begin{equation}
  \label{eq:X_ito_process1}
    X^x_{t,s} = x + \int_t^s \mu(r, X^x_{t,r}) \,\d r 
    + \int_t^s \sigma(r, X^x_{t,r}) \,\d W_r,
  \end{equation}
  and assume for all
  $t\in [0,T]$,
  $\omega \in \Omega$
  that
  $\left([t,T] \times O \ni (s,x) 
  \mapsto X^x_{t,s}(\omega) \in O \right) \in C^{0,1}([t,T] \times O, O)$.

\end{setting}

\subsection{Moment estimates}
\label{subsec:Z_bounds}

The next lemma establishes several
boundedness results for SDE solutions
and the stochastic process in \eqref{eq:Z_1} coming from the 
Bismut-Elworthy-Li formula.

\begin{lemma}\label{lem:Z_L_2_bd}
  Let $d \in \N$, 
  $\alpha, c, T \in (0, \infty)$,
  let $t\in [0,T]$,
  let
  $\langle \cdot, \cdot \rangle \colon \R^d \times \R^d \to \R$
  and $\norm{\cdot}\colon\R^d\to[0,\infty)$
  satisfy for all
  $x=(x_1,\ldots,x_d)$, $y=(y_1,\ldots,y_d)\in\R^d$ that
  $\langle x, y\rangle
  =\sum_{i=1}^d x_i y_i$
  and $\norm{x}=\textstyle (\sum_{i=1}^d \abs{x_i}^2)^{1/2}$,
  let $\norm{\cdot}_F\colon\R^{d\times d}\to[0,\infty)$ satisfy for all
  $A=(A_{ij})_{i, j\in\{1,\ldots, d\}}\in\R^{d\times d}$ that
  $\norm{A}_F = (\sum_{i=1}^d\sum_{j=1}^d \abs{A_{ij}}^2)^{\frac{1}{2}}$,
  let $O\subseteq \R^d$ be an
  open set,
  let $(\Omega, \mathcal{F}, \mathbb{P}, (\mathbb{F}_s)_{s \in [0,T]})$ 
  be a filtered probability space
  satisfying the usual conditions,
  let $W \colon [0,T] \times \Omega \to \R^d$ 
  be a standard $(\mathbb{F}_s)_{s \in [0,T]}$-Brownian motion,
  let $\mu \in  C^{0,1}([0,T] \times O, \R^d)$, 
  $\sigma \in C^{0,1} ([0,T] \times O, \R^{d \times d})$
  satisfy
  for all 
  $s\in[t,T]$, $x, y \in O$, $v\in\R^d$
  that
  \begin{equation}
  \label{eq:mu_sigma_globmon2}
  \max\Big\{\langle x-y,\mu(s,x)-\mu(s,y)\rangle, 
  \frac{1}{2}\norm{\sigma(s,x)-\sigma(s,y)}_F^2\Big\}
  \leq \frac{c}{2}\norm{x-y}^2
  \end{equation}
  and $v^* \sigma(s,x) (\sigma(s,x))^* v \geq \alpha \norm{v}^2$,
  let $\mathfrak{m}\in [0,\infty)$
  satisfy $\mathfrak{m}
  =\max_{s\in [0,T]}[\frac{1}{2}
  \norm{\mu(s,0)}^2
  +\norm{\sigma(s,0)}_F^2]$,
  for every  
  $x \in O$ let
  $X^x = ((X^{x}_{s,1},\ldots, X^x_{s,d}))_{s \in [t,T]} \colon [t,T] \times \Omega \to O$   
  be an 
  $(\mathbb{F}_s)_{s \in [t,T]}$-adapted 
  stochastic process with continuous
  sample paths satisfying that for all 
  $s \in [t,T]$ it holds a.s.\! that
  \begin{equation}
  \label{eq:X_ito_process2}
    X^x_{s} = x + \int_t^s \mu(r, X^x_{r}) \,\d r 
    + \int_t^s \sigma(r, X^x_{r}) \,\d W_r,
  \end{equation}
  assume for all
  $\omega \in \Omega$
  that
  $\left([t,T] \times O \ni (s,x) 
  \mapsto X^x_{s}(\omega) \in O \right) \in C^{0,1}([t,T] \times O, O)$,
  and for every 
  $x \in O$ let
  $Z^x = (Z^x_{s})_{s \in (t,T]} 
  \colon (t,T] \times \Omega \to \R^d$ 
  be an $(\mathbb{F}_s)_{s \in (t,T]}$-adapted stochastic process 
  with continuous sample paths
  satisfying that 
  for all $s \in (t,T]$ 
  it holds a.s.\!   that
  \begin{equation}
  \label{eq:Z_1}
    Z^x_{s} = \frac{1}{s-t} \int_t^s
  (\sigma(r,   X^x_{r}))^{-1} \; \Big(\frac{\partial}{\partial x} X^x_{r}\Big) \,\d W_r.
  \end{equation} 
  Then
  \begin{enumerate}[label=(\roman*)]
  \item\label{it:X_L2_bd}
  for all $x\in O$, for every
  stopping time
  $\tau\colon\Omega\to[t,T]$
  it holds that
  \begin{equation}
    \Big(\E\big[\norm{X^x_\tau}^2\big]\Big)^{\frac{1}{2}}
    \leq \exp((2c+1)T)
    \Big(\norm{x}^2+
    \frac{\mathfrak{m}}{2c+1}\Big),
  \end{equation}
  \item\label{it:der_X_bd}
  for all $x\in O$,
  for every
  stopping time
  $\tau\colon\Omega\to[t,T]$
  it holds that
  \begin{equation}
  \label{eq:der_X_bd}
  \E\Big[ \normm{\frac{\partial}{\partial x} X^x_{\tau}}_F^2\Big]
  \leq d\exp\big(2c (T-t) \big),
  \end{equation}
  \item \label{it:Y_bd_st}
  for all $x\in O$ 
  it holds that
  \begin{equation}
  \E\bigg[\normmm{\int_t^\tau
  (\sigma(r,   X^x_{r}))^{-1} \; \Big(\frac{\partial}{\partial x} X^x_{r}\Big) \,\d W_r}^2\bigg]
  \leq \frac{dT}{\alpha}\exp(2cT),
  \end{equation}
  and 
  \item \label{it:Z_L2_bd}
  for all $s\in(t,T]$,
  $x \in O$
  it holds that
  \begin{equation}
  \E\Big[\norm{Z^x_{s}}^2\Big]
  \leq\frac{d}{\alpha(s-t)^2}
  \int_t^s  \exp(2(r-t)c) \,\d r.
  \end{equation}
  \end{enumerate}
\end{lemma}

\begin{proof}[Proof of Lemma~\ref{lem:Z_L_2_bd}]
  Throughout this proof let
  $\mathbf{e}_1,\mathbf{e}_2,\dots,\mathbf{e}_d\in\R^d$ 
  satisfy that $\mathbf{e}_1=(1,0,\dots,0)$, 
  $\mathbf{e}_2=(0,1,0,\dots,0)$, $\dots$,
  $\mathbf{e}_d=(0,\dots,0,1)$.
  First observe that
  \eqref{eq:mu_sigma_globmon2},
  the Cauchy-Schwarz inequality,
  and the fact that for all
  $a,b\in[0,\infty)$
  it holds that
  $2ab\leq a^2+b^2$ 
  imply that for all
  $x\in O$, $s\in[t,T]$
  it holds that
  \begin{equation}
  \begin{split}
    &\langle X^x_s, \mu(s,X^x_s)\rangle
    +\frac{1}{2}\norm{\sigma(s,X^x_s)}_F^2\\
    &\leq \langle X^x_s, \mu(s,X^x_s)
    -\mu(s,0) \rangle
    + \langle X^x_s, \mu(s,0) \rangle
    +\norm{\sigma(s,X^x_s)-\sigma(s,0)}_F^2
    +\norm{\sigma(s,0)}_F^2\\
    &\leq \frac{c}{2}\norm{X^x_s}^2
    + \norm{X^x_s} \norm{\mu(s,0)}
    +c\norm{X^x_s}^2
    +\norm{\sigma(s,0)}_F^2\\
    &\leq \frac{c}{2}\norm{X^x_s}^2
    + \frac{1}{2}\norm{X^x_s}^2
    +\frac{1}{2}\norm{\mu(s,0)}^2
    +c\norm{X^x_s}^2
    +\norm{\sigma(s,0)}_F^2
    \leq (2c+1)\norm{X^x_s}^2
    +\mathfrak{m}.
  \end{split}
\end{equation}    
  Combining this with
  \cite[Corollary 2.5 (i)]{HHM2019} 
  (applied for every $x\in\R^d$
  with $H\curvearrowleft \R^d$, $U\curvearrowleft \R^d$,
   $T\curvearrowleft T-t$, 
   $a\curvearrowleft ([0,T-t]\times \Omega
   \ni (s,\omega)\mapsto \mu(t+s,x)\in\R^d)$,
   $b\curvearrowleft ([0,T-t]\times \Omega
   \ni (s,\omega)\mapsto \sigma(t+s,x)\in\R^{d\times d})$,
     $X\curvearrowleft ([0,T-t]\times\Omega
  \ni(s,\omega)\mapsto X^x_{t+s}(\omega)\in O)$,
  $p \curvearrowleft 2$, 
  $\alpha \curvearrowleft ([0,T-t]\times\Omega \ni (s,\omega) \mapsto (2c+1) \in \R)$,
  $\beta \curvearrowleft ([0,T-t]\times\Omega \ni (s,\omega) \mapsto \sqrt{2\mathfrak{m}} \in \R)$,
  $q_1 \curvearrowleft 2$, 
  $q_2 \curvearrowleft \infty$
  in the notation of \cite[Corollary 2.5 (i)] {HHM2019})
  shows that for all
  $x\in O$, $\tau\colon\Omega\to [t,T]$
  stopping time
  it holds that
  \begin{equation}
  \begin{split}
    &\Big(\E\big[\norm{X^x_\tau}^2\big]\Big)^{\frac{1}{2}}
    \leq \exp((2c+1)T)
    \Big(\norm{x}^2+2\mathfrak{m}\int_0^T
    \exp(-2(2c+1)s)
    \, \d s \Big)\\
    &\leq \exp((2c+1)T)
    \Big(\norm{x}^2+
    \frac{2\mathfrak{m}}{4c+2}
    \big(1-\exp(-(4c+2)T)\big) \Big)\\
    &\leq \exp((2c+1)T)
    \Big(\norm{x}^2+
    \frac{\mathfrak{m}}{2c+1}\Big).
  \end{split}
  \end{equation}
  This establishes 
  item~\ref{it:X_L2_bd}.
  Next note that \eqref{eq:X_ito_process2}, 
  the Leibniz integral rule,
  the chain rule, and the fact that 
  $\mu\in C^{0,1}([0,T]\times O,\R^d)$
  and $\sigma\in C^{0,1}([0,T]\times O,\R^{d\times d})$ show that
  for all $s \in [t,T]$, $x \in O$, 
  $j\in\{1,\ldots,d\}$
  it holds a.s.\! that
  \begin{equation}\label{eq:X_der}
  \begin{split}
    \frac{\partial}{\partial x_j} X^{x}_{s} 
    &=  \mathbf{e}_j + \sum_{i=1}^d \Bigg[\int_t^s \Big(\frac{\partial \mu}{\partial x_i} \Big)(r, X^{x}_{r})
    \Big(\frac{\partial}{\partial x_j} X^{x}_{r,i}\Big)\,\d r 
    + \int_t^s \Big(\frac{\partial\sigma}{\partial x_i} \Big)(r, X^{x}_{r})
    \Big(\frac{\partial}{\partial x_j} X^{x}_{r,i}\Big)\,\d W_r\Bigg].
  \end{split}
  \end{equation}
  Moreover, observe that
  the chain rule, 
  \eqref{eq:mu_sigma_globmon2},
  and the assumption that 
  for all $\omega \in \Omega$
  it holds that
  $([t,T] \times O \ni (s,x) $ $\mapsto X^x_{s}(\omega) \in O) \in C^{0,1}([t,T] \times O, O)$
  ensure that for all $r\in[t,T]$, 
  $x\in O$, $j\in\{1,\ldots,d\}$
  it holds a.s.\! that
  \begin{equation}
  \begin{split}
  &\Big\langle \Big(\frac{\partial}{\partial x_j}X^x_{r}\Big), \sum_{i=1}^d \Big(\frac{\partial\mu}{\partial x_i}\Big)(r,X^x_{r})\Big(\frac{\partial}{\partial x_j}X^{x}_{r,i}\Big) \Big\rangle
  +\frac{1}{2}\normmm{\sum_{k=1}^d\Big(\frac{\partial\sigma}{\partial x_k}\Big)(r,X^x_{r})
  \Big(\frac{\partial}{\partial x_j}X^{x}_{r,k}\Big)}_F^2\\
  &=\lim_{(0,\infty)\ni h\to 0}\,\frac{1}{h^2}
  \bigg[\Big\langle X^{x+h\mathbf{e}_j}_{r}-X^x_r, \mu(r,X^{x+h\mathbf{e}_j}_{r})-\mu(r,X^{x}_{r})\Big\rangle
  +\frac{1}{2}\normm{\sigma(r,X^{x+h\mathbf{e}_j}_{r})-\sigma(r,X^x_r)}_F^2\bigg]\\
  &\leq \Big(\frac{c}{2}+\frac{c}{2}\Big)
  \lim_{(0,\infty)\ni h\to 0}\frac{1}{h^2}
  \bigg[ \normm{X^{x+h\mathbf{e}_j}_r-X^x_r}^2\bigg] 
  = c \normmm{\frac{\partial}{\partial x_j}X^x_{r}}^2.
  \end{split}
  \end{equation}
  This, \eqref{eq:X_der}, and e.g.,
  \cite[Corollary 2.5(i)]{HHM2019} 
  (applied for every $x\in\R^d$, $j\in\{1,\ldots,d\}$ 
  with $H\curvearrowleft \R^d$, $U\curvearrowleft \R^d$,
   $T\curvearrowleft T-t$, 
   $a\curvearrowleft ([0,T-t]\times \Omega
   \ni (s,\omega)\mapsto \mu(t+s,x)\in\R^d)$,
   $b\curvearrowleft ([0,T-t]\times \Omega
   \ni (s,\omega)\mapsto \sigma(t+s,x)\in\R^{d\times d})$,
  $X\curvearrowleft (\frac{\partial}{\partial x_j}X^x_{t+r})_{r\in[0,T-t]}$,
  $p \curvearrowleft 2$, $q_1 \curvearrowleft 2$, 
  $q_2 \curvearrowleft \infty$, 
  $\alpha \curvearrowleft ([0,T-t] \ni s \mapsto c \in \R)$, and
  $\beta \curvearrowleft 0$
  in the notation of \cite[Corollary 2.5(i)] {HHM2019})
  demonstrate that
  for all $j\in\{1,\ldots,d\}$,
  $x\in O$
  it holds that
  \begin{equation}
  \label{eq:der_j_X_bd}
  \Big(\E\Big[ \normm{\frac{\partial}{\partial x_j}X^x_{\tau}}^2\Big]\Big)^{\frac{1}{2}}
  \leq \E\big[\exp\big( c (\tau-t)\big)\big].
  \end{equation}
  Hence, we obtain that
  for all $x\in O$
  it holds that
  \begin{equation}
  \label{eq:der_X_bd2}
  \E\Big[ \normm{\frac{\partial}{\partial x} X^x_{\tau}}_F^2\Big]
  = \sum_{j=1}^d
  \E\Big[ \normm{\frac{\partial}{\partial x_j}X^x_{\tau}}^2\Big]
  \leq d \Big(\E\big[\exp\big( c (\tau-t)\big)\big]\Big)^{2}
  \leq d\exp\big(2c (T-t) \big).
  \end{equation}
  This establishes item~\ref{it:der_X_bd}.
  Next observe that the assumption that
  for all $r \in [t,T]$, $x \in O$, $y\in\R^d$ 
  it holds that
  $y^* \sigma(r,x) (\sigma(r,x))^* y  
  \geq \alpha \lvert \lvert y \rvert \rvert^2$ 
  ensures that for all 
  $r \in [t,T]$,
  $x \in O$, $y\in\R^d$
  it holds that
  $\sigma(r,X^x_r)$ is an invertible matrix
  and 
  \begin{equation}
  \norm{y}^2 = y^*  \left( \sigma(r, X^x_{r}) \right)^{-1} \sigma(r, X^x_{r}) ( \sigma(r, X^x_{r}))^* \left( (\sigma(r, X^x_{r}))^{-1} \right)^* y
  \geq \alpha \Big| \Big| \left( (\sigma(r, X^x_{r}))^{-1} \right)^* y \Big|\Big|^2.
  \end{equation}
  This implies for all  
  $r \in [t,T]$,
  $x \in O$
  that
  \begin{equation}\label{eq:sigma_inv_bd}
  \lvert \lvert  (\sigma(r, X^x_{r}))^{-1}   \rvert \rvert^2_{L(\R^d)}
  =\sup \limits_{y \in \R^d \backslash \{ 0 \}} \frac{\lvert \lvert \left( (\sigma(r, X^x_{r}))^{-1} \right)^* y \rvert \rvert^2}{\lvert \lvert y \rvert \rvert^2} 
  \leq \frac{1}{\alpha}.
  \end{equation}
  The Burkholder-Davis-Gundy inequality,
  Fubini's theorem, 
  and \eqref{eq:der_X_bd2}
  therefore prove 
  that for all $x \in O$,
  $\tau\colon\Omega\to [t,T]$
  stopping time
  it holds that
  \begin{equation}
  \begin{split}
    &\E \left[ \normmm{\int_t^\tau (\sigma(r, X^x_{r}))^{-1} \Big(\frac{\partial}{\partial x} X^x_{r}\Big)  \,\d W_r}^2 \right] 
  \leq  \E \left[ \int_t^\tau
  \normmm{(\sigma(r, X^x_{r}))^{-1} \Big(\frac{\partial}{\partial x} X^x_{r}\Big)}^2_F \,\d r  \right] \\
  &\leq   \E \left[ \int_t^\tau   \normm{(\sigma(r, X^x_{r}))^{-1}}^2_{L(\R^d)} \normmm{\frac{\partial}{\partial x} X^x_{r}}^2_{F} \,\d r  \right] 
  \leq \frac{1}{\alpha} \E \left[ \int_t^T
  \normmm{\frac{\partial}{\partial x} X^x_{r}}_F^2
\, \d r \right]\\
  &= \frac{1}{\alpha} \int_t^T 
  \E \left[ \normmm{\frac{\partial}{\partial x} X^x_{r}}_F^2\right]\, \d r 
  \leq \frac{d}{\alpha} \int_t^T \exp(2cT) \,\d r
  \leq \frac{dT}{\alpha}\exp(2cT).
  \end{split}
  \end{equation}
  This establishes 
  item~\ref{it:Y_bd_st}.  
  Next note that combining
  the Burkholder-Davis-Gundy inequality,
  Fubini's theorem, 
  \eqref{eq:der_j_X_bd},
  and \eqref{eq:sigma_inv_bd}
  demonstrates that for all
  $s \in (t,T]$, $x \in O$
  it holds that
  \begin{equation}\label{eq:Z_bounded}
  \begin{split}
  &(s-t)^2 \E \left[ \norm{Z^x_{s}}^2 \right]
  = \E \left[ \normmm{\int_t^s (\sigma(r, X^x_{r}))^{-1} \Big(\frac{\partial}{\partial x} X^x_{r}\Big)  \,\d W_r}^2 \right] \\
  &\leq  \E \left[ \int_t^s\normmm{(\sigma(r, X^x_{r}))^{-1} \Big(\frac{\partial}{\partial x} X^x_{r}\Big)}^2_F \,\d r  \right] 
  \leq   \E \left[ \int_t^s   \normm{(\sigma(r, X^x_{r}))^{-1}}^2_{L(\R^d)} \normmm{\frac{\partial}{\partial x} X^x_{r}}^2_{F} \,\d r  \right] \\
  &\leq \frac{1}{\alpha} \E \left[ \int_t^s 
  \normmm{\frac{\partial}{\partial x} X^x_{r}}_F^2
\, \d r \right]
  = \frac{1}{\alpha} \int_t^s 
  \E \left[ \normmm{\frac{\partial}{\partial x} X^x_{r}}_F^2\right]\, \d r 
  \leq \frac{d}{\alpha} \int_t^s \exp(2(r-t)c) \,\d r.\\
  \end{split}
  \end{equation}
  This establishes item~\ref{it:Z_L2_bd}. 
  The proof of Lemma~\ref{lem:Z_L_2_bd}
  is thus complete.
\end{proof}
  
\subsection{Continuity in the starting point and starting time}
\label{subsec:Z_convergence}

In this section we provide 
two convergence results for stochastic
processes in the starting point. 
The following lemma demonstrates 
convergence in probability 
of SDE solutions in the starting point.
Lemma~\ref{lem:X_cont_in_prob} 
is a generalization of 
\cite[Lemma 3.7]{BGHJ2019}.
 
\begin{lemma}\label{lem:X_cont_in_prob}
  Let $d,m \in \N$, 
  $T \in (0, \infty)$,
  let $\norm{\cdot}\colon\R^d\to[0,\infty)$
  and $\lvert\lvert\lvert\cdot\rvert\rvert\rvert\colon\R^{d\times m}\to[0,\infty)$ be norms,
  let $O\subseteq \R^d$ be a non-empty
  open set,
  for every $r \in (0, \infty)$
  let 
  $O_r \subseteq O$
  satisfy
  $O_r = \{  x \in O\colon \lvert \lvert x \rvert \rvert \leq r \text{ and } \{ y \in \R^d\colon \lvert \lvert y-x \rvert \rvert 
  < \frac{1}{r} \} \subseteq O \}$,
  let $\mu \in  C([0,T] \times O, \R^d)$, 
  $\sigma \in C([0,T] \times O, \R^{d \times m})$
  satisfy for all $r\in(0,\infty)$
  that
  \begin{equation}     
  \label{eq:mu_sigma_loclip}
  \sup \bigg(\bigg\{
  \frac{\norm{\mu(t,x)-\mu(t,y)}
  +\lvert\lvert\lvert\sigma(t,x)
  -\sigma(t,y)\rvert\rvert\rvert}{\norm{x-y}}\colon t\in[0,T], x,y\in O_r, x\neq y \bigg\} 
  \cup \{0\} \bigg)
  <\infty,
  \end{equation}
  let $(\Omega, \mathcal{F}, \mathbb{P}, (\mathbb{F}_s)_{s \in [0,T]})$ 
  be a filtered probability space
  satisfying the usual conditions,
  let $W \colon [0,T] \times \Omega \to \R^m$ 
  be a standard $(\mathbb{F}_s)_{s \in [0,T]}$-Brownian motion,
  for every  $t\in[0,T]$,
  $x \in O$ let
  $X^x_t = (X^{x}_{t,s})_{s \in [t,T]} \colon [t,T] \times \Omega \to O$   
  be an 
  $(\mathbb{F}_s)_{s \in [t,T]}$-adapted 
  stochastic process with continuous
  sample paths satisfying that for all 
  $s \in [t,T]$ it holds a.s.\! that
  \begin{equation}
  X^x_{t,s} = x + \int_t^s \mu(r, X^x_{t,r}) \,\d r + \int_t^s \sigma(r, X^x_{t,r}) \,\d W_r,
  \end{equation}
  and assume for all
  non-empty, compact
  $\cK\subseteq [0,T]\times O$
  that 
  \begin{equation}\label{eq:X_cont_ass}
    \inf_{k\in\N}\Big[\sup_{(t,x)\,\in \cK}\Big(
  \sup_{\tau\colon\Omega\to[t,T] \text{ stopping time}}
    \mathbb{P}\big(\norm{X^{x}_{t,\tau}}\geq k\big)\Big)\Big]=0.
  \end{equation}
  Then it holds for all  
  $\varepsilon\in(0,\infty)$,
  $s\in[0,T]$, and all 
  $(t_n,x_n)\in [0,T]\times O$, $n\in\N_0$,
  with $\limsup_{n\to\infty}[\abs{t_n-t_0}
  +\norm{x_n-x_0}]=0$ that
  \begin{equation}\label{eq:X_cont}
    \limsup\nolimits_{n\to\infty} 
    \Big[ 
    \mathbb{P} \Big( \normm{X^{x_n}_{t_n, \max\{s,t_n\}} 
    - X^{x_0}_{t_0, \max\{s,t_0\}}}
    \geq \varepsilon \Big)\Big] 
    = 0.
  \end{equation}
\end{lemma}

\begin{proof}[Proof of Lemma~\ref{lem:X_cont_in_prob}]
  Throughout this proof let 
  $(t_n,x_n)\in [0,T]\times O$, $n\in\N_0$,
  satisfy $\limsup_{n\to\infty}[\abs{t_n-t_0}
  +\norm{x_n-x_0}]=0$
  and let
  $U_n\subseteq O$, $n\in\N$,
  satisfy for all $n\in\N$ that
  $U_n=\{x\in  O\colon
  (\exists y\in O_n\colon\norm{y-x}
  < \frac{1}{2n})\}$.
  Note that for every $n\in\N$
  it holds that $O_n\subseteq O$ 
  is a compact set,
  $U_n\subseteq O$ is an open set,
  and $O_n\subseteq U_n$.
  Combining this with
  \cite[Theorem II.3.7]{Lang2001}
  (applied for every $n\in\N$
  with $E\curvearrowleft [0,T]\times O$,
  $A_1 \curvearrowleft [0,T]\times O\setminus U_n$,
  $A_2 \curvearrowleft [0,T]\times O_n$
  in the notation of 
  \cite[Theorem II.3.7]{Lang2001})
  demonstrates that
  for all $n\in\N$
  there exists $\varphi_n\in C^\infty([0,T]\times O,\R)$
  with compact support
  which satisfies for all 
  $t\in [0,T]$, $x\in O$ that
  \begin{equation}\label{eq:varphi_n}
    \mathbbm{1}_{[0,T]\times O_n}(t,x)
    \leq \varphi_n(t,x)
    \leq \mathbbm{1}_{[0,T]\times U_n}(t,x).
  \end{equation}    
  Let $m_n\colon[0,T]\times \R^d\to\R^d$,
  $n\in\N$, and
  $s_n\colon[0,T]\times\R^d\to\R^{d\times m}$,
  $n\in\N$, satisfy for all
  $n\in\N$, $t\in [0,T]$, $x\in \R^d$ that
  \begin{equation}
  \label{eq:m_n_and_s_n}
    m_n(t,x) = \begin{cases}
      \varphi_n(t,x)\mu(t,x)
      &\colon x\in O\\
      0 &\colon x\in\R^d\setminus O
    \end{cases}
    \quad \text{and} \quad
    s_n(t,x) = \begin{cases}
      \varphi_n(t,x)\sigma(t,x)
      &\colon x\in O\\
      0 &\colon x\in\R^d\setminus O.
    \end{cases}
  \end{equation}
  Combining this with
  \eqref{eq:mu_sigma_loclip} and
  \eqref{eq:varphi_n}
  shows that
  $m_n\colon [0,T]\times \R^d\to \R^d$,
  $n\in\N$,
  and $s_n\colon [0,T]\times \R^d\to \R^{d\times m}$, $n\in\N$,
  are compactly supported, 
  continuous functions which satisfy that
  \begin{enumerate}[label=(\Roman*)]
  \item\label{it:m_n_s_n_properties1}
  for all $n\in\N$
  it holds that
    \begin{equation}
    \begin{split}
      &\sup_{t\in [0,T]}
      \sup_{\substack{x,y\in\R^d\\x\neq y}}
      \Bigg[ 
      \frac{\norm{m_n(t,x)-m_n(t,y)}
      +\lvert\lvert\lvert s_n(t,x)-s_n(t,y)\rvert\rvert\rvert}{\norm{x-y}} 
      \Bigg]
      <\infty,
    \end{split}
    \end{equation}
  \item\label{it:m_n_s_n_properties2}
     for all $n\in\N$, $t\in [0,T]$,
  $x\in O_n$
  it holds that
  \begin{equation}
    \big[\norm{m_n(t,x)-\mu(t,x)}
    +\lvert\lvert\lvert s_n(t,x)-\sigma(t,x)\rvert\rvert\rvert\big]
    =0,
  \end{equation}
  and 
  \item\label{it:m_n_s_n_properties3}
  for all $n\in\N$, $t\in [0,T]$,
  $x\in O\setminus U_n$
  it holds that
  \begin{equation}
    \big[\norm{m_n(t,x)}
    +\lvert\lvert\lvert s_n(t,x)\rvert\rvert\rvert \big]
    =0.
  \end{equation}
\end{enumerate}    
  Observe that
  \cite[Theorem 5.2.9]{KaratzasShreve1991}
  and item~\ref{it:m_n_s_n_properties1}
  demonstrate that for every
  $n\in\N$, $t\in [0,T]$, $x\in O$
  there exists an 
  $(\mathbb{F}_s)_{s\in [t,T]}$-adapted
  stochastic process 
  $X^{x,n}_t=(X^{x,n}_{t,s})_{s\in[t,T]}
  \colon[t,T]\times\Omega\to\R^d$ 
  with continuous 
  sample paths
  which satisfies 
  that for all $s\in[t,T]$ it holds
  a.s.\! that
  \begin{equation}
  \label{eq:X_app}
    X^{x,n}_{t,s} = x
    +\int_t^s m_n(r, X^{x,n}_{t,r})\,\d r
    +\int_t^s s_n(r, X^{x,n}_{t,r})\,\d W_r.
  \end{equation}
  Moreover, note that item~\ref{it:m_n_s_n_properties3}
  ensures that for all $n\in\N$
  it holds that
  $\operatorname{supp}(m_n)\cup
  \operatorname{supp}(s_n) \subseteq [0,T]\times U_n$.
  This and \cite[Lemma 3.4]{BGHJ2019}
  (applied for every $n\in\N$, 
  $t\in[0,T]$, $x\in O$ 
  with $T\curvearrowleft T-t$,
  $\mathcal{O}\curvearrowleft U_n$,
  $\mu\curvearrowleft ([0,T-t]\times O \ni (s,y)\mapsto m_n(t+s,y)\in\R^d)$,
  $\sigma\curvearrowleft ([0,T-t]\times O \ni (s,y)\mapsto s_n(t+s,y)\in\R^{d\times m})$,
  $\mathbb{F}\curvearrowleft (\mathbb{F}_{t+s})_{s\in[0,T-t]}$,
  $W\curvearrowleft ([0,T-t]\times\Omega\ni(s,\omega)\mapsto W_{t+s}(\omega)-W_t(\omega)\in\R^m)$,
  $X\curvearrowleft ([0,T-t]\times\Omega\ni(s,\omega)\mapsto X^{x,n}_{t,t+s}\in O)$
  in the notation of \cite[Lemma 3.4]{BGHJ2019})
  show that
  for all $n\in\N$, $t\in[0,T]$,
  $x\in U_n$ it holds that
  $\mathbb{P}(\forall s\in[t,T]\colon
  X^{x,n}_{t,s}\in \overline{U_n})=1$
  and 
  for all $n\in\N$, $t\in[0,T]$,
  $x\in O\setminus U_n$ it holds that
  $\mathbb{P}(\forall s\in[t,T]\colon
  X^{x,n}_{t,s}=x)=1$.
  This ensures that for all 
  $n\in\N$, $t\in[0,T]$, $x\in O$ 
  there exists an 
  $(\mathbb{F}_s)_{s\in [t,T]}$-adapted
  stochastic process 
  $\mathcal{X}^{x,n}_t=(\mathcal{X}^{x,n}_{t,s})_{s\in[t,T]}
  \colon[t,T]\times\Omega\to O$ 
  with continuous 
  sample paths
  which satisfies 
  that for all $s\in[t,T]$ it holds
  a.s.\! that
  \begin{equation}
  \label{eq:X_app2}
    \mathcal{X}^{x,n}_{t,s} = x
    +\int_t^s m_n(r, \mathcal{X}^{x,n}_{t,r})\,\d r
    +\int_t^s s_n(r, \mathcal{X}^{x,n}_{t,r})\,\d W_r.
  \end{equation}
  In the next step let
  $\tau^{n,t,x}\colon \Omega\to [t,T]$,
  $n\in\N$, $t\in[0,T]$, $x\in O$
  satisfy for every $n\in\N$, $t\in [0,T]$,
  $x\in O$, $\omega\in\Omega$ that
  $\tau^{n,t,x}(\omega)=
  \inf(\{s\in [t,T]\colon
  \max\{\norm{\mathcal{X}^{x,n}_{t,s}},
  \norm{X^{x}_{t,s}}\}> n \}
  \cup \{T\})$.
  Note that for every $n\in\N$,
  $t\in[0,T]$, $x\in O$
  it holds that 
  $\tau^{n,t,x}\colon \Omega\to [t,T]$
  is an $(\mathbb{F}_s)_{s\in[t,T]}$-
  stopping time.  
  Next observe that
  \cite[Lemma 3.5]{BGHJ2019}
  (applied for every $n\in\N$, 
  $t\in[0,T]$, $x\in O$ 
  with $T\curvearrowleft T-t$,
  $\mathcal{C}\curvearrowleft
  [0,T-t]\times O_n$,
  $\mu_1\curvearrowleft ([0,T-t]\times O\ni (s,y)\mapsto \mu(t+s,y)\in\R^d)$,
  $\mu_2\curvearrowleft ([0,T-t]\times O \ni (s,y)\mapsto m_n(t+s,y)\in\R^d)$,
  $\sigma_1\curvearrowleft ([0,T-t]\times O\ni (s,y)\mapsto \sigma(t+s,y)\in\R^{d\times m})$,
  $\sigma_2\curvearrowleft ([0,T-t]\times O \ni (s,y)\mapsto s_n(t+s,y)\in\R^{d\times m})$,
  $\mathbb{F}\curvearrowleft (\mathbb{F}_{t+s})_{s\in[0,T-t]}$,
  $W\curvearrowleft ([0,T-t]\times\Omega\ni(s,\omega)\mapsto W_{t-s}(\omega)-W_t(\omega)\in\R^m)$,
  $X^{(1)}\curvearrowleft ([0,T-t]\times\Omega\ni(s,\omega)\mapsto X^x_{t,t+s}\in O)$,
  $X^{(2)}\curvearrowleft ([0,T-t]\times\Omega\ni(s,\omega)\mapsto \mathcal{X}^{x,n}_{t,t+s}\in \R^d)$,
  $\tau\curvearrowleft \tau^{n,t,x}-t$
  in the notation of \cite[Lemma 3.5]{BGHJ2019}),
  item~\ref{it:m_n_s_n_properties2},  
  and the fact that $O_n$ 
  is a compact set
  demonstrate that 
  for all $n\in\N$, $t\in[0,T]$,
  $x\in O$
  it holds that
  \begin{equation}
    \mathbb{P}(\forall s\in [t,T]\colon
    \mathbbm{1}_{\{s\leq \tau^{n,t,x}\}}
    \norm{\mathcal{X}^{x,n}_{t,s}-X^{x}_{t,s}}
    =0)
    =1.
\end{equation}    
  This implies that 
  for all $\varepsilon\in(0,\infty)$, 
  $n\in\N$, $t\in[0,T]$,
  $s\in [t,T]$, $x\in O$
  it holds that
  \begin{equation}
  \label{eq:X_n_conv}
  \begin{split}
    \mathbb{P}(\norm{\mathcal{X}^{x,n}_{t,s}
    -X^x_{t,s}}\geq \varepsilon)
    \leq \mathbb{P}(\tau^{n,t,x} < s)
    \leq \mathbb{P}(\norm{X^{x}_{t,\tau^{n,t,x}}}\geq n)
    \leq \sup_{\substack{\tau\colon\Omega\to[t,T]\\ \text{stopping time}}}
    \mathbb{P}(\norm{X^{x}_{t,\tau}}\geq n). 
  \end{split}  
  \end{equation}
  Moreover, observe that 
  \cite[Lemma 3.6]{BGHJ2019}
  and the fact that all norms on
  a finite-dimensional 
  $\R$-vector space
  are equivalent
  prove that there exist 
  $c_k$, $k\in\N$,
  which satisfy for all $k, n\in\N$,
  $s\in[t_0,T]$ that
  \begin{equation}
  \label{eq:X_diff_in_prob}
    \E\big[\norm{\mathcal{X}^{x_n,k}_{t_n,\max\{s,t_n\}}
    -\mathcal{X}^{x_0,k}_{t_0,\max\{s,t_0\}}}^2\big]
    \leq c_k \big[\abs{t_n-t_0}
    +\norm{x_n-x_0}^2\big].
  \end{equation}
  Furthermore, note that
  \eqref{eq:X_app}  
  shows that for all $k,n\in\N$,
  $s\in[t_0,T]$ it holds a.s.\! that
  \begin{equation}
  \label{eq:X_app_diff1}
  \begin{split}
    \mathcal{X}^{x_0,k}_{t_0,\max\{s,t_n\}}
    -\mathcal{X}^{x_0,k}_{t_0,\max\{s,t_0\}}
    =\int_s^{\max\{s,t_n\}} m_k(r,\mathcal{X}^{x_0,k}_{t_0,r})\,\d r
    +\int_s^{\max\{s,t_n\}} s_k(r,\mathcal{X}^{x_0,k}_{t_0,r})\,\d W_r.
  \end{split}
  \end{equation}
  Combining this with
  Minkowski's inequality,
  Itô's isometry,
  and the fact that $m_n$, $n\in\N$,
  and $s_n$, $n\in\N$,
  are continuous functions with
  compact support
  implies that for all $k,n\in\N$,
  $s\in[t_0,T]$ it holds that
  \begin{equation}
  \label{eq:X_app_diff2}
  \begin{split}
    &\Big(\E\Big[\normm{\mathcal{X}^{x_0,k}_{t_0,\max\{s,t_n\}}
    -\mathcal{X}^{x_0,k}_{t_0,\max\{s,t_0\}}}^2\Big]\Big)^{\frac{1}{2}}\\
    &\leq \int_s^{\max\{s,t_n\}}
    \Big(\E\Big[ \normm{m_k(r,\mathcal{X}^{x_0,k}_{t_0,r})}^2\Big]\Big)^{\frac{1}{2}}\,\d r
    +\bigg(\int_s^{\max\{s,t_n\}}
    \E\Big[ \normm{s_k(r,\mathcal{X}^{x_0,k}_{t_0,r})}^2\Big]\,\d r \bigg)^{\frac{1}{2}}\\
    &\leq \abs{\max\{t_n-s,0\}}^{\frac{1}{2}}
    \bigg[ \sqrt{T}\Big( \sup_{t\in [0,T]}\sup_{x\in O} \norm{m_k(t,x)}\Big)
    + \sup_{t\in [0,T]}\sup_{x\in O} \norm{s_k(t,x)}\Big) \bigg]
    <\infty.
  \end{split}
  \end{equation}
  The fact that for all $a,b\in\R$
  it holds that $(a+b)^2\leq 2(a^2+b^2)$
  and \eqref{eq:X_diff_in_prob}
  therefore show that there
  exist $\tilde{c}_k\in [0,\infty)$, 
  $k\in\N$, which satisfy 
  for all $k,n\in\N$, $s\in[t_0,T]$
  that
  \begin{equation}
  \label{eq:X_app_diff3}
  \begin{split}
    &\E\Big[\normm{\mathcal{X}^{x_n,k}_{t_n,\max\{s,t_n\}}
    -\mathcal{X}^{x_0,k}_{t_0,\max\{s,t_0\}}}^2\Big]\\
    &\leq 2\E\Big[\normm{\mathcal{X}^{x_n,k}_{t_n,\max\{s,t_n\}}
    -\mathcal{X}^{x_0,k}_{t_0,\max\{s,t_n\}}}^2\Big]
    +2\E\Big[\normm{\mathcal{X}^{x_0,k}_{t_0,\max\{s,t_n\}}
    -\mathcal{X}^{x_0,k}_{t_0,\max\{s,t_0\}}}^2\Big]\\
    &\leq \tilde{c}_k \big[\abs{t_n-t_0}
    +\norm{x_n-x_0}^2+\max\{t_n-s,0\}\big].
  \end{split}
  \end{equation}
  In addition, observe that
  \eqref{eq:X_app}
  shows that for all $k,n\in\N$,
  $s\in[0,t_0]$ it holds a.s.\! that
  \begin{equation}
  \begin{split}
  \label{eq:X_app_diff4}
    &\mathcal{X}^{x_n,k}_{t_n,\max\{s,t_n\}}
    -\mathcal{X}^{x_0,k}_{t_0,\max\{s,t_0\}}
    = \mathcal{X}^{x_n,k}_{t_n,\max\{s,t_n\}}
    -x_0\\
    &=x_n-x_0+\int_{t_n}^{\max\{s,t_n\}} m_k(r,\mathcal{X}^{x_n,k}_{t_n,r})\,\d r
    +\int_{t_n}^{\max\{s,t_n\}} s_k(r,\mathcal{X}^{x_n,k}_{t_n,r})\,\d W_r.
  \end{split}
  \end{equation}
  Minkowski's inequality,
  Itô's isometry
  and the fact that $m_n$, $n\in\N$,
  and $s_n$, $n\in\N$,
  are continuous functions 
  with compact support 
  ensure
  that for all $k,n\in\N$,
  $s\in[0,t_0]$ it holds that
  \begin{equation}
  \label{eq:X_app_diff5}
  \begin{split}
    &\Big(\E\Big[\normm{\mathcal{X}^{x_n,k}_{t_n,\max\{s,t_n\}}
    -\mathcal{X}^{x_0,k}_{t_0,\max\{s,t_0\}}}^2\Big]\Big)^{\frac{1}{2}}
    -\norm{x_n-x_0}\\
    &\leq 
    \int_{t_n}^{\max\{s,t_n\}}
    \Big(\E\Big[ \normm{m_k(r,\mathcal{X}^{x_n,k}_{t_n,r})}^2\Big]\Big)^{\frac{1}{2}}\,\d r
    +\bigg(\int_{t_n}^{\max\{s,t_n\}}
    \E\Big[ \normm{s_k(r,\mathcal{X}^{x_n,k}_{t_n,r})}^2\Big]\,\d r \bigg)^{\frac{1}{2}}\\
    &\leq 
    \abs{\max\{t_n-s,0\}}^{\frac{1}{2}}
    \bigg[ \sqrt{T}\Big( \sup_{t\in [0,T]}\sup_{x\in O} \norm{m_k(t,x)}\Big)
    + \sup_{t\in [0,T]}\sup_{x\in O} \norm{s_k(t,x)}\Big) \bigg]
    <\infty.
  \end{split}
  \end{equation}
  This and
  \eqref{eq:X_app_diff3}   
  imply that there
  exist $\gamma_k\in [0,\infty)$, 
  $k\in\N$, which satisfy 
  for all $k,n\in\N$, $s\in[0,T]$
  that
  \begin{equation}
  \begin{split}
  \label{eq:X_app_diff6}
    &\E\Big[\normm{\mathcal{X}^{x_n,k}_{t_n,\max\{s,t_n\}}
    -\mathcal{X}^{x_0,k}_{t_0,\max\{s,t_0\}}}^2\Big]\\
    &\leq \gamma_k \big[\abs{t_n-t_0}
    +\norm{x_n-x_0}^2
    +\mathbbm{1}_{[0,t_0]}(s)
    \max\{s-t_n,0\}
    +\mathbbm{1}_{[t_0,T]}(s)
    \max\{t_n-s,0\}\big].
  \end{split}
  \end{equation}
  Next note that the fact that
  $\limsup_{n\to\infty}[\abs{t_n-t_0}
  +\norm{x_n-x_0}]=0$
  ensures that there exists 
  a compact set
  $\tilde{\cK}\subseteq [0,T]\times O$
  which satisfies for all
  $n\in\N_0$ that
  $(t_n,x_n)\in \tilde{\cK}$.
  Markov's inequality,
  \eqref{eq:X_cont_ass},
  \eqref{eq:X_n_conv}, 
  and \eqref{eq:X_app_diff6}  
  hence
  show that for all 
  $\varepsilon\in(0,\infty)$, $s\in[0,T]$
  it holds that
  \begin{equation}
  \label{eq:X_app_diff7}
  \begin{split}
    &\limsup_{n\to\infty} \Big[
    \mathbb{P}\Big(\norm{X^{x_n}_{t_n,\max\{s,t_n\}}
    - X^{x_0}_{t_0,\max\{s,t_0\}}}\geq \varepsilon\Big) \Big]\\
    &\leq \inf_{k\in N} \bigg(\limsup_{n\to\infty} \bigg[
    \mathbb{P}\Big(\norm{X^{x_n}_{t_n,\max\{s,t_n\}}
    - \mathcal{X}^{x_n,k}_{t_n,\max\{s,t_n\}}} 
    \geq \frac{\varepsilon}{3} \Big)\\
    &\qquad + \mathbb{P}\Big(\norm{\mathcal{X}^{x_n,k}_{t_n,\max\{s,t_n\}}
    - \mathcal{X}^{x_0,k}_{t_0,\max\{s,t_0\}}} 
    \geq \frac{\varepsilon}{3} \Big)
    + \mathbb{P}\Big(\norm{\mathcal{X}^{x_0,k}_{t_0,\max\{s,t_0\}}
    - X^{x_0}_{t_0,\max\{s,t_0\}}} 
    \geq \frac{\varepsilon}{3} \Big)
    \bigg]\bigg)\\
    &\leq \inf_{k\in N} \bigg(\limsup_{n\to\infty} \bigg[
    \sup_{\substack{\tau\colon\Omega\to[t_n,T]\\ \text{stopping time}}}
    \mathbb{P}(\norm{X^{x_n}_{t_n,\tau}}\geq k)
    + \frac{9}{\varepsilon^2}
    \E\Big[\norm{\mathcal{X}^{x_n,k}_{t_n,\max\{s,t_n\}}
    - \mathcal{X}^{x_0,k}_{t_0,\max\{s,t_0\}}}^2\Big]\\
    &\qquad 
    + \sup_{\substack{\tau\colon\Omega\to[t_0,T]\\ \text{stopping time}}}
    \mathbb{P}(\norm{X^{x_0}_{t_0,\tau}}\geq k)
    \bigg]\bigg)\\
    &\leq \inf_{k\in N} \bigg(\limsup_{n\to\infty} \bigg[
    \sup_{\substack{\tau\colon\Omega\to[t_n,T]\\ \text{stopping time}}}
    \mathbb{P}(\norm{X^{x_n}_{t_n,\tau}}\geq k)
    + \sup_{\substack{\tau\colon\Omega\to[t_0,T]\\ \text{stopping time}}}
    \mathbb{P}(\norm{X^{x_0}_{t_0,\tau}}\geq k)\\
    &\qquad + \frac{9\gamma_k}{\varepsilon^2}
    (\abs{t_n-t_0}+\norm{x_n-x_0}^2
    +\mathbbm{1}_{[0,t_0]}(s)\max\{s-t_n,0\}
    +\mathbbm{1}_{[t_0,T]}(s)\max\{t_n-s,0\}) 
    \bigg]\bigg)\\
    &= \inf_{k\in\N} 
    \bigg(\limsup_{n\to\infty} 
    \Big[\sup_{\substack{\tau\colon\Omega\to[t_n,T]\\ \text{stopping time}}}
    \mathbb{P}(\norm{X^{x_n}_{t_n,\tau}}\geq k)
     +\sup_{\substack{\tau\colon\Omega\to[t_0,T]\\ \text{stopping time}}}
    \mathbb{P}(\norm{X^{x_0}_{t_0,\tau}}\geq k)\Big]
    \bigg)\\
    &\leq \inf_{k\in \N} 
    \bigg(\sup_{(u,y)\in \tilde{\cK}} 
    \Big[2\sup_{\substack{\tau\colon\Omega\to[u,T]\\ \text{stopping time}}}
    \mathbb{P}(\norm{X^{y}_{u,\tau}}\geq k)\Big]
    \bigg)
    =0.
  \end{split}
  \end{equation}
  This establishes \eqref{eq:X_cont}.
  The proof of Lemma~\ref{lem:X_cont_in_prob}
  is thus complete.  
\end{proof}

The next lemma is a consequence
of Lemma~\ref{lem:X_cont_in_prob}
and proves convergence in mean  for the stochastic process in \eqref{eq:Z_2}
which is the stochastic process
arising from the Bismut-Elworthy-Li 
formula.

\begin{lemma}\label{lem:Z_continuity}
  Assume Setting~\ref{setting:sfpe2},
  for every $r \in (0, \infty)$
  let $O_r \subseteq O$
  satisfy
  $O_r = \{  x \in O\colon \lvert \lvert x \rvert \rvert \leq r \text{ and } \{ y \in \R^d\colon \lvert \lvert y-x \rvert \rvert 
  < \frac{1}{r} \} \subseteq O \}$,
  assume for all $r\in (0,\infty)$,
  $j\in\{1,\ldots, d\}$ that
  \begin{equation}     
  \label{eq:der_mu_sigma_loclip}
  \begin{split}
  \sup \bigg(\bigg\{
  &\frac{\norm{\frac{\partial \mu}{\partial x}(t,x)-\frac{\partial \mu}{\partial x}(t,y)}_F
  +\norm{\frac{\partial \sigma}{\partial x_j}(t,x)
  -\frac{\partial \sigma}{\partial x_j}(t,y)}_F}{\norm{x-y}}\colon\\ 
  &\hspace{11em} t\in[0,T], x,y\in O_r, x\neq y \bigg\} 
  \cup \{0\} \bigg)
  <\infty,
  \end{split}
  \end{equation}
  and for every $t\in [0,T]$, 
  $x \in O$ let
  $Z^x_t = (Z^x_{t,s})_{s \in (t,T]} 
  \colon (t,T] \times \Omega \to \R^d$ 
  be an $(\mathbb{F}_s)_{s \in (t,T]}$-adapted stochastic process 
  with continuous sample paths
  satisfying that 
  for all $s \in (t,T]$ 
  it holds a.s.\!   that
  \begin{equation}
  \label{eq:Z_2}
    Z^x_{t,s} = \frac{1}{s-t} \int_t^s
  (\sigma(r,   X^x_{t,r}))^{-1} \; \Big(\frac{\partial}{\partial x} X^x_{t,r}\Big) \,\d W_r.
  \end{equation} 
  Then it holds for all  
  $t \in [0,T)$, $s\in(t,T]$, 
  $x\in O$ that
  \begin{equation}\label{eq:Z_cont}
    \limsup\nolimits_{[0,s)\times O \ni (u,y)\to (t,x)} 
    \Big[ 
     \E \big[ \norm{ Z^{y}_{u, s} 
    - Z^{x}_{t, s}} \big] \Big] 
    = 0.
  \end{equation}
\end{lemma}

\begin{proof}[Proof of Lemma~\ref{lem:Z_continuity}]
  Throughout this proof
  let $\mathfrak{m}\in [0,\infty)$
  satisfy $\mathfrak{m}
  =\max_{s\in [0,T]}[\frac{1}{2}
  \norm{\mu(s,0)}^2
  +\norm{\sigma(s,0)}_F^2]$
  and
  for every $t\in [0,T]$, 
  $x \in O$ let
  $Y^x_t = (Y^x_{t,s})_{s \in [t,T]} 
  \colon [t,T] \times \Omega \to \R^d$ 
  be an $(\mathbb{F}_s)_{s \in [t,T]}$-adapted stochastic process 
  with continuous sample paths
  satisfying that 
  for all $s \in [t,T]$ 
  it holds a.s.\!   that
  \begin{equation}
    Y^x_{t,s} = \int_t^s
  (\sigma(r,   X^x_{t,r}))^{-1} \; \Big(\frac{\partial}{\partial x} X^x_{t,r}\Big) \,\d W_r.
  \end{equation}
  Next observe that 
  Markov's inequality and
  item~\ref{it:X_L2_bd} and
  \ref{it:der_X_bd} of 
  Lemma~\ref{lem:Z_L_2_bd}
  show that
  for all $t\in[0,T]$, $x\in O$,
  $\tau\colon\Omega\to [t,T]$
  stopping time it holds that
  \begin{equation}
  \begin{split}
    &\mathbb{P}\Big(\norm{X^x_{t,\tau}}
    +\normm{\frac{\partial }{\partial x}
    X^x_{t,\tau}}_F \geq k\Big)
    \leq \frac{1}{k^2} \E\Big[
    \big(\norm{X^x_{t,\tau}}
    +\normm{\frac{\partial }{\partial x}
    X^x_{t,\tau}}_F \big)^{2} \Big]\\
    &\leq \frac{2}{k^2} \E\Big[
    \norm{X^x_{t,\tau}}^2
    +\normm{\frac{\partial }{\partial x}
    X^x_{t,\tau}}_F^2 \Big]
    \leq \frac{2}{k^2}\Big(
    \exp((4c+2)T)
    \Big(\norm{x}^2+
    \frac{\mathfrak{m}}{2c+1}\Big)^2
    +d\exp(2cT)\Big).
  \end{split}
  \end{equation}
  This implies that 
  for all
  non-empty, compact
  $\cK\subseteq [0,T]\times O$
  it holds that 
  \begin{equation}
  \label{eq:X_der_X_cont_ass}
  \begin{split}
    &\inf_{k\in\N}\Big[\sup_{(t,x)\,\in \cK}\Big(
  \sup_{\tau\colon\Omega\to[t,T] \text{ stopping time}}
    \mathbb{P}\Big(\norm{X^{x}_{t,\tau}}+\normm{\frac{\partial }{\partial x}
    X^x_{t,\tau}}_F
    \geq k\Big)\Big)\Big]\\
    &\leq \inf_{k\in\N}\Big[\sup_{(t,x)\,\in \cK}\Big(
  \frac{2}{k^2}\Big(
    \exp((4c+2)T)
    \Big(\norm{x}^2+
    \frac{\mathfrak{m}}{2c+1}\Big)^2
    +d\exp(2cT)\Big) \Big)\Big]    
    =0.
  \end{split}
  \end{equation}    
  Lemma~\ref{lem:X_cont_in_prob}
  (applied with
  $d\curvearrowleft 2d$,
  $m\curvearrowleft d$,  
  $O\curvearrowleft O\times\R^d$,
  $\mu\curvearrowleft ([0,T]\times\R^d\times\R^d \ni (t,x,z) \mapsto (\mu(t,x),
  (\frac{\partial \mu}{\partial x})(t,x)z)\in\R^d\times\R^d)$,
   $\sigma\curvearrowleft ([0,T]\times\R^d\times\R^d\ni (t,x,z) \mapsto (\sigma(t,x),
  \sum_{i=1}^d (\frac{\partial \sigma}{\partial x_i})(t,x)z_i)\in\R^{d\times d}\times\R^{d\times d})$,
  and for every
  $j\in\{1,\ldots,d\}$,
  $t\in[0,T]$, $x\in O$
  with 
  $X^x_t\curvearrowleft ([t,T]\times \Omega \ni (s,\omega)\mapsto (X^x_{t,s}(\omega), \frac{\partial}{\partial x_j}X^x_{t,s}(\omega)) \in O\times \R^d)$
  in the notation of 
  Lemma~\ref{lem:X_cont_in_prob})
  and \eqref{eq:der_mu_sigma_loclip} 
  therefore demonstrate that
  for all 
  $j\in\{1,\ldots,d\}$,
  $\varepsilon\in(0,\infty)$,
  $s\in[0,T]$, and all 
  $(t_n,x_n)\in [0,T]\times O$, $n\in\N_0$,
  with $\limsup_{n\to\infty}[\abs{t_n-t_0}
  +\norm{x_n-x_0}]=0$
  it holds that
  \begin{equation}
  \label{eq:X_der_X_conv_in_prob}
  \begin{split}
    \limsup_{n\to\infty} \Big[\mathbb{P}
    \Big(\norm{X^{x_n}_{t_n,\max\{s,t_n\}}
    -X^{x_0}_{t_0,\max\{s,t_0\}}}
    +\normm{\frac{\partial }{\partial x_j}
    X^{x_n}_{t_n,\max\{s,t_n\}}
    -\frac{\partial }{\partial x_j}
    X^{x_0}_{t_0,\max\{s,t_0\}}}
    \geq \varepsilon\Big)\Big]
    =0.
  \end{split}
  \end{equation}  
  Furthermore, observe that 
  the assumption that
  for all $t\in[0,T]$,
  $x\in O$, $v\in\R^d$
  it holds that
  $v^* \sigma(t,x) (\sigma(t,x))^* v \geq \alpha \norm{v}^2$
  ensures that 
  for all $t\in [0,T]$, $x\in\R^d$
  it holds that
  $\sigma(t,x)(\sigma(t,x))^*$
  is invertible.
  Combining this with the fact
  that for all 
  $t\in [0,T]$, $x\in O$
  it holds that
  $\operatorname{rank}(\sigma(t,x)(\sigma(t,x))^*)
  =\operatorname{rank}(\sigma(t,x))$
  shows that for all
  $t\in [0,T]$, $x\in O$
  it holds that
  $\sigma(t,x)$ is invertible.
  Moreover, note that
  the assumption that $\sigma$
  is continuous 
  implies that $\sigma^{-1}$ 
  is continuous.
  Combining this with
  \eqref{eq:X_der_X_conv_in_prob}
  and the Continuous Mapping Theorem
  shows that for all 
  $\varepsilon\in (0,\infty)$,
  $s\in [0,T]$, 
  $(t_n,x_n)\in [0,T]\times O$, 
  $n\in\N$, with
  $\limsup_{n\to\infty}[\abs{t_n-t_0}
  +\norm{x_n-x_0}]=0$
  it holds that
  \begin{equation}
  \label{eq:sigma_inv_X_cont_in_prob}
    \limsup_{n\to\infty} \big[
    \mathbb{P}(
    \norm{(\sigma(\max\{s,t_n\}, X^{x_n}_{t_n,\max\{s,t_n\}}))^{-1}
    -(\sigma(\max\{s,t_0\}, X^{x_0}_{t_0,\max\{s,t_0\}}))^{-1}}\geq \varepsilon \big]
    =0.
  \end{equation}
   This and 
  \eqref{eq:X_der_X_conv_in_prob}
  ensure that for all
  $\varepsilon\in (0,\infty)$,
  $s\in [0,T]$, 
  $(t_n,x_n)\in [0,T]\times O$, 
  $n\in\N$, with
  $\limsup_{n\to\infty}[\abs{t_n-t_0}
  +\norm{x_n-x_0}]=0$
  it holds that
  \begin{equation}
  \label{eq:integrand_cont_in_prob}
  \begin{split}
    &\limsup_{n\to\infty} 
    \Bigg[\mathbb{P}\Big(
    \normm{
    (\sigma(\max\{s,t_n\}, X^{x_n}_{t_n,\max\{s,t_n\}}))^{-1}    
    \frac{\partial}{\partial x}X^{x_n}_{t_n,\max\{s,t_n\}}\\
    &\hspace{13em}
    -(\sigma(\max\{s,t_0\}, X^{x_0}_{t_0,\max\{s,t_0\}}))^{-1}    
    \frac{\partial}{\partial x}X^{x_0}_{t_0,\max\{s,t_0\}}}
    \geq \varepsilon \Big)\Big]
    =0. 
  \end{split}
  \end{equation}
  Next observe that the assumption that
  for all $r \in [t,T]$, $x \in O$,
  $v\in\R^d$ 
  it holds that
  $v^* \sigma(r,x) (\sigma(r,x))^* v  
  \geq \alpha \lvert \lvert v \rvert \rvert^2$ 
  ensures that for all 
  $t \in [0,T]$,
  $x \in O$, $v\in\R^d$
  it holds that
  \begin{equation}
  \norm{v}^2 = v^*  \left( \sigma(t, x) \right)^{-1} \sigma(t,x) ( \sigma(t, x))^* \left( (\sigma(t, x))^{-1} \right)^* v
  \geq \alpha \normm{ \left( (\sigma(t, x))^{-1} \right)^* v}^2. 
  \end{equation}
  Hence, we obtain for all  
  $t \in [0,T]$,
  $x \in O$
  that
  \begin{equation}
  \label{eq:sigma_inv_bd2}
  \lvert \lvert  (\sigma(t, x))^{-1}   \rvert \rvert^2_{L(\R^d)}
  =\sup \limits_{v \in \R^d \backslash \{ 0 \}} \frac{\lvert \lvert \left( (\sigma(t, x))^{-1} \right)^* v \rvert \rvert^2}{\lvert \lvert v \rvert \rvert^2} 
  \leq \frac{1}{\alpha}.
  \end{equation}
  Combining this with
  item~\ref{it:der_X_bd} of
  Lemma~\ref{lem:Z_L_2_bd}
  shows that for all
  $t\in [0,T]$,
  $s \in (t,T]$, $x \in \R^d$
  it holds that
  \begin{equation}
  \label{eq:sigma_inv_der_X_L^2_bd}
  \begin{split}
    &\E \left[ \normmm{ (\sigma(s, X^x_{t,s}))^{-1} 
    \Big(\frac{\partial}{\partial x} X^x_{t,s}\Big)}_F^2 \right] 
    \leq  \E \left[\normmm{(\sigma(s, X^x_{t,s}))^{-1} 
    \Big(\frac{\partial}{\partial x} X^x_{t,s}\Big)}^2_F \right] \\
    &\leq \E\left[ \normm{(\sigma(s, X^x_{t,s}))^{-1}}^2_{L(\R^d)}
    \normmm{\frac{\partial}{\partial x} X^x_{t,s}}^2_{F} \right] 
    \leq \frac{1}{\alpha} \E \left[ 
    \normmm{\frac{\partial}{\partial x} X^x_{t,s}}_F^2 \right]\\
    &= \frac{1}{\alpha}
    \E \left[ \normmm{\frac{\partial}{\partial x} X^x_{t,s}}_F^2\right]
    \leq  \frac{d}{\alpha}\exp(2cT).
  \end{split}
  \end{equation}
  Combining this and
  \eqref{eq:integrand_cont_in_prob}, 
  with
  \cite[Corollary 18.13]{Kallenberg2021}
  (applied for every $n\in\N$,
  $t\in [0,T]$, $s\in (t,T]$,
  $x\in O$, $(t_m)_{m\in\N}\subseteq [0,s)$ 
  which satisfies $\lim_{m\to\infty} t_m =t$,
  $(x_m)_{m\in\N}\subseteq O$ 
  which satisfies $\lim_{m\to\infty} x_m=x$,
  with
  $X\curvearrowleft W$,
  $U\curvearrowleft (\frac{d}{\alpha})^{\nicefrac{1}{2}} \exp(cT)$,
  $V \curvearrowleft
  (([0,T]\times\Omega\ni (s,\omega)
  \mapsto \sigma(s, X^x_{t,s}(\omega)))^{-1} 
  (\frac{\partial}{\partial x} X^x_{t,s}(\omega)\in \R^{d\times d}))$,
  $V_n \curvearrowleft
  (\Omega\times[0,T] \ni (\omega,s)\mapsto 
  (\sigma(s, X^{x_n}_{t_n,s}(\omega)))^{-1} 
  (\frac{\partial}{\partial x} X^{x_n}_{t_n,s}(\omega))\in\R^{d\times d})$
  in the notation of 
  \cite[Corollary 18.13]{Kallenberg2021})
  ensures that   
  for all  
  $\varepsilon\in(0,\infty)$,
  $s\in[0,T]$, and all 
  $(t_n,x_n)\in [0,T]\times O$, $n\in\N_0$,
  with $\limsup_{n\to\infty}[\abs{t_n-t_0}
  +\norm{x_n-x_0}]=0$ 
  it holds that
  \begin{equation}  
  \label{eq:Y_cont_in_prob}
  \begin{split}
    \limsup\nolimits_{n\to\infty} 
    \Big[ 
    \mathbb{P} \Big( \lvert \lvert 
     Y^{x_n}_{t_n, \max\{s,t_n\}} 
    - Y^{x_0}_{t_0, \max\{s,t_0\} } \rvert \rvert 
    \geq \varepsilon \Big)\Big] 
    = 0.
    \end{split}
  \end{equation}  
  In addition, note that 
  item~\ref{it:Y_bd_st}
  of Lemma~\ref{lem:Z_L_2_bd}
  ensures that
  for all $t\in [0,T]$,
  $x\in O$, $s\in [t,T]$ 
  it holds that
  \begin{equation}\label{eq:Y_L2_bd}
  \begin{split}
    &\E\big[\norm{Y^x_{t,s}}^2\big] 
    \leq \frac{dT}{\alpha}\exp(2cT). 
  \end{split}
  \end{equation}
  Next observe that Hölder's inequality
  demonstrates that
  for all $\varepsilon\in(0,\infty)$,
  $s\in (0,T]$, $u,t\in [0,s)$,
  $x,y\in O$
  it holds that
  \begin{equation}
  \begin{split}
    &\E\big[\norm{Y^y_{u,s}-Y^x_{t,s}}\big]
    \leq \varepsilon +
    \E\big[\norm{Y^y_{u,s}-Y^x_{t,s}}
    \mathbbm{1}_{\{\norm{Y^y_{u,s}-Y^x_{t,s}}>\varepsilon\}}\big]\\
    &\leq \varepsilon +
    \Big(\E\big[\norm{Y^y_{u,s}-Y^x_{t,s}}^2\big]\Big)^{\frac{1}{2}}
    \Big(\E\big[\mathbbm{1}_{\{\norm{Y^y_{u,s}-Y^x_{t,s}}>\varepsilon\}}\big]\Big)^{\frac{1}{2}}\\
    &\leq \varepsilon + 
    \Big(\E\big[\norm{Y^y_{u,s}-Y^x_{t,s}}^2\big]\Big)^{\frac{1}{2}}
    \Big(\mathbb{P}\big(\norm{Y^y_{u,s}-Y^x_{t,s}}>\varepsilon\big)\Big)^{\frac{1}{2}}.
  \end{split}
  \end{equation}
  Combining this with
  \eqref{eq:Y_cont_in_prob}
  and \eqref{eq:Y_L2_bd}
  proves
  that for all
  $t \in [0,T)$, $s\in[t,T]$, 
  $x\in O$ it holds that
  \begin{equation}\label{eq:Y_cont_in_mean}
    \limsup\nolimits_{[0,s)\times O \ni (u,y)\to (t,x)} 
    \Big[ 
     \E \big[ \norm{ Y^{y}_{u, s} 
    - Y^{x}_{t, s}} \big] \Big] 
    = 0.
  \end{equation}  
  Throughout the rest of the proof
  let $(t_0,x_0)\in [0,T]\times O$,
  $s_0 \in (t_0,T]$, and
  $(t_n,x_n)\in [0,s_0)\times O$, 
  $n\in\N$ satisfy that
  $\limsup_{n\to\infty}[\abs{t_n-t_0}
  +\norm{x_n-x_0}]=0$.
  Observe that for the proof of 
  \eqref{eq:Z_cont} it is sufficient
  to show that 
  \begin{equation}\label{eq:Z_cont_seq}
    \limsup_{n\to\infty} 
    \Big[ 
     \E \big[ \norm{ Z^{x_n}_{t_n, s_0} 
    - Z^{x_0}_{t_0, s_0}} \big] \Big] 
    = 0.
  \end{equation}  
  Next note that the fact that
  $\limsup_{n\to\infty}[\abs{t_n-t_0}]=0$
  ensures that there exists $r\in\N$
  which satisfies that 
  for all $n\in\N$ it holds that 
  $t_n\in [0,s_0-\frac{1}{r}]$.
  Hence, we obtain for all
  $n\in\N$ that
  \begin{equation}
  \begin{split} 
    &\E \big[ \norm{ Z^{x_n}_{t_n, s_0} 
    - Z^{x_0}_{t_0, s_0}} \big]
    =\E\bigg[\normmm{\frac{1}{s_0-t_n}Y^{x_n}_{t_n, s_0}
    -\frac{1}{s_0-t_0} Y^{x_0}_{t_0,s_0}} \bigg]\\
  &\leq \frac{1}{s_0-t_n}\E\big[
  \norm{Y^{x_n}_{t_n,s_0}
    -Y^{x_0}_{t_0,s_0}} \big]
    + \Big| \frac{1}{s_0-t_n}-
  \frac{1}{s_0-t_0}\Big|
  \,\E\big[\norm{Y^{x_0}_{t_0,s_0}}\big]\\    
    &\leq \sup_{u\in [0,T-\frac{1}{r}]}
    \Big[\frac{1}{s_0-u}\Big]
    \E\bigg[
  \norm{Y^{x_n}_{t_n,s_0}
    -Y^{x_0}_{t_0,s_0}} \bigg]
    +\Big| \frac{1}{s_0-t_n}-\frac{1}{s_0-t_0}\Big|
  \,\Big(\E\big[\norm{Y^{x_0}_{t_0,s_0}}^2 
    \big]\Big)^{\frac{1}{2}}.    
  \end{split}
  \end{equation} 
  Combining this with 
  \eqref{eq:Y_L2_bd}
  and \eqref{eq:Y_cont_in_mean}
  demonstrates \eqref{eq:Z_cont_seq}.
  The proof of Lemma~\ref{lem:Z_continuity}
  is thus complete.
\end{proof}

\subsection{Existence and uniqueness of solutions of SFPEs associated with SDEs}
\label{subse:Z_exist_uniq}

In this section we combine  
Theorem~\ref{thm:ex_cont_solution}
with the results in Section~\ref{subsec:Z_bounds}-\ref{subsec:Z_convergence}
to obtain our main result,
Theorem ~\ref{thm:ex_cont_sol_Z}.
Theorem ~\ref{thm:ex_cont_sol_Z} 
extends the abstract existence and
uniqueness result in
Theorem~\ref{thm:ex_cont_solution}
to SFPEs associated
with SDEs.
Theorem~\ref{thm:ex_cont_sol_Z}
generalizes \cite[Theorem 3.8]{BGHJ2019}
to the case of  gradient-dependent
nonlinearities.

\begin{theorem}
\label{thm:ex_cont_sol_Z}
  Assume Setting~\ref{setting:sfpe2},
  let 
  $b, K, L\in (0, \infty)$,
  let $\lvert\lvert\lvert\cdot\rvert\rvert\rvert\colon\R^{d+1}\to[0,\infty)$
  satisfy for all
  $x=(x_1,\ldots,x_{d+1})\in\R^{d+1}$
   that
  $\lvert\lvert\lvert x \rvert\rvert\rvert =\textstyle (\sum_{i=1}^{d+1} \abs{x_i}^2)^{1/2}$,
  for every $r \in (0, \infty)$
  let $K_r\subseteq [0,T)$,  
  $O_r \subseteq O$
  satisfy
  $K_r=[0,\max\{T-\frac{1}{r},0\}]$
  and
  $O_r = \{  x \in O\colon \lvert \lvert x \rvert \rvert \leq r \text{ and } \{ y \in \R^d\colon \lvert \lvert y-x \rvert \rvert 
  < \frac{1}{r} \} \subseteq O \}$,
  assume for all $r\in (0,\infty)$,
  $j\in\{1,\ldots, d\}$ that
  \begin{equation}     
  \label{eq:der_mu_sigma_loclip2}
  \begin{split}
  &\sup \bigg(\bigg\{
  \frac{\norm{\frac{\partial \mu}{\partial x}(t,x)-\frac{\partial \mu}{\partial x}(t,y)}_F
  +\norm{\frac{\partial \sigma}{\partial x_j}(t,x)
  -\frac{\partial \sigma}{\partial x_j}(t,y)}_F}{\norm{x-y}}
  \colon t\in[0,T], x,y\in O_r, x\neq y \bigg\} 
  \cup \{0\} \bigg)
  <\infty,
  \end{split}
  \end{equation}
  for every $t\in[0,T]$, 
  $x \in O$ let
  $Z^x_t = (Z^x_{t,s})_{s \in (t,T]} 
  \colon (t,T] \times \Omega \to \R^{d+1}$ 
  be an $(\mathbb{F}_s)_{s \in (t,T]}$-adapted
  stochastic process 
  with continuous sample paths
  satisfying that 
  for all $s \in (t,T]$ 
  it holds a.s.\! that
  \begin{equation}
    Z^x_{t,s} = 
    \begin{pmatrix}
      1\\
      \frac{1}{s-t} \int_t^s
  (\sigma(r,   X^x_{t,r}))^{-1} \; \Big(\frac{\partial}{\partial x} X^x_{t,r}\Big) \,\d W_r,
    \end{pmatrix}     
  \end{equation}   
  let $V \in C^{1,2}([0,T]\times O,(0, \infty))$
  satisfy a.s.\! for all 
  $t\in [0,T]$, $s\in[t,T]$, $x\in O$
  that 
  \begin{equation}\label{eq:V_ass}
  \begin{split}
    &(\tfrac{\partial V}{\partial s})(s, X^x_{t,s})
    +\langle(\tfrac{\partial V}{\partial x})(s, X^x_{t,s}),\mu(s, X^x_{t,s})\rangle
    +\tfrac{1}{2}\operatorname{Trace}(\sigma(s, X^x_{t,s})[\sigma(s, X^x_{t,s})]^*(\operatorname{Hess}_x V)(s, X^x_{t,s}))\\
    &\qquad +\tfrac{1}{2}\tfrac{\norm{(\frac{\partial V}{\partial x})(s, X^x_{t,s})\sigma(s, X^x_{t,s})}^2_F}{V(s, X^x_{t,s})}
    \leq K V(s, X^x_{t,s})+b,
  \end{split}
  \end{equation}
  let $f \in C([0,T) \times O \times \R^{d+1}, \R)$,
  $g \in C(O, \R)$ satisfy for all
  $t \in [0,T)$, $x \in O$, $v, w \in \R^{d+1}$
  that
  $\lvert f(t,x,v) - f(t,x,w) \rvert 
  \leq L\lvert\lvert\lvert  v-w\rvert\rvert \rvert $,
  and assume that
  $\inf_{r \in (0, \infty)} [ \sup_{t \in [0,T)\setminus K_r} $
  $\sup_{x \in O \setminus O_r} 
  (\frac{\lvert  g(x) \rvert}{V(T,x)} )
   +\frac{\lvert f(t,x,0) \rvert}{V(t,x)}
  \sqrt{T-t})] 
  = 0$,
  $\liminf_{r\to\infty}
  [\inf_{t\in [0,T]}
  \inf_{x\in O\setminus O_r}
  V(t,x)]=\infty$, 
  and 
  $\inf_{t\in[0,T]}\inf_{x\in O}
V(t,x)>0$.
  Then there exists a unique 
  $v \in C([0,T) \times O, \R^{d+1})$
such that
\begin{enumerate}[label=(\roman*)]
\item\label{it:cor_ex_cont_sol1}
it holds that
\begin{equation}\label{eq:cor_ex_cont_solution1}
\limsup \limits_{r \to \infty} \left[ \sup \limits_{t \in [0,T)\setminus K_r} \sup \limits_{x \in O \setminus O_r} \left( \frac{\lvert\lvert\lvert v(t,x)\rvert\rvert\rvert}{V(t,x)}  \, \sqrt{T-t} \right) \right] = 0,
\end{equation}
\item\label{it:cor_ex_cont_sol2}
 it holds that
  \begin{equation}\label{eq:cor_ex_cont_sol2}
     [0,T) \times O \ni (t,x) \mapsto 
     \E \left[  g(X^x_{t,T}) Z^x_{t,T}  
     + \int_t^T  f(r, X^x_{t,r}, v(r, X^x_{t,r})) Z^x_{t,r}  \,\d r \right] \in \R^{d+1}
  \end{equation}
is well-defined and continuous,
and
\item\label{it:cor_ex_cont_sol3}
for all
$t \in [0,T)$, $x \in O$
it holds that
\begin{equation}\label{eq:cor_ex_cont_solution2}
v(t,x) = \E \left[ g(X^x_{t,T})Z^x_{t,T} + \int_t^T f(r, X^x_{t,r}, v(r, X^x_{t,r}))Z^x_{t,r} \,\d r \right].
\end{equation}
\end{enumerate}
\end{theorem}

\begin{proof}[Proof of Theorem~\ref{thm:ex_cont_sol_Z}]
  First observe that
  \cite[Theorem 2.4]{HHM2019}
  (applied for every
  $t\in[0,T]$, $x\in O$
  with 
  $H\curvearrowleft \R^d$,
  $U\curvearrowleft \R^d$,
  $T\curvearrowleft T-t$,
  $X\curvearrowleft ([0,T-t]\times\Omega\ni (s,\omega)\mapsto X^x_{t,t+s}(\omega)\in\R^d)$,
  $a\curvearrowleft ([0,T-t]\times\Omega \ni (s,\omega) \mapsto\mu(t+s,x)\in\R^d)$,
  $b\curvearrowleft ([0,T-t]\times\Omega\ni (s,\omega) \mapsto\sigma(t+s,x)\in\R^{d\times d})$,
  $p\curvearrowleft 2$,
  $\alpha\curvearrowleft ([0,T-t]\times\Omega \ni (s,\omega)\mapsto K \in [0,\infty))$,
  $\beta \curvearrowleft ([0,T-t]\times\Omega \ni (s,\omega)\mapsto b \in [0,\infty))$,
  $q_1\curvearrowleft 2$,
  $q_2\curvearrowleft \infty$
  in the notation of \cite[Theorem 2.4]{HHM2019})
  and \eqref{eq:V_ass}
  demonstrate that
  for all $t\in[0,T]$, $x\in\R^d$, and
  all stopping times $\tau\colon\Omega\to [t,T]$
  it holds that
  \begin{equation}
  \label{eq:exp_V_Z_ass}
  \begin{split}  
    &\Big(\E\big[\abs{V(\tau, X^x_{t,\tau})}^2
    \big]\Big)^{\frac{1}{2}}
    \leq \exp(KT)
    \Big(\E\big[\abs{V(t, x)}^2
    \big]\Big)^{\frac{1}{2}}
    +\frac{b}{K}(\exp(KT)-1)\\
    &= \exp(KT) \, V(t,x)
    +\frac{b}{K}(\exp(KT)-1).
  \end{split}
  \end{equation}
  Moreover, note that item~\ref{it:Z_L2_bd}
  of Lemma~\ref{lem:Z_L_2_bd}
  and the fact that for all $a,b\in\R$
  it holds that $\sqrt{a+b}
  \leq \sqrt{a}+\sqrt{b}$
  show that for all
  $t\in[0,T]$, $s\in(t,T]$, $x\in\R^d$
  it holds that  
  \begin{equation}\label{eq:Z_L2_bound1}
  \begin{split}
    &\Big(\E\big[\lvert\lvert\lvert Z^x_{t,s}\rvert\rvert\rvert^2\big]\Big)^{\frac{1}{2}}
    \leq 
    \bigg( 1+\frac{1}{\alpha(s-t)^2}
  \int_t^s d \exp(2(r-t)c) \,\d r
  \bigg)^{\frac{1}{2}}\\
  &\leq 
  \bigg(1+\frac{d}{\alpha(s-t)}
 \exp(2cT)\bigg)^{\frac{1}{2}}
  \leq 1+\bigg(\frac{d}{\alpha(s-t)} 
    \exp(2cT)\bigg)^{\frac{1}{2}}.
  \end{split}
  \end{equation}
  The Cauchy-Schwarz inequality
  and \eqref{eq:exp_V_Z_ass}
  therefore show that
  for all $t\in[0,T]$, $s\in (t,T]$, 
  $x\in O$
  it holds that
  \begin{equation}\label{eq:Z_L2_bound2}
  \begin{split}
    &\E\big[V(s, X^x_{t,s})\lvert\lvert\lvert
    Z^x_{t,s}\lvert\lvert\lvert\big]
    \leq \Big(\E\big[\abs{V(s,X^x_{t,s})}^2\big]\Big)^{\frac{1}{2}}
    \Big(\E\big[\lvert\lvert\lvert Z^x_{t,s}\lvert\lvert\lvert^2\big]\Big)^{\frac{1}{2}}\\
    &\leq \Big( \exp(KT) \, V(t,x)
    +\frac{b}{K}(\exp(KT)-1) \Big)
     \bigg[1+\bigg(\frac{d}{\alpha(s-t)} 
    \exp(2cT)\bigg)^{\frac{1}{2}}\bigg]\\
    &\leq \frac{1}{\sqrt{s-t}}\Big( 
    \exp(KT) \, V(t,x)
    +\frac{b}{K}(\exp(KT)-1) \Big)
     \bigg[\sqrt{T}+\bigg(\frac{d}{\alpha} 
    \exp(2cT)\bigg)^{\frac{1}{2}}\bigg]\\
    &\leq \frac{V(t,x)}{\sqrt{s-t}}\Big( 
    \exp(KT)
    +\frac{b}{K \,[\inf_{u\in [0,T]}
    \inf_{y\in O}V(u,y)]}(\exp(KT)-1) \Big)\\
    &\qquad \cdot
     \bigg[\sqrt{T}+\bigg(\frac{d}{\alpha} 
    \exp(2cT)\bigg)^{\frac{1}{2}}\bigg].
  \end{split}
  \end{equation}
  Furthermore, observe that
  for all $k\in\N$, $t\in[0,T]$, $x\in O$,
  $\tau\colon\Omega\to[t,T]$
  stopping time 
  it holds that
  \begin{equation}
  \begin{split}
    &\E[V(\tau, X^x_{t,\tau})]
    =\int_\Omega V(\tau(\omega), X^x_{t,\tau(\omega)}(\omega))
    \,\mathbb{P}(\d \omega)\\
    &\geq \int_{\{\tilde{\omega}\in\Omega\colon \norm{X^x_{t,\tau(\tilde{\omega})}(\tilde{\omega})}\geq k\}} 
    V(\tau(\omega), X^x_{t,\tau(\omega)}(\omega))
    \,\mathbb{P}(\d \omega)\\
    &\geq \Big[\inf_{s\in[t,T]}\inf_{y\in O\setminus O_k}
    V(s,y)\Big] \,
    \mathbb{P}(\norm{X^x_{t,\tau}}\geq k)
    \geq \Big[\inf_{s\in[0,T]}\inf_{y\in O\setminus O_k}
    V(s,y)\Big] \,
    \mathbb{P}(\norm{X^x_{t,\tau}}\geq k).
  \end{split}
  \end{equation}
  Combining this with 
  \eqref{eq:exp_V_Z_ass}
  implies that for all $k\in\N$,
  $t\in [0,T]$, $x\in O$, 
  $\tau\colon\Omega\to[t,T]$
  stopping time  
  it holds that
  \begin{equation}
  \begin{split}
    &\mathbb{P}(\norm{X^x_{t,\tau}}\geq k)
    \leq \frac{\E[V(\tau, X^x_{t,\tau})]}{\inf_{s\in[0,T]}\inf_{y\in O\setminus O_k}
    V(s,y)}
    \leq \frac{(\E[\abs{V(\tau, X^x_{t,\tau})}^2])^{\frac{1}{2}}}{\inf_{s\in[0,T]}\inf_{y\in O\setminus O_k}
     V(s,y)}\\
    &\leq \frac{\exp(KT) \, V(t,x)
    +\frac{b}{K}(\exp(KT)-1)}{\inf_{s\in[0,T]}\inf_{y\in O\setminus O_k}
     V(s,y)}.
  \end{split}
  \end{equation}
  The fact that
  $\sup_{r\in(0,\infty)}[\inf_{t\in[0,T]\inf_{x\in O\setminus O_r}}
  V(t,x)]=\infty$
  therefore proves that
  for all non-empty, compact
  $\mathcal{K}\subseteq [0,T]\times O$
  it holds that 
  \begin{equation}
  \label{eq:X_cont_ass3}
  \begin{split}
    &\inf_{k\in\N}\Big[\sup_{(t,x)\in \mathcal{K}}\Big(
  \sup_{\tau\colon\Omega\to[0,T] \text{ stopping time}}
    \mathbb{P}\big(\norm{X^{x}_{t,\tau}}\geq k\big)\Big)\Big]\\
    &\leq \inf_{k\in\N}\bigg[
  \frac{\exp(KT)[\sup_{(t,x)\in \mathcal{K}}V(t,x)]
  +\frac{b}{K}(\exp(KT)-1)}{\inf_{s\in[0,T]}\inf_{y\in O\setminus O_k}
    V(s,y)}\bigg]\\
    &\leq
  \frac{\exp(KT)[\sup_{(t,x)\in \mathcal{K}}V(t,x)]
  +\frac{b}{K}(\exp(KT)-1)}{\sup_{k\in\N}[\inf_{s\in[0,T]}\inf_{y\in O\setminus O_k}
    V(s,y)]}
    =0.
  \end{split}
  \end{equation}  
  Lemma~\ref{lem:X_cont_in_prob}
  and the assumption that
  $\mu\in C^{0,1}([0,T]\times O,\R^d)$
  and $\sigma\in C^{0,1}([0,T]\times O,\R^{d\times d})$
  hence demonstrate that for all
  $\varepsilon\in (0,\infty)$,
  $s\in [0,T]$, 
  $(t_n,x_n)\in [0,T]\times O$, 
  $n\in\N$, with
  $\limsup_{n\to\infty}[\abs{t_n-t_0}
  +\norm{x_n-x_0}]=0$
  it holds that
  \begin{equation}
  \label{eq:X_cont2}
    \limsup_{n\to\infty} \big[
    \mathbb{P}(
    \norm{X^{x_n}_{t_n,\max\{s,t_n\}}
    -X^{x_0}_{t_0,\max\{s,t_0\}}})
    >\varepsilon\big]
    =0.
  \end{equation}
  In addition, observe that
  Lemma~\ref{lem:Z_continuity}
  ensures that 
  for all  
  $t \in [0,T)$, $s\in(t,T]$, 
  $x\in O$ it holds that
  \begin{equation}
  \label{eq:Z_cont2}
    \limsup\nolimits_{[0,s)\times O \ni (u,y)\to (t,x)} 
    \Big[ 
     \E \big[ \lvert\lvert\lvert Z^{y}_{u, s} 
    - Z^{x}_{t, s}\rvert\rvert\rvert \big] \Big] 
    = 0.
  \end{equation}
  Next note that
  \eqref{eq:Z_L2_bound1}
  and the fact that for all
  $t\in[0,T]$ it holds that
  $\int_t^T \frac{1}{\sqrt{s-t}}\,\d s
  =2\sqrt{T-t}$
  demonstrate that
  \begin{equation}
  \begin{split}
    &\sup_{t \in [0,T)\setminus K_r} \sup_{x \in O \setminus O_r} \bigg[
    \E\Big[\int_t^T \lvert\lvert\lvert Z^x_{t,s}\rvert\rvert\rvert
    (\d s + \delta_T(\d s))\Big]
    \sqrt{T-t}\bigg] \\
    &\leq 
    \sup_{t\in [0,T)\setminus K_r}
    \sup_{x\in O\setminus O_r}
    \bigg[\int_t^T
   \Big(\E\big[\lvert\lvert\lvert Z^x_{t,s}\rvert\rvert\rvert^2 \big]\Big)^{\frac{1}{2}} (\d s +\delta_T(\d s))\sqrt{T-t}\bigg]\\
  &\leq 
  \sup_{t\in [0,T)\setminus K_r}
  \bigg[\int_t^T
  \bigg( 1+ \sqrt{\frac{d}{\alpha(s-t)}} 
  \exp(cT)\bigg)(\d s +\delta_T(\d s))\sqrt{T-t}\bigg]\\
  &= 
  \sup_{t\in [0,T)\setminus K_r}
  \bigg[\bigg( T-t +1
  + \sqrt{\frac{d}{\alpha}}\exp(cT)
  \Big(\int_t^T \frac{1}{\sqrt{s-t}}\d s 
  +\frac{1}{\sqrt{T-t}}\Big)\bigg)
  \sqrt{T-t}\bigg]\\
  &= \sup_{t\in [0,T)\setminus K_r}
  \bigg(\sqrt{T-t}(T-t+1)
  +\sqrt{\frac{d}{\alpha}}\exp(cT)
  (2(T-t)+1)\bigg)\\
  &\leq \sqrt{T}(T+1)
  +\sqrt{\frac{d}{\alpha}}\exp(cT)
  (2T+1).
  \end{split}
  \end{equation}
  The assumption that
  $\liminf_{r\to\infty}
  [\inf_{t\in [0,T]}
  \inf_{x\in O\setminus O_r}
  V(t,x)]=\infty$
  hence ensures that
  \begin{equation}
  \begin{split}
    &\limsup_{r \to\infty} \Bigg[ \sup_{t \in [0,T)\setminus K_r} \sup_{x \in O \setminus O_r} \Bigg(
    \frac{\E\Big[\int_t^T \lvert\lvert\lvert Z^x_{t,s}\rvert\rvert\rvert
    (\d s + \delta_T(\d s))\Big]
    \sqrt{T-t}}{V(t,x)}\Bigg) \Bigg]\\
    &\leq \limsup_{r \to\infty} \Bigg[
    \frac{\sup_{t \in [0,T)\setminus K_r} \sup_{x \in O \setminus O_r}\Big(
    \E\Big[\int_t^T \lvert\lvert\lvert Z^x_{t,s}\rvert\rvert\rvert
    (\d s + \delta_T(\d s))\Big]
    \sqrt{T-t}\Big)}{\inf_{t\in [0,T)\setminus K_r}
  \inf_{x\in O\setminus O_r}
  V(t,x)}\Bigg]\\
  &\leq \limsup_{r \to\infty} \Bigg[
  \frac{ \sqrt{T}(T+1)
  +\sqrt{\frac{d}{\alpha}}\exp(cT)
  (2T+1)}{\inf_{t\in [0,T)\setminus K_r}
  \inf_{x\in O\setminus O_r}
  V(t,x)}\Bigg]\\
  &\leq \frac{ \sqrt{T}(T+1)
  +\sqrt{\frac{d}{\alpha}}\exp(cT)
  (2T+1)}{\liminf_{r \to\infty}[\inf_{t\in [0,T]}
  \inf_{x\in O\setminus O_r}
  V(t,x)]}
  =0.
  \end{split}
  \end{equation}  
  Combining this,
  \eqref{eq:Z_L2_bound2},
  \eqref{eq:X_cont2}, and
  \eqref{eq:Z_cont2}
  with
  Theorem~\ref{thm:ex_cont_solution}
  establishes item~\ref{it:cor_ex_cont_sol1},
  \ref{it:cor_ex_cont_sol2},
  and \ref{it:cor_ex_cont_sol3}.
  The proof of Theorem~\ref{thm:ex_cont_sol_Z}
  is thus complete.
\end{proof}

\subsection*{Acknowledgements}
This work has been funded by the
Deutsche Forschungsgemeinschaft
(DFG, German Research Foundation)
through the research grant
HU1889/6-2.

\printbibliography


\end{document}